\documentclass[11pt, reqno]{amsart}
\usepackage{amsmath,amsfonts,amssymb,amsmath,latexsym,mathrsfs}
\usepackage[all]{xy}
\usepackage{setspace}

\def\C{\mathbb C}
\def\Q{\mathbb Q}
\def\Z{\mathbb Z}
\def\L{\mathfrak{L}}
\def\N{\mathbb N}
\def\P{\mathbb P}

\def\1{\mathbbm 1}

\def\g{\mathfrak{g}}
\def\h{\mathfrak{h}}

\newcommand{\interior}[1]{\raise0.2ex\hbox{$\displaystyle{\mathop{#1}^{\circ}}$}}

\renewcommand\phi{\varphi}

\newtheorem{theorem}{Theorem}[section]

\newtheorem{proposition}[theorem]{Proposition}

\newtheorem{remark}[theorem]{Remark}

\newtheorem{lemma}[theorem]{Lemma}
\numberwithin{equation}{section}

\newcommand{\quot}{\mathit{quot}}

\title{On a lower central series filtration of the Grothendieck-Teichm\"{u}ller Lie algebra $\mathfrak{grt}_1$}
\author{N. Arbesfeld}
\address{N.A.: Department of Mathematics, Columbia University, New York, NY 10027, USA \newline 
E-mail address: {\tt  nma@math.columbia.edu}}
\author{B. Enriquez}
\address{B.E..: IRMA (CNRS) et D\'epartement de math\'ematiques, Universit\'e de Strasbourg, 7 rue Ren\'e Descartes, 67000 
Strasbourg,  France \newline 
E-mail address: {\tt  b.enriquez@math.unistra.fr}}

\dedicatory{Dedicated to B.L. Feigin on his 60th birthday}

\begin{document}

\begin{abstract} The Grothendieck-Teichm\"{u}ller Lie algebra is a Lie subalgebra of a Lie algebra of derivations of the free Lie algebra in two 
generators. We show that the lower central series of the latter Lie algebra induces a decreasing filtration of the Grothendieck-Teichm\"{u}ller Lie 
algebra and we study the corresponding graded Lie algebra. Its degree zero part had been previously computed by the second author. 
We show that the degree one part is a module over a symmetric algebra, which are both equipped with compatible decreasing filtrations, 
and we exhibit an explicit lower bound for the associated graded module. We derive from there some information on explicit expression of the 
depth 3 part of the depth-graded of the Grothendieck-Teichm\"{u}ller Lie algebra.
\end{abstract}

\maketitle


\section{Introduction}

In this article, we study the graded Lie algebra of the Grothendieck-Teichm\"{u}ller group, denoted $\mathfrak{grt}_1$. This Lie algebra was introduced in \cite{Dr} in connection with the theory of associators, and in \cite{Ih1, Ih2} under the name of ``stable derivation algebra'' as a Lie algebra of compatible collections of derivations of Lie algebras of infinitesimal braids in genus 0. One source of interest in $\mathfrak{grt}_1$ is its link with the theory of motives and multizeta values (\cite{A}).

It was shown in \cite{Br1} that $\mathfrak{grt}_1$ contains a free Lie algebra with one generator in each odd, $\geq 3$ degree. On the other 
hand, $\mathfrak{grt}_1$ is equipped with a decreasing filtration, the {\it depth filtration,} compatible with its grading (which will henceforth be 
called the {\it weight grading}) and is the subject of conjectures 
(\cite{BK,Br2}). For each integer $d\geq 1$, these conjectures predict the Hilbert series of the depth $d$ part of the depth-graded of the 
Grothendieck-Teichm\"{u}ller Lie algebra. They were established in [G] for $d=2,3$. 

In this paper, we introduce another decreasing filtration of $\mathfrak{grt}_1$, also compatible with the weight grading, which we call the 
{\it l.c.s.-filtration}. This filtration is constructed as follows. The Lie algebra $\mathfrak{grt}_1$ is a Lie subalgebra of a Lie algebra 
$\overline{\mathfrak{L}}$, which is equipped with a decreasing filtration $\overline{\mathfrak L}=\overline{\mathfrak L}^0\supset 
\overline{\mathfrak L}^1\supset\cdots$ arising from a lower central series. The associated graded Lie algebra 
${\mathfrak L}={\mathfrak L}_0\oplus{\mathfrak L}_1\oplus\cdots$ is then equipped with a graded bracket, such 
that ${\mathfrak L}_0$ is abelian (here ${\mathfrak L}_i=\overline{\mathfrak L}^i/\overline{\mathfrak L}^{i+1}$). The decreasing filtration 
of $\overline{\mathfrak L}$ induces a decreasing filtration on $\mathfrak{grt}_1$. The associated graded Lie algebra is 
$\mathrm{gr}_{lcs}(\mathfrak{grt}_1):=\oplus_i \mathrm{gr}^i_{lcs}(\mathfrak{grt}_1)$, where  $\mathrm{gr}^i_{lcs}(\mathfrak{grt}_1) :=(\mathfrak{grt}_1\cap\overline{\mathfrak L}^i)/(\mathfrak{grt}_1\cap\overline{\mathfrak L}^{i+1})$. The depth filtration 
of $\mathfrak{grt}_1$ also arises from a  decreasing filtration of $\overline{\mathfrak L}$, and it induces a decreasing filtration 
on each $\mathrm{gr}^i_{lcs}(\mathfrak{grt}_1)$.  

The space $\mathrm{gr}^0_{lcs}(\mathfrak{grt}_1)\hookrightarrow {\mathfrak L}_0$ has been explicitly computed in \cite{En1}. In what 
follows, we set 
$$
\Sigma:=\mathrm{gr}^0_{lcs}(\mathfrak{grt}_1). 
$$
The Lie bracket defines a $S(\Sigma)$-module structure on each $\mathrm{gr}^i_{lcs}(\mathfrak{grt}_1)$. Let us define the 
{\it $\Sigma$-degree} on $S(\Sigma)$ by the condition that $\Sigma$ has degree 1; this degree is compatible with the 
weight degree.  
Let us define a {\it $\Sigma$-structure} as the following data: 
\begin{itemize}
\item a vector space ${\mathbb M}$, bigraded for ($\Sigma$-degree, weight degree) and equipped with a decreasing 
depth filtration compatible with the bigrading; the decomposition for the $\Sigma$-degree starts in degree 2 and is denoted 
${\mathbb M}={\mathbb M}_0\oplus {\mathbb M}_1\oplus\cdots$, where ${\mathbb M}_i$ has degree $i+2$; 
\item a $S(\Sigma)$-module structure on ${\mathbb M}$ and a linear map $\Lambda^2(\Sigma)\to {\mathbb M}_0$, 
$\sigma\wedge\sigma'\mapsto\{\sigma,\sigma'\}$,  
\end{itemize}
such that $\sigma\cdot\{\sigma',\sigma''\}+\mathrm{\ cycl.\ perm.}=0$, and such that the action map $S(\Sigma)\otimes {\mathbb M}
\to {\mathbb M}$ and the linear map $\Lambda^2(\Sigma)\to {\mathbb M}_0$ are compatible with all degrees and filtrations, except for the 
l.c.s. filtration for which the map has degree 1. 

We associate to the space $\mathrm{gr}^1_{lcs}(\mathfrak{grt}_1)$ a $\Sigma$-structure 
$[\mathrm{gr}^1_{lcs}(\mathfrak{grt}_1)]$, whose underlying $S(\Sigma)$-module is the graded module associated to the 
$S(\Sigma)$-module $\mathrm{gr}^1_{lcs}(\mathfrak{grt}_1)$, so ${\mathbb M}_i:=S^i(\Sigma)\cdot\mathrm{gr}^1_{lcs}
(\mathfrak{grt}_1)/S^{i+1}(\Sigma)\cdot\mathrm{gr}^1_{lcs}(\mathfrak{grt}_1)$ for any $i\geq 0$.  

Our main result is the computation of an explicit lower bound ${\mathbb M}^{min}(\Sigma)$ for the $\Sigma$-structure 
$[\mathrm{gr}^1_{lcs}(\mathfrak{grt}_1)]$ (see Subsection \ref{subsect:summary} and Theorem 
\ref{thm:Mmin}). One derives from there a lower bound for the Hilbert series of $\mathrm{gr}^1_{lcs}(\mathfrak{grt}_1)$,
as well as for the double Hilbert series of $\mathrm{gr}_{dpth}\mathrm{gr}^1_{lcs}(\mathfrak{grt}_1)$ (Theorem \ref{thm:lower:bound}). 

Here is a table of the filtrations/degrees discussed above. The entry (object $X$, filtration ${\mathcal F}$) contains the mention ``graded'' 
if the filtration ${\mathcal F}$ on $X$ actually comes from a grading, together with the list of degrees $i$ for which 
$\mathrm{gr}_{{\mathcal F}}^i(X)$ is nontrivial; the filtrations are always compatible with the gradings, and when 
an object is equipped with several gradings, they are always compatible. 


\begin{tabular}{|l||l|l|l|l|}
  \hline
 & depth filtration & lcs filtration & weight degree& $\Sigma$-degree\\
  \hline
$\overline{\mathfrak L}$ & graded/1,2,3,...  & 0,1,2,3,...& 2,3,4,5,...&undefined\\
 \hline
 ${\mathfrak L}$ & graded/1,2,3,... & graded/0,1,2,3,... & 2,3,4,5,...&undefined\\
 \hline
 ${\mathfrak{grt}}_1$ & 1,2,3,... & 0,1,2,3,... & 3,5,7,8,9,10,...&undefined\\
   \hline
 $\Sigma=\mathrm{gr}_{lcs}^0({\mathfrak{grt}}_1) $ & 1 & graded/0 & 3,5,7,9,11,...& 1\\
 \hline
 $\mathrm{gr}_{lcs}^1({\mathfrak{grt}}_1)$ & 2,3,4,5,6,... & graded/1 & 8,10,11,12,13,...& undefined\\
\hline
${\mathbb M}_0^{min}(\Sigma)$ & 2,3,4,5,6,... & graded/1 & 8,10,12,14,16,...& 2\\
\hline
${\mathbb M}_1^{min}(\Sigma)$ & 2,3,4,5,6,... & graded/1 & 11,13,15,17,...& 3\\
\hline
${\mathbb M}_i^{min}(\Sigma)$ & 2,3,4,5,6,... & graded/1 & $3i+8,3i+10,\ldots$ & $i+2$\\
\hline
\end{tabular}

The plan of the paper is as follows. In Section \ref{Section2}, we describe the Lie algebra $\overline{\mathfrak L}$, its l.c.s. filtration, the 
associated graded Lie algebra $\mathfrak{L}$, and their depth filtrations. In Section \ref{Section3}, we describe a quotient $\mathfrak{L}_{quot}$ of 
$\overline{\mathfrak{L}}$, equipped with the structure of an extension of abelian Lie algebras. In Section \ref{Section4}, we discuss the structures 
associated with Lie subalgebras of extensions of abelian Lie algebras. We apply this discussion in Section \ref{Section5}, where we construct 
the $\Sigma$-structure $[\mathrm{gr}_{lcs}^1({\mathfrak{grt}}_1)]$ and derive a lower bound for it which we express in terms of certain 
commutative rings. In Section \ref{Section:comm:alg}, we gather some results on these commutative rings. In Section 
\ref{Section:comp:Mminlowest}, we compute the lowest degree part ${\mathbb{M}}_0^{min}(\Sigma)$ of the lower bound and in 
Section \ref{Section:comp:Mminfull}, we compute the rest of it. As a corollary, we derive lower bounds for generating series in 
Section \ref{Section:comp:hilbert}. In Section \ref{Section:depth3}, we express the depth 3 part of the depth-graded of the Lie algebra
${\mathfrak{grt}}_1$ in terms of ${\mathbb{M}}^{min}(\Sigma)$. 

\section{The Lie algebra $\overline{\mathfrak L}$ and its l.c.s. filtration} \label{Section2}

In this section, we describe the Lie algebra $\overline{\mathfrak L}$, its l.c.s. filtration, the associated graded Lie algebra 
${\mathfrak L}$ and the depth filtration on these Lie algebras. 

\subsection{The Lie algebra $\overline{\mathfrak L}$}\label{sec:structure}

\subsubsection{The vector space $\overline{\mathfrak{L}}$}
If $X$ is a vector space over $\C$, we denote by $\mathbb{L}(X)$ the free Lie algebra generated by $X$, and if $k\geq 1$, we denote by $\mathbb{L}_k(X)$ the component of $\mathbb{L}(X)$ of degree $k$ (where $X$ has degree $1$). We therefore have $\mathbb{L}_1(X)=X, \mathbb{L}_2(X)=\Lambda^2(X)$, and so on.

We define $V_0$ as the vector space over $\C$ freely generated by a pair of letters $x,y$. We set $\tilde{\mathfrak{L}}:=\mathbb{L}(V_0)$; $\tilde{\mathfrak{L}}$ is then the free Lie algebra generated by $x,y$. Its Lie bracket is denoted $[,]$. It is $\N$-graded by the convention that $x,y$ have degree $1$. We set $\overline{\mathfrak{L}}:=[\tilde{\mathfrak{L}},\tilde{\mathfrak{L}}];$ we have $\tilde{\mathfrak{L}}=V_0\oplus \overline{\mathfrak{L}}$. The Lie algebra $\overline{\mathfrak{L}}$ can be identified with the part 
of $\tilde{\mathfrak{L}}$ of degree $\geq 2$. 
\newline

\subsubsection{The Lie bracket on $\overline{\mathfrak{L}}$}
For $f,g$ in $\tilde{\mathfrak{L}}$ we set $$\langle f,g\rangle:=[f,g]+D_f(g)-D_g(f),$$ where $D_f$ is the derivation of $(\tilde{\mathfrak{L}},[,])$ such that $D_f(x)=0, D_f(y)=[y,f]$. Then $(\tilde{\mathfrak{L}},\langle,\rangle)$ is a Lie algebra, and the map $\tilde{\mathfrak{L}}\to\mathrm{Der}(\tilde{\mathfrak{L}},[,])$ sending $f\mapsto D_f$ is a Lie algebra homomorphism (where the left side is equipped with the bracket $\langle,\rangle$). One checks that $\overline{\mathfrak{L}}$ is a Lie 
subalgebra of $\tilde{\mathfrak{L}}$ for the bracket $\langle,\rangle$. Moreover, $V_0$ is contained in the center of the Lie algebra $\tilde{\mathfrak{L}}$, so the isomorphism $\tilde{\mathfrak{L}}\simeq V_0\oplus \overline{\mathfrak{L}}$ provides an isomorphism between $\tilde{\mathfrak{L}}$ and the direct sum of $\overline{\mathfrak{L}}$ and an abelian Lie algebra of dimension $2$.

\subsection{The l.c.s. filtration on $\overline{\mathfrak{L}}$}\label{subsec:filtration}

\subsubsection{Definition of a descending filtration on the vector space $\overline{\mathfrak{L}}$}
By the Lazard elimination theorem, the Lie algebra $(\overline{\mathfrak{L}},[,])$ is freely generated by the elements $(\mathrm{ad} x)^i(\mathrm{ad} y)^j([x,y])$ where $i,j\geq 0$. We then set $\overline{\mathfrak{L}}^0:=\overline{\mathfrak{L}}$, and $\overline{\mathfrak{L}}^{i+1}:=[\overline{\mathfrak{L}},\overline{\mathfrak{L}}^i]$ for $i\geq 0.$ 

So we have 
\begin{align}\label{eqa}
[\overline{\mathfrak{L}}^i,\overline{\mathfrak{L}}^j]\subset\overline{}\mathfrak{L}^{i+j+1}; 
\end{align} 
equivalently, $\overline{\mathfrak{L}}=\overline{\mathfrak{L}}^0\supset \overline{\mathfrak{L}}^1\supset\cdots$ is, up to shiting the degree by $1$, a descending filtration on the Lie algebra $(\overline{\mathfrak{L}},[,])$. We define the Lie algebra $(\mathfrak{L},[,])$ by $\mathfrak{L}:=\oplus_{i\geq 0}\mathfrak{L}_i$ where $\mathfrak{L}_i:=\overline{\mathfrak{L}}^i/\overline{\mathfrak{L}}^{i+1}$; the bracket 
is the sum of the maps $\mathfrak{L}_i\otimes\mathfrak{L}_j\to\mathfrak{L}_{i+j+1}$ induced by the bracket $[,]$ on $\overline{\mathfrak{L}}$, so $\mathfrak{L}[1]$ identifies with the associated 
graded of $(\overline{\mathfrak{L}}[1],[,])$. We therefore have Lie algebra isomorphisms $\overline{\mathfrak{L}}[1]\simeq\mathfrak{L}[1]
\simeq\mathbb{L}(AB\C[A,B])$
(where $[1]$ means shifting the filtration or degree by 1),  where the 
first isomorphism is a non-canonical isomorphism of filtered Lie algebras, 
and the second is an isomorphism of graded Lie algebras, canonically defined by 
$$\mathfrak{L}_0\ni (\text{class of}\ (\mathrm{ad} x)^i
(\mathrm{ad} y)^j([x,y]))\mapsto A^{i+1}B^{j+1}\in {\mathbb L}_1(AB\C[A,B]).$$

\subsubsection{Lie algebraic properties of the filtration on $\overline{\mathfrak{L}}$}
\begin{proposition}\label{Ptes:L}
\bigskip
(1) If $i,j\geq 0$, we have $\langle\overline{\mathfrak{L}}^i,\overline{\mathfrak{L}}^j\rangle\subset\overline{\mathfrak{L}}^{i+j}.$
(2) $\langle\overline{\mathfrak{L}}^0,\overline{\mathfrak{L}}^0\rangle\subset\overline{\mathfrak{L}}^1$.
\end{proposition}
\begin{proof}
First we prove $(1)$. We have already seen that $[\overline{\mathfrak{L}}^i,\overline{\mathfrak{L}}^j]\subset \overline{\mathfrak{L}}^{i+j+1}$. 
If $f$ is in $\overline{\mathfrak{L}}^i$, then $D_f$ sends $x$ to $0$ and $y$ to $[y,f]$ which is an element of $\overline{\mathfrak{L}}^i$. 
We therefore have $D_f(\tilde{{\mathfrak{L}}})\subset\overline{\mathfrak{L}}^i$, thus $D_f(\overline{\mathfrak{L}})\subset\overline{\mathfrak{L}}^i$. We therefore have 
$D_{\overline{\mathfrak{L}}^i}(\overline{\mathfrak{L}})\subset \overline{\mathfrak{L}}^i$.

Thus, 
$D_{\overline{\mathfrak{L}}^i}(\overline{\mathfrak{L}}^j)=D_{\overline{\mathfrak{L}}^i}(\underbrace{[\overline{\mathfrak{L}},[\ldots[\overline{\mathfrak{L}},\overline{\mathfrak{L}}]]]}_{j+1\ \text{terms}})
\subset \underbrace{[\overline{\mathfrak{L}},[\ldots[\overline{\mathfrak{L}},D_{\overline{\mathfrak{L}}^i}(\overline{\mathfrak{L}})]]]}_{j+1\ \text{terms}}\subset 
\underbrace{[\overline{\mathfrak{L}},[\ldots[\overline{\mathfrak{L}},\overline{\mathfrak{L}}^i]]]}_{j+1\ \text{terms}}\subset\overline{\mathfrak{L}}^{i+j}$, 
therefore \begin{align}\label{eqb} D_{\overline{\mathfrak{L}}^i}(\overline{\mathfrak{L}}^j)\subset\overline{\mathfrak{L}}^{i+j}.\end{align}

Equations (\ref{eqa}) and (\ref{eqb}) then imply $\langle\overline{\mathfrak{L}}^i,\overline{\mathfrak{L}}^j\rangle \subset\overline{\mathfrak{L}}^{i+j}$. 

We now prove $(2)$. By (\ref{eqb}), the map $\overline{\mathfrak{L}}^{\otimes 2}\to\overline{\mathfrak{L}}$ given by $f\otimes g\mapsto D_f(g)$ factors into a 
map $(\overline{\mathfrak{L}}^0/\overline{\mathfrak{L}}^1)^{\otimes 2}\to\overline{\mathfrak{L}}^0/\overline{\mathfrak{L}}^1$ which we will denote by 
$\alpha\otimes \beta\to \Psi(\alpha\otimes \beta)$.

We calculate the map $(AB\C[A,B])^{\otimes 2}\to AB\C[A,B]$ induced by $\Psi$ and the isomorphism 
$\overline{\mathfrak{L}}^0/\overline{\mathfrak{L}}^1\simeq AB\C[A,B].$ We have 
\begin{align*} D&_{(\mathrm{ad} x)^i(\mathrm{ad} y)^j([x,y])}((\mathrm{ad} x)^{i'}(\mathrm{ad} y)^{j'}([x,y]))\\&=\sum_{\alpha=0}^{j'-1}(\mathrm{ad}x)^{i'}(\mathrm{ad}y)^{\alpha}([(\mathrm{ad} y)(\mathrm{ad} x)^i(\mathrm{ad} y)^j([x,y]), (\mathrm{ad} y)^{j'-1-\alpha}([x,y])])\\&+(\mathrm{ad} x)^{i'}(\mathrm{ad} y)^{j'}(\mathrm{ad} x)(\mathrm{ad} y)(\mathrm{ad} x)(\mathrm{ad} y)(\mathrm{ad} x)^i(\mathrm{ad} y)^j([x,y]).\end{align*}
Each term of the first sum belongs to $[\overline{\mathfrak{L}},\overline{\mathfrak{L}}]=\overline{\mathfrak{L}}^1$. The image of the last term under the isomorphism 
$\overline{\mathfrak{L}}/\overline{\mathfrak{L}}^1\simeq AB\C[A,B]$ is $A^{i+i'+2}B^{j+j'+2}=A^{i+1}B^{j+1}\cdot A^{i'+1}B^{j'+1}.$ The map 
$(AB\C[A,B])^{\otimes 2}\to AB\C[A,B]$ obtained from $\Psi$ by the isomorphism $\overline{\mathfrak{L}}^0/\overline{\mathfrak{L}}^1\simeq AB\C[A,B]$ 
is therefore the product of polynomials. It follows that the map $\Psi$ satisfies $\Psi(\beta\otimes \alpha)=\Psi(\alpha\otimes\beta)$.

We thus deduce that for $f,g$ in $\overline{\mathfrak{L}}$, one has $D_f(g)-D_g(f)\in\overline{\mathfrak{L}}^1$, which implies that 
$\langle\overline{\mathfrak{L}},\overline{\mathfrak{L}}\rangle\subset \overline{\mathfrak{L}}^1$.
\end{proof}

\subsection{Calculation of the graded Lie algebra $(\mathrm{gr}_{lcs}(\overline{\mathfrak{L}}),\langle,\rangle)$}

By Proposition \ref{Ptes:L} we deduce that the bracket $\langle,\rangle$ of $\overline{\mathfrak{L}}$ induces a graded Lie bracket 
on ${\mathfrak{L}}=\oplus_{i\geq 0}{\mathfrak{L}}_i$, where ${\mathfrak{L}}=\mathrm{gr}_{lcs}(\overline{\mathfrak{L}})$, ${\mathfrak{L}}_i=\mathrm{gr}_{lcs}^i(\overline{\mathfrak{L}})$; its
components will also be denoted $\langle,\rangle:{\mathfrak{L}}_i\otimes{\mathfrak{L}}_j\to{\mathfrak{L}}_{i+j}$. 

\begin{proposition}\label{prop:gr:L}
\begin{enumerate}

\item  Let $\mathcal{V}=AB\C[A,B]$. For $P=p_1\otimes \cdots\otimes p_i\in \mathcal{V}^{\otimes i}$ and $q\in \mathcal{V}$, set 
$P\circledast q:=(q^{(1)}\cdot p_1)\otimes \cdots\otimes (q^{(i)}\cdot p_i)$, where for $q(A,B)\in \C[A,B]$, we have $q^{(1)}\otimes\cdots
\otimes q^{(i)}=q(A_1+\ldots+A_i,B_1+\ldots+B_i)\in \C[A_1,B_1,\ldots,A_i,B_i]\simeq \C[A,B]^{\otimes i}$ and the operation $\cdot$ 
corresponds to the structure of $\mathcal{V}$ as a $\C[A,B]$-module. Then, for $Q=q_1\otimes\cdots\otimes q_j\in 
\mathcal{V}^{\otimes j},$ we set 
$$
P\circledast Q=\sum_{\alpha=1}^j q_1\otimes\cdots\otimes q_{\alpha-1}\otimes (P\circledast q_{\alpha})\otimes q_{\alpha+1}\otimes\cdots\otimes q_j\in\mathcal{V}^{\otimes i+j-1}
$$ 
and we set $\langle P,Q\rangle=P\circledast Q-Q\circledast P$. This bracket extends linearly to a map $T(\mathcal{V})^{\otimes 2}\to T(\mathcal{V})$ which equips $T(\mathcal{V})$ with the structure of a Lie algebra. The subspace $\mathbb{L}(\mathcal{V})\subset 
T(\mathcal{V})$ is a Lie subalgebra of $T(\mathcal{V})$.

\item The Lie algebras $({\mathfrak{L}},\langle,\rangle)$ and $(\mathbb{L}(\mathcal{V}),\langle,\rangle)$ are isomorphic.
\end{enumerate}
\end{proposition}

\begin{proof}
(1) The fact that $T(\mathcal{V})$ is a Lie algebra (equipped with the structure described above) is shown in \cite[Prop. 4.1]{En2}.

We now show that $\mathbb{L}(\mathcal{V})$ is a Lie subalgebra of $T(\mathcal{V})$. Let $P\in \mathbb{L}_i(\mathcal{V})$ and 
$q\in \mathcal{V}$. Then $P$ is a linear combination of expressions of the form $[p_1,[\ldots,[p_{i-1},p_i]]]$ where the 
$p_{\alpha}$ belong to $\mathcal{V}$. Hence $P\circledast q$ is a linear combination of expressions of the form $[q^{(1)}\cdot p_1,
[\ldots,[q^{(i-1)}\cdot p_{i-1},q^{(i)}\cdot p_i]]]$ and so $P\circledast q\in\mathbb{L}_i(\mathcal{V})$. 

Now suppose $P\in\mathbb{L}_i(\mathcal{V})$ and $Q\in\mathbb{L}_j(\mathcal{V})$; so, $Q$ is a linear combination of expressions of the 
form $[q_1,[\ldots[q_{j-1},q_j]]]$, where the $q_{\beta}$ belong to $\mathcal{V}$. Then $P\circledast Q$ is a linear combination of 
expressions of the form $P\circledast [q_1,[\ldots[q_{j-1},q_j]]]$, each of which expresses as 
$\sum_{\alpha=1}^j[q_1,[\ldots[P\circledast q_{\alpha},[\ldots[q_{j-1},q_j]]]$, from which we deduce that 
$P\circledast Q\in \mathbb{L}_{i+j-1}(\mathcal{V})$. Therefore $\langle P,Q\rangle\in
\mathbb{L}_{i+j-1}(\mathcal{V}),$ which shows that $\mathbb{L}(\mathcal{V})$ is a Lie subalgebra of $T(\mathcal{V})$.

(2) It follows from the proof of Proposition \ref{Ptes:L} that 
$$
[\overline{\mathfrak{L}}^i,\overline{\mathfrak{L}}^j]\subset \overline{\mathfrak{L}}^{i+j+1},\ D_{\overline{\mathfrak{L}}^i}
(\overline{\mathfrak{L}}^j)\subset \overline{\mathfrak{L}}^{i+j}.
$$ 
From this, we deduce that the bracket $\langle,\rangle$ on ${\mathfrak{L}}$ is given by 
$$
\langle f,g\rangle=f\star g-g\star f,
$$ 
where $f\otimes g\to f\star g$ is the map ${\mathfrak{L}}_i\otimes{\mathfrak{L}}_j\to{\mathfrak{L}}_{i+j}$ induced by $\overline{\mathfrak{L}}^i\otimes\overline{\mathfrak{L}}^j\to
\overline{\mathfrak{L}}^{i+j}, f\otimes g\to D_f(g)$. It remains to show that the following diagram commutes: 
\begin{align}\label{grand:diag} \ \xymatrix{
{\mathfrak{L}}\otimes{\mathfrak{L}}
\ar[r] \ar[d]_{\star}& T(\mathcal{V})\otimes T(\mathcal{V}) \ar[d]_{\circledast} \\ {\mathfrak{L}}\ar[r] & 
T(\mathcal{V}) }\end{align} 
where the map ${\mathfrak{L}}\to T(\mathcal{V})$ is the direct sum over $i\geq 0$ of the injections ${\mathfrak{L}}_i\simeq \mathbb{L}_i(\mathcal{V})\subset \mathcal{V}^{\otimes i}$.

For all $f$ in ${\mathfrak{L}}_i$, the map $g\mapsto f\star g$ is a derivation with respect to the Lie algebra structure of $({\mathfrak{L}},[,])$. Moreover, for each $P$ 
in $T(\mathcal{V}),$ the map $Q\mapsto P\circledast Q$ is a derivation with respect to the associative algebra structure of $T(\mathcal{V})$. 
As the map ${\mathfrak{L}}\to T(\mathcal{V})$ is a Lie algebra homomorphism, it then suffices to show that the following diagram commutes  
\begin{align}\label{petit:diag} 
\xymatrix{{\mathfrak{L}}\otimes {\mathfrak{L}}^1\ar[r] 
\ar[d]_{\star}& T(\mathcal{V})\otimes \mathcal{V} \ar[d]_{\circledast} \\ {\mathfrak{L}}\ar[r] & T(\mathcal{V}) 
}\end{align}

If $f\in{\mathfrak{L}}_i=\overline{\mathfrak{L}}^i/\overline{\mathfrak{L}}^{i+1}$ is represented by $[f_1,\ldots,[f_i,f_{i+1}]]$, 
where each $f_{\alpha}$ is in $\overline{\mathfrak{L}}$, and if $g\in{\mathfrak{L}}_1\simeq \mathcal{V}$ is the class of 
$(\mathrm{ad} x)^k(\mathrm{ad} y)^l([x,y])$ (which corresponds to $\overline{g}=A^{k+1}B^{l+1}\in \mathcal{V}$) then we compute 
\begin{align*}
D_f(g)=&\sum_{\alpha=0}^{l-1}(\mathrm{ad} x)^k(\mathrm{ad} y)^{\alpha}(
\big[ [y,[f_1,\ldots,[f_i,f_{i+1}]]],(\mathrm{ad} y)^{l-1-\alpha}([x,y])\big])
\\&+(\mathrm{ad} x)^k(\mathrm{ad} y)^l(\mathrm{ad} x)(\mathrm{ad} y)([f_1,\ldots,[f_{i-1},f_i]]).
\end{align*}

Each term of the first sum belongs to $\overline{\mathfrak{L}}^{i+1},$ and thus is sent to $0$ in ${\mathfrak{L}}_i=\overline{\mathfrak{L}}^i/\overline{\mathfrak{L}}^{i+1}$. 
The remaining term belongs to the class $[\overline{g}^{(1)}\cdot f_1,\ldots,[\overline{g}^{(i-1)}\cdot f_{i-1},\overline{g}^{(i)}\cdot f_{i}]]$ in 
${\mathfrak{L}}_i=\overline{\mathfrak{L}}^i/\overline{\mathfrak{L}}^{i+1}$. Therefore, 
\begin{equation}
f\star g=[\overline{g}^{(1)}\cdot f_1,\ldots,[\overline{g}^{(i-1)}\cdot f_{i-1},\overline{g}^{(i)}\cdot f_i]]
\label{expression:fstarg}
\end{equation}
The image $P$ of $f$ in $T(\mathcal{V})$ is $f_1\otimes\cdots\otimes f_i-f_1\otimes\cdots\otimes f_i\otimes f_{i-1}+\ldots
+(-1)^{i-1}f_i\otimes\cdots\otimes f_{1}$. So, 
$$
P\circledast \overline{g}=(\overline{g}^{(1)}\cdot f_1)\otimes\cdots\otimes (\overline{g}^{(i)}\cdot f_i)-
(\overline{g}^{(1)}\cdot f_1)\otimes\cdots\otimes (\overline{g}^{(i-1)}\cdot f_i)\otimes (\overline{g}^{(i)}\cdot f_{i-1})
+\ldots+(-1)^{i-1}(\overline{g}^{(1)}\cdot f_i)\otimes\cdots\otimes(\overline{g}^{(i)}\cdot f_1),
$$ 
which is the image of the right hand side of (\ref{expression:fstarg}) under the map ${\mathfrak{L}}_i\to \mathcal{V}^{\otimes i}$. 
This shows that diagram (\ref{petit:diag}) commutes, and therefore so does diagram (\ref{grand:diag}).
\end{proof} 

\subsection{Depth filtration on $\overline{\mathfrak{L}}$ and ${\mathfrak{L}}$} \label{subsect:depth}
The Lie brackets $\langle,\rangle$ and $[,]$ on ${\mathfrak{L}}$ and $\overline{\mathfrak{L}}$ are bigraded
for the degrees in $x$ and $y$. For $i\geq 0$, set $F^i_{dpth}(\overline{\mathfrak{L}}):=\oplus_{j|j\geq i}\{$part of $\overline{\mathfrak{L}}$ of $y$-degree $j\}$. 
Then there is a decreasing filtration $\overline{\mathfrak{L}}=F^1_{dpth}(\overline{\mathfrak{L}})\supset F^2_{dpth}(\overline{\mathfrak{L}})\supset\cdots$, which is 
compatible with the brackets $\langle,\rangle$, $[,]$ and with the weight degree. 

For $i\geq 0$, set $F^i_{dpth}({\mathfrak{L}}):=\mathrm{im}(F^i_{dpth}(\overline{\mathfrak{L}})\to{\mathfrak{L}})$ and $F^i_{dpth}({\mathfrak{L}}^j):=
\mathrm{im}(F^i_{dpth}(\overline{\mathfrak{L}})\cap\overline{\mathfrak{L}}^j\to{\mathfrak{L}}^j)$. Then ${\mathfrak{L}}=F^1_{dpth}({\mathfrak{L}})\supset F^2_{dpth}({\mathfrak{L}})\supset\cdots$
is a decreasing filtration of ${\mathfrak{L}}$, compatible with the Lie bracket $\langle,\rangle$, with the weight degree and with the l.c.s.-degree. 

\section{The Lie algebra ${\mathfrak{L}}_{quot}$}\label{Section3}

The filtered Lie algebra $\overline{\mathfrak{L}}=\overline{\mathfrak{L}}^0\supset\overline{\mathfrak{L}}^1\supset\cdots$ gives rise to a Lie algebra 
${\mathfrak{L}}_{quot}:=\overline{\mathfrak{L}}/\overline{\mathfrak{L}}^2$, which naturally fits in an exact sequence of abelian Lie algebras. After reviewing 
the invariants attached to this situation (Subsection \ref{subsection:invariants}), we introduce ${\mathfrak{L}}_{quot}$ (Subsection 
\ref{subsec:Lquot}) and determine its invariants (Subsections \ref{subsection:module} and \ref{subsection:cocycle}). 

\subsection{Extensions of abelian Lie algebras} \label{subsection:invariants}
Let $\mathfrak{g}$ be a Lie algebra fitting in an exact sequence of 
abelian Lie algebras $0\to\mathfrak{g}_1\to{\mathfrak{g}}\to\mathfrak{g}_0\to 0$. 
The data of ${\mathfrak{g}}$ and of this exact sequence gives rise to: 
\begin{enumerate}
\item a module structure for $\mathfrak{g}_1$ over the abelian Lie algebra $\mathfrak{g}_0$, and therefore over the symmetric algebra $S(\mathfrak{g}_0)$. 

\item a cocycle $c\in H^2(\mathfrak{g}_0,\mathfrak{g}_1)$ (where $H^2$ denotes Lie algebra cohomology).
\end{enumerate}
Let us recall how this data is obtained: we fix a section $\sigma: \g_0\to \g$ of the linear map $\g\to\g_0$.  The action of $x_0\in\g_0$ on 
$x_1\in\g_1$ is given by $x_0\cdot x_1:=[\sigma(x_0),x_0]$. As $\g_0$ is abelian, this action equips $\g_1$ with the structure of an 
$S(\g_0)$-module; furthermore, this structure is independent of the choice of $\sigma$. On the other hand, the formula 
$\tilde c(x_0,x'_0):=[\sigma(x_0),\sigma(x'_0)]$ defines a $2$-cocycle for the Lie algebra with values in $\g_1$, whose cohomology class 
does not depend on the choice of $\sigma$.

\subsection{Definition and structure of ${\mathfrak{L}}_{\quot}$}\label{subsec:Lquot}

According to Proposition \ref{Ptes:L}, the subspace $\overline{\mathfrak{L}}^2$ is an ideal of the Lie algebra $(\overline{\mathfrak{L}},\langle,\rangle)$. This gives rise 
to a quotient Lie algebra ${\mathfrak{L}}_{\quot}:=\overline{\mathfrak{L}}/\overline{\mathfrak{L}}^2,$ whose bracket will again be denoted $\langle,\rangle$. 
Proposition \ref{Ptes:L} also implies that $\overline{\mathfrak{L}}^1$ is an ideal of $\overline{\mathfrak{L}}$. The quotient 
space ${\mathfrak{L}}_1=\overline{\mathfrak{L}}^1/\overline{\mathfrak{L}}^2$ is therefore an ideal of ${\mathfrak{L}}_{\quot},$ which is abelian as $\langle\overline{\mathfrak{L}}^1,\overline{\mathfrak{L}}^1
\rangle\subset\overline{\mathfrak{L}}^2$. The quotient Lie algebra is ${\mathfrak{L}}_0=\overline{\mathfrak{L}}/\overline{\mathfrak{L}}^1,$ which is also abelian as $\langle\overline{\mathfrak{L}},
\overline{\mathfrak{L}}\rangle\subset\overline{\mathfrak{L}}^1$. The Lie algebra ${\mathfrak{L}}_{quot}$ therefore fits in an extension $0\to{\mathfrak{L}}_1\to{\mathfrak{L}}_{\quot}\to{\mathfrak{L}}_0\to 0$ 
of abelian Lie algebras.

In what follows, we will set 
\begin{equation} \label{def:V:M}
V:={\mathfrak{L}}_0, \quad M:={\mathfrak{L}}_1
\end{equation}
so that ${\mathfrak{L}}_{quot}$ fits in an exact sequence of abelian Lie algebras 
\begin{align}\label{suite:ex:L:quot} 0\to M\to {\mathfrak{L}}_{\quot}\to V\to 0.\end{align}

\subsection{Module structure associated to ${\mathfrak{L}}_{\quot}$} \label{subsection:module}
Set 
\begin{align*}
& \C[A,B,A',B']^{as}:=\{\mathrm{polynomials\ in\ }\C[A,B,A',B'],\mathrm{\ antisymmetric\ under\ the\ exchange\ }(A,B)\leftrightarrow (A',B')\},  
\\ & 
\C[A,B,A',B']^{sym}:=\{\mathrm{polynomials\ in\ }\C[A,B,A',B'],\mathrm{\ symmetric\ under\ the\ exchange\ }(A,B)\leftrightarrow (A',B')\}. 
\end{align*}
There exist isomorphisms 
\begin{align}\label{isos}
V\simeq \mathbb{L}_1(AB\C[A,B])=AB\C[A,B],\quad M\simeq \mathbb{L}_2(AB\C[A,B])=ABA'B'\C[A,B,A',B']^{as}.
\end{align} 
There is a linear map $AB\C[A,B]\to \C[A,B,A',B']^{sym}$, given by $v\mapsto \delta_v(A,B,A',B')$, 
where
\begin{align}\label{application:delta} 
\delta_v(A,B,A',B'):=v(A+A',B+B')-v(A,B)-v(A',B'). 
\end{align}
The $V$-module structure of $M$ corresponds to the linear map $V\otimes M\to M$, defined by $v\otimes m\mapsto 
\langle \sigma(v),m\rangle$, where $\sigma:V\to{\mathfrak{L}}_{\quot}$ is any section of the projection ${\mathfrak{L}}_{\quot}\to V$. Proposition 
\ref{prop:gr:L} can then be rephrased as follows:
\begin{lemma}\label{lemme:actions}
Under the isomorphisms (\ref{isos}), this linear map becomes the linear map 
\begin{align*}
AB\C[A,B]\otimes ABA'B'\C[A,B,A',B']^{as}&\to ABA'B'\C[A,B,A',B']^{as}\\ v\otimes m&\mapsto -\delta_v\cdot m
\end{align*}
where the product $\cdot$ is multiplication of polynomials. 
\end{lemma}

\subsection{Cocycle associated to ${\mathfrak{L}}_{\quot}$} \label{subsection:cocycle}

Set $\overline M:=\C[A,B,A',B']^{as}$. Then $M$ is a vector subspace of $\overline M$. The formula defining the linear map from 
Lemma \ref{isos} extends to a linear map $AB\C[A,B]\otimes \C[A,B,A',B']^{as}\to\C[A,B,A',B']^{as}$, i.e., to a linear map 
$V\otimes\overline M\to\overline M$, which equips $\overline M$ with the structure of a {\it $V$-module, of which $M$ is a submodule.} 

We then define a linear map $V\to \overline M$ by $v\mapsto \lambda_v$, where 
$$
\lambda_v(A,B,A',B'):=\frac{1}{2}(-v(A+A',B)+v(A+A',B')+v(A,B+B')-v(A',B+B')-v(A,B')+v(A',B)).
$$
Multiplication of polynomials restricts to a map $\C[A,B,A',B']^{sym}\otimes \overline M\to \overline M.$ For $v,v'\in V$ we set 
\begin{equation} \label{def:c}
c(v,v'):=\delta_v\cdot \lambda_{v'}-\delta_{v'}\cdot \lambda_{v}\in\overline M.
\end{equation}

The space $M$ is the intersection of $\overline M$ with $ABA'B'\C[A,B,A',B']$, which can be identified with the intersection of the kernels of the maps from $\C[A,B,A',B']$ to the rings $\C[A,B,A'],$ $\C[A,B,B'],$ $\C[A,A',B'],$ $\C[B,A',B']$ given by the evaluations $P\mapsto P_{|B'=0},$ 
$P\mapsto P_{|A'=0},$ $P\mapsto P_{|B=0},$ $P\mapsto P_{|A=0},$ respectively.

Then, for $v\in AB\C[A,B]$, we have $(\lambda_v)_{|B'=0}=-{1\over 2}(\delta_v)_{|B'=0}$ so that 
$$
c(v,v')_{|B'=0}=(\delta_v)_{|B'=0}\cdot(-\frac{1}{2})(\delta_{v'})_{|B'=0}-(\delta_{v'})_{|B'=0}\cdot(-\frac{1}{2})(\delta_v)_{|B'=0}=0.
$$
Similarly, $(\lambda_v)_{|A'=0}=-{1\over 2}(\delta_v)_{|A'=0},(\lambda_v)_{|B=0}=-{1\over 2}(\delta_v)_{|B=0},$ and 
$(\lambda_v)_{|A=0}=-{1\over 2}(\delta_v)_{|A=0}$ so that $c(v,v')_{|A'=0}=c(v,v')_{|B'=0}=c(v,v')_{|A=0}=0$. So, the element 
$c(v,v')$ belongs to the subspace $ABA'B'\C[A,B,A',B']$ of $\C[A,B,A',B']$; as this element also belongs to $\overline M$, it belongs to $M$. 

By its definition, the map $c$ is a 2-coboundary of the Lie algebra $V$ with values in the module $\overline M$. It is therefore 
also a 2-cocycle in the same module. Since its takes its values in the submodule $M$, it is a 2-cocycle with values in $M$. 

We have proved : 

\begin{proposition}\label{prop:cocycle:1}
The map $(v,v')\mapsto c(v,v')$ defined by (\ref{def:c}) takes its values in $M\subset \overline M$. The map 
$c:V\times V\to M$ is a 2-cocycle of the abelian Lie algebra $V$ with values in the $V$-module $M$, so $c\in Z^2(V,M)$. 
\end{proposition}

\begin{lemma}\label{lem:linearmap}
For $u\in AB\C[A,B]$ we set 
\begin{align*}
\lambda_u^0:=&\frac{1}{2} A\cdot \big( \frac{u(A+A',B+B')-u(A+A',B)-u(A+A',B')}{A+A'}-\frac{u(A,B+B')-u(A,B)-u(A,B')}{A}\big) \\&
-((A,B)\leftrightarrow (A',B')).
\end{align*} 
Then the map $u\mapsto \lambda_u^0$ is a linear map from $AB\C[A,B]$ to $M$.
\end{lemma}

\begin{proof} 
Fix $u\in AB\C[A,B]$. For any $\phi\in B\C[B]$ we have $\phi(B+B')-\phi(B)-\phi(B')\in BB'\C[B,B']$. Consequently, 
$u(A,B+B')-u(A,B)-u(A,B')$ is in $ABB'\C[A,B,B']$. If we set $\psi(A,B,B'):=(u(A,B+B')-u(A,B)-u(A,B'))/A$ we then have 
$\psi(A,B,B')\in BB'\C[A,B,B']$. Also, for any $f\in \C[A]$ we have $f(A+A')-f(A)\in A'\C[A,A']$. So $\psi(A+A',B,B')-\psi(A,B,B')$ 
is in $A'BB'\C[A,A',B,B']$ and thus 
$$
\frac{1}{2}A\cdot(\psi(A+A',B,B')-\psi(A,B,B'))\in AA'BB'\C[A,B,B,B'].
$$ 
As the space $AA'BB'\C[A,A',B,B']$ is preserved under the exchange $((A,B)\leftrightarrow(A',B'))$ and as $\lambda_u^0=\big(\mathrm{id}-((A,B)\leftrightarrow(A',B'))\big)\big({1\over 2}A\cdot(\psi(A+A',B,B')-\psi(A,B,B'))\big)$, the element 
$\lambda_u^0$ belongs to $AA'BB'\C[A,A',B,B'];$ as this element is antisymmetric under this exchange, the element belongs to $M$.
\end{proof}

We introduce a section $s_0:V\to {\mathfrak{L}}_{\quot}$ of ${\mathfrak{L}}_{\quot}\to V,$ given by $s_0(A^{a+1}B^{b+1}):=(\mathrm{ad} x)^a(\mathrm{ad} y)^b([x,y])$.
We now define two linear maps $\Phi,\Psi:AB\C[A,B]\to M$ by $$\Phi(u):=\langle s_0(AB),s_0(u) \rangle, \ \Psi(u):=\delta_{AB}\cdot(\lambda_u+\lambda_u^0)-\delta_u\cdot\lambda_{AB}.$$

\begin{lemma}\label{lem:Phieq}
For $u\in AB\C[A,B]$ the equality 
$$
\Phi(A\cdot u)=(A+A')\cdot \Phi(u)+ABA'u(A',B')-A'B'Au(A,B)
$$ holds and for $u\in AB\C[B]$ the equality 
$$
\Phi(B\cdot u)=(B+B')\cdot \Phi(u)-ABB'u(A',B')+A'B'Bu(A,B)
$$ 
holds.
\end{lemma}

\begin{proof}
Let $\xi$ be an arbitrary element of $\overline{\mathfrak{L}}$. We have 
$$
\langle [x,y],[x,\xi]\rangle=[[x,y],[x,\xi]]+D_{[x,y]}([x,\xi])-D_{[x,\xi]}([x,y])=[[x,y],[x,\xi]]+[x,D_{[x,y]}(\xi)]-[x,[y,[x,\xi]]],
$$ 
where the first equality follows from the definition of the bracket $\langle,\rangle$ and the second equality follows from the facts that $D_{[x,y]}$ is a derivation killing $x$ and that $D_{[x,y]}$ is a derivation for which the images of $x$ and $y$ are known. Moreover, 
$$
\langle [x,y],\xi\rangle=[[x,y],\xi]+D_{[x,y]}(\xi)-D_{\xi}([x,y])=[[x,y],\xi]+D_{[x,y]}(\xi)-[x,[y,\xi]].
$$ 
So, 
\begin{align}\label{onestar} \langle [x,y],[x,\xi]\rangle-[x,\langle[x,y],\xi\rangle]=[[x,y],[x,\xi]],\end{align}
and similarly
\begin{align}\label{twostar}\langle [x,y],[y,\xi]\rangle-[y,\langle[x,y],\xi\rangle]=-[[x,y],[y,\xi]]. \end{align}

When $a,b\geq 0$, setting $\xi=(\mathrm{ad} x)^a(\mathrm{ad} y)^b([x,y])$ in Equation (\ref{onestar}), we obtain 
\begin{align}\label{threestar} 
\langle [x,y],(\mathrm{ad} x)^{a+1}(\mathrm{ad} y)^b([x,y])\rangle=[x,\langle [x,y], (\mathrm{ad} x)^a(\mathrm{ad} y)^b([x,y])\rangle]
+[[x,y],(\mathrm{ad} x)^{a+1}(\mathrm{ad} y)^b([x,y])].
\end{align} 
The bijection $\overline{\mathfrak{L}}^1/\overline{\mathfrak{L}}^2\simeq{\mathfrak{L}}_1\to ABA'B'\C[A,B,A',B']^{as}$ (see (\ref{def:V:M}) and (\ref{isos})) sends
the element $\langle[x,y],(\mathrm{ad} x)^a(\mathrm{ad} y)^b([x,y])\rangle$ 
to $\Phi(A^{a+1}B^{b+1}),$ 
the endomorphism $\mathrm{ad}(x)$ to multiplication by $A+A'$ and 
the commutator 
$[[x,y],(\mathrm{ad} x)^{a+1}(\mathrm{ad} y)^b([x,y])]$ to the element $AB(A')^{a+2}(B')^{a+1}-((A,B)\leftrightarrow(A',B'))$. 

Equation (\ref{threestar}) then implies that 
$$
\Phi(A^{a+2}B^{b+1})=(A+A')\cdot \Phi(A^{a+1}B^{b+1})+AB(A')^{a+2}(B')^{a+1}-A^{a+2}B^{a+1}A'B',
$$ 
which is the first identity in the lemma in the case $u=A^{a+1}B^{b+1}$. 

When $b\geq 0$, setting $\xi=(\mathrm{ad} y)^b([x,y])$ in equation (\ref{twostar}), we obtain \begin{align}\label{fourstar} \langle [x,y],
(\mathrm{ad} y)^{b+1}([x,y])\rangle=[y,\langle [x,y], (\mathrm{ad} y)^b([x,y])\rangle]-[[x,y],(\mathrm{ad} y)^{b+1}([x,y])].\end{align} 
The bijection $\overline{\mathfrak{L}}^1/\overline{\mathfrak{L}}^2\to ABA'B'\C[A,B,A',B']^{as}$ sends $\langle[x,y],(\mathrm{ad} y)^b([x,y])\rangle$ to 
$\Phi(AB^{b+1}),$ the endomorphism $\mathrm{ad}(y)$ to multiplication by $B+B'$ and the commutator $[[x,y],(\mathrm{ad} y)^b([x,y])]$ to 
the element $ABA'(B')^{b+2}-AB^{b+2}A'B'$. Equation (\ref{fourstar}) then implies that 
$$
\Phi(AB^{b+2})=(B+B')\cdot \Phi(AB^{b+1})-ABA'(B')^{b+2}+AB^{b+2}A'B',
$$ 
which is the second identity in the lemma in the case $u=AB^{b+1}$. 
\end{proof}

\begin{lemma}\label{lem:Psieq}
For $u\in AB\C[A,B]$ the equality 
$$
\Psi(A\cdot u)=(A+A')\cdot\Psi(u)+ABA'u(A',B')-A'B'Au(A,B)
$$ 
holds, and for $u\in AB\C[B]$ the equality 
$$
\Psi(B\cdot u)=(B+B')\cdot\Psi(u)-ABB'u(A',B')+A'B'Bu(A,B)
$$ holds.
\end{lemma}

\begin{proof} We compute 
\begin{align}\label{star} 
\Psi(u)=&\frac{AA'(B-B')}{A+A'}u(A+A',B+B')-\frac{A(AB'+A'B)}{A+A'}u(A+A',B)\nonumber 
\\&+\frac{A'(AB'+A'B)}{A+A'}u(A+A',B')+AB'u(A,B)-A'Bu(A',B'),
\end{align}
which implies the first claim in the lemma. Moreover, equation (\ref{star}) implies that for $l\geq 1$ one has 
$$
\Psi(AB^l)=AA'(B-B')(B+B')^l-AA'B^{l+1}+AA'(B')^{l+1},
$$ which implies the second claim.
\end{proof}

\begin{lemma}\label{lem:seq}
For $u\in AB\C[A,B]$ the identity 
$$
\langle s_0(AB),s_0(u)\rangle=\delta_{AB}\cdot(\lambda_u+\lambda_u^0)-\delta_u\cdot\lambda_{AB}
$$ holds.
\end{lemma}
\begin{proof}
By the second parts of Lemmas \ref{lem:Phieq} and \ref{lem:Psieq}, the sequences $(\Phi(AB^l))_{l\geq 1}$ and $(\Psi(AB^l))_{l\geq 1}$ obey the same recurrence relation; moreover, $\Phi(AB)=\langle s_0(AB),s_0(AB)\rangle=0=\Psi(AB),$ where the last equality follows from the fact that $\lambda_{AB}^0=0$. So, the initial terms of the two sequences coincide; we deduce that for $l\geq 1$ we have $\Phi(AB^l)=\Psi(AB^l)$.

By the first parts of Lemmas \ref{lem:Phieq} and \ref{lem:Psieq}, for a fixed integer $l\geq 1$, the sequences $(\Phi(A^kB^l))_{k\geq 1}$ and $(\Psi(A^kB^k))_{k\geq 1}$ obey the same recurrence relation. As their initial terms coincide, we deduce that these two sequences are equal so that $\Phi(A^kB^l)=\Psi(A^kB^l)$ for all $k,l\geq 1$, which implies the lemma.
\end{proof}

\begin{lemma}\label{lem:s0eq}
For $u,v\in AB\C[A,B]$ the equality 
$$
\langle s_0(u),s_0(v)\rangle=\delta_u\cdot(\lambda_v+\lambda_v^0)-\delta_v\cdot(\lambda_u+\lambda_u^0)
$$ holds.
\end{lemma}

\begin{proof} For $u,v\in AB\C[A,B]$, set 
$$
\Phi(u,v):=\langle s_0(u),s_0(v)\rangle, \quad 
\Psi(u,v):=\delta_u\cdot(\lambda_v+\lambda_v^0)-\delta_v\cdot(\lambda_u+\lambda_u^0).
$$ 
So, $\Phi, \Psi$ are linear maps $\Lambda^2(V)\to M,$ both satisfying the identity 
\begin{align}\label{iden:delta} 
\delta_u\cdot f(v,w)+\delta_v\cdot f(w,u)+\delta_w\cdot f(u,v)=0 
\end{align} 
for all $u,v,w\in V$. In the case $f=\Phi$, this identity follows from the fact that $\Phi$ is a $2$-cocycle for $V$ with values in $M$, and 
from the explicit description of the action of $V$ on $M$; in the case $f=\Psi$, this identity follows from an explicit calculation (using the 
fact that $\Psi$ is a $2$-coboundary and therefore a $2$-cocycle). 
\newline

Moreover, $\Phi(AB,u)=\Phi(u)$, $\Psi(AB,u)=\Psi(u)$. Combining these identities with (\ref{iden:delta}) for $w=AB$, $f=\Phi,\Psi$, 
we deduce 
$$
\delta_{AB}\cdot\Phi(u,v)=\delta_u\cdot\Phi(v)-\delta_v\cdot\Phi(u),\quad 
\delta_{AB}\cdot\Psi(u,v)=\delta_u\cdot\Psi(v)-\delta_v\cdot\Psi(u).
$$ 
As $\Phi(u)=\Psi(u)$, $\Phi(v)=\Psi(v)$, it follows that 
$$
\delta_{AB}\cdot(\Phi(u,v)-\Psi(u,v))=0,
$$ 
where the product is in $\C[A,A',B,B']$. As $\delta_{AB}=AB'+A'B,$ we conclude that $\Phi(u,v)=\Psi(u,v)$, which implies the lemma.
\end{proof}

For $u\in AB\C[A,B]$, we set $s(u):=s_0(u)+\lambda_u^0$. The map $s:V\to{\mathfrak{L}}_{\quot}$ is a section of the projection ${\mathfrak{L}}_{\quot}\to V$. Then \begin{align*}\langle s(u),s(v)\rangle&=\langle s_0(u)+\lambda_u^0,s_0(v)+\lambda_v^0\rangle=\langle s_0(u),s_0(v)\rangle
+\langle s_0(u),\lambda_v^0\rangle+\langle\lambda_u^0,s_0(v)\rangle\\&=\delta_u\cdot (\lambda_v+\lambda_v^0)-\delta_v\cdot(\lambda_u+\lambda_u^0)-\delta_u\cdot\lambda_v^0+\delta_v\cdot\lambda_u^0
\\&=\delta_u\cdot\lambda_v-\delta_v\cdot\lambda_u,\end{align*}
where the second inequality follows from commutativity of the Lie subalgebra $M$ of ${\mathfrak{L}}_{\quot}$ and the third equality follows from Lemma \ref{lem:s0eq} and Lemma \ref{lemme:actions}. All this proves: 

\begin{proposition}\label{prop:cocycle:2}
The map $s:V\to {\mathfrak{L}}_{\quot}$ is a section of the projection ${\mathfrak{L}}_{\quot}\to V$, such that for all $v,v'\in V$ one has 
$\langle s(v),s(v')\rangle=c(v,v')$.
\end{proposition}

\begin{remark} \rm
To the short exact sequence $0\to M\to \overline M\to \overline M/M\to 0$ of $V$-modules is associated to a long exact sequence 
in Lie algebra cohomology 
$$\ldots\to H^1(V,\overline M/M)\to H^2(V,M)\to H^2(V,\overline M)\to\ldots$$ Propositions \ref{prop:cocycle:1} and
\ref{prop:cocycle:2} then say that 

(a) the element $[{\mathfrak{L}}_{\quot}]\in H^2(V,M)$ corresponding to the exact sequence (\ref{suite:ex:L:quot}) lies in the kernel of the map $H^2(V,M)\to 
H^2(V,\overline M)$;

(b) if $\overline{\lambda}$ is the linear map $V\to \overline M/M$ defined by the composition of $\lambda:V\to\overline M$ with the projection 
$\overline M\to \overline M/M$, then $\overline{\lambda}$ is an element of $H^1(V,\overline M/M)$;

(c) $[{\mathfrak{L}}_{\quot}]$ is the image of $\overline{\lambda}$ under the connecting homomorphism $H^1(V,\overline M/M)\to H^2(V,M)$. 
\end{remark}

\section{Lower bound for the structure associated with the Lie subalgebra of an extension of abelian Lie algebras}\label{Section4}

In this section, we associate to each vector space $V$ (or more generally each object in a monoidal abelian category $\mathcal C_0$)
a category of {\it $V$-structures} (Subsection \ref{subsection:structures}). We then show that any extension $0\to\g_1\to\g\to\g_0\to0$ 
gives rise to a $\g_0$-structure $\mathbb{M}(\g)$. In the same situation, we associate an $X$-structure $\mathbb{M}^{min}(X)$ to any 
subobject $X\hookrightarrow\g_0$ (Subsection \ref{subsection:lower:bound}). The main result of the section is Proposition 
\ref{proposition:lower:bound}, which states a lower bound property of $\mathbb{M}^{min}(X)$. In Subsection \ref{subsection:GF}, we make 
explicit the particular case where $\mathcal C_0$ is the category $GF$ of graded-filtered vector spaces, and the consequences of 
Proposition \ref{proposition:lower:bound} in terms of Hilbert-Poincar\'e series. 

\subsection{Graded modules associated to modules over symmetric algebras}

Let $V$ be a vector space and let $M$ be a module over the symmetric algebra $S(V)$. Set 
$\mathrm{gr}_V(M):=\oplus_i\mathrm{gr}^i_V(M)$, where $\mathrm{gr}^i_V(M)=S^i(V)\cdot M/
S^{i+1}(V)\cdot M$. Then $\mathrm{gr}_V(M)$ is a graded module over $S(V)$, where $S(V)$ is graded
by assigning degree 1 to $V$. 

\subsection{Structures associated with an extension of abelian Lie algebras} \label{subsection:structures}

For $V$ a vector space, define a {\it $V$-structure} as the data of: 
\begin{enumerate}
\item a graded vector space ${\mathbb M}={\mathbb M}_0\oplus{\mathbb M}_1\oplus\cdots$, with ${\mathbb M}_i$ of degree $i+2$; 
\item a $S(V)$-module structure on ${\mathbb M}$, compatible with the grading of $S(V)$ for which $V$ has degree 1, and
\item a 
linear map $\Lambda^2(V)\to {\mathbb M}_0$,  $v\wedge v'\mapsto \{v,v'\}$, such that the identity 
$v\cdot\{v',v''\}+\mathrm{cycl.\ perm.}=0$ holds. 
\end{enumerate}
One defines morphisms of $V$-structures in the natural way. 
$V$-structures then form a category, where a morphism is onto (resp., epi) iff the  underlying morphism of graded $S(V)$-modules
is injective (resp., surjective). 

Let $0\to\g_1\to\g\to\g_0\to0$ be an extension of abelian Lie algebras. As we have seen (Subsection \ref{subsection:invariants}), 
it gives rise to a $S(\g_0)$-module structure on $\g_1$. Let ${\mathbb M}(\g)$ be the graded module $\mathrm{gr}_{\g_0}(\g_1)$, 
with degrees shifted by 2; so ${\mathbb M}_i(\g)=\mathrm{gr}_{\g_0}^i(\g_1)$. Let $\tilde c\in Z^2(\g_0,\g_1)$ be any 
cocycle representing the extension $0\to\g_1\to\g\to\g_0\to0$. Composing the map $\tilde c:\Lambda^2(\g_0)\to\g_1$
with the projection $\g_1\to{\mathbb M}_0(\g)$, one obtains a linear map $\Lambda^2(\g_0)\to{\mathbb M}_0(\g)$ which 
turns out to be independent of the choice of $\tilde c$.  One then checks that {\it the pair 
${\mathbb M}(\g):=(\mathrm{graded\ vector\ space\ }{\mathbb M}(\g),\mathrm{\ linear\ map\ }
\Lambda^2(\g_0)\to{\mathbb M}_0(\g))$ yields a $\g_0$-structure
canonically attached to the extension $0\to\g_1\to\g\to\g_0\to0$.}  One also checks that this 
$\g_0$-structure can be defined in terms of the invariants ($\g_0$-module structure on $\g_1$, 
cocycle class in $H^2(\g_0,\g_1)$) attached to the extension. 

\subsection{Lower bound structures in the case of Lie subalgebras} \label{subsection:lower:bound}

Let $0\to\g_1\to\g\to\g_0\to 0$ be an extension of abelian Lie algebras. Let $\h\subset\g$ be a Lie subalgebra of $\g$
and set $X:=\mathrm{im}(\h\to\g\to\g_0)$. Set $\mathit{Mod}:=\h\cap\g_1$. Then $\h$ is an extension of abelian Lie algebras
$0\to \mathit{Mod}\to\h\to X\to0$. We set: 
\begin{equation} \label{def:M:h}
{\mathbb M}(\h):=\mathrm{the\ }X\text{-structure\ attached\ to\ the\ extension\ }0\to \mathit{Mod}\to\h\to X\to0
\end{equation}
(see previous Subsection). 

The morphism $\mathit{Mod}\hookrightarrow\g_1$ of $S(X)$-modules induces a morphism $\mathrm{gr}_X(\mathit{Mod})
\to\mathrm{gr}_X(\g_1)$ of graded $S(X)$-modules. Let $Q$ be the image of this morphism; then $Q$ is a graded $S(X)$-module, 
which when equipped with the natural map $\Lambda^2(X)\to Q^0$ gives rise to an $X$-structure. The above morphism then factors 
as $\mathrm{gr}_X(\mathit{Mod})\twoheadrightarrow Q\hookrightarrow\mathrm{gr}_X(\g_1)$, 
whose degree 0 part is denoted $\mathrm{gr}^0_X(\mathit{Mod})\twoheadrightarrow Q^0\hookrightarrow\mathrm{gr}_X^0(\g_1)$. 
The cocycle $c$ gives rise to a linear map $\Lambda^2(X)\to\mathrm{gr}^0_X(\mathit{Mod})$. 

If $M$ is a graded $S(X)$-module and if $Sub$ is a subspace of the degree 0 part of $M$, we denote by 
$\langle Sub\rangle_M$ the $S(X)$-submodule of $M$ generated by $Sub$, so $\langle Sub\rangle_M=\mathrm{im}(S(V)\otimes Sub\to M)$.

Then there is an injection of graded $S(X)$-modules 
$$
\langle\mathrm{im}(\Lambda^2(X)\to\mathrm{gr}^0_X(\mathit{Mod}))\rangle_{\mathrm{gr}_X(\mathit{Mod})}
\hookrightarrow \mathrm{gr}_X(\mathit{Mod}). 
$$
The image of the left side of this inclusion in $Q$ under the morphism $\mathrm{gr}_X(\mathit{Mod})\to Q$ is the space $\langle\mathrm{im}
(\Lambda^2(X)\to Q^0)\rangle_{Q}$, and the image of this space in $\mathrm{gr}_X(\g_1)$ under the morphism $Q\to\mathrm{gr}_X(\g_1)$ 
is $\langle\mathrm{im}(\Lambda^2(X)\to \mathrm{gr}^0_X(\g_1))\rangle_{\mathrm{gr}_X(\g_1)}$. Since the morphism 
$Q\to\mathrm{gr}_X(\g_1)$ is injective, its restriction to a morphism $\langle\mathrm{im}(\Lambda^2(X)\to Q^0)\rangle_{Q}
\to\langle\mathrm{im}(\Lambda^2(X)\to \mathrm{gr}^0_X(\g_1))\rangle_{\mathrm{gr}_X(\g_1)}$ is an isomorphism. 

It follows that there is a diagram 
\begin{align*}  \xymatrix{
\langle\mathrm{im}(\Lambda^2(X)\to\mathrm{gr}^0_X(\mathit{Mod}))\rangle_{\mathrm{gr}_X(\mathit{Mod})} \ar@{->>}[r] \ar@{^{(}->}[d]& 
\langle\mathrm{im}(\Lambda^2(X)\to Q^0)\rangle_{Q} \ar@{^{(}->}[d] \ar^{\!\!\!\!\!\!\!\!\!\!\!\!\!\!\!\!\!\!\simeq}[r] & 
\langle\mathrm{im}(\Lambda^2(X)\to \mathrm{gr}^0_X(\g_1))\rangle_{\mathrm{gr}_X(\g_1)} \ar@{^{(}->}[d]
\\ \mathbb{M}(\h) \simeq \mathrm{gr}_X(\mathit{Mod}) \ar@{->>}[r] & 
Q \ar@{^{(}->}[r]& \mathrm{gr}_X(\g_1)}
\end{align*} 
of $S(X)$-modules. Set 
\begin{equation} \label{def:M:min}
\mathbb{M}^{min}(X):=\langle\mathrm{im}(\Lambda^2(X)\to \mathrm{gr}^0_X(\g_1))\rangle_{\mathrm{gr}_X(\g_1)};  
\end{equation}
When equipped with the natural map $\Lambda^2(X)\to\mathbb{M}^{min}(X)$, this graded $S(X)$-module gives rise to 
a $X$-structure. 

From the above diagram, one extracts the following diagram $\mathbb{M}(\h)\twoheadrightarrow Q\hookleftarrow\mathbb{M}^{min}(X)$, from 
where one derives the following result. 

\begin{proposition} \label{proposition:lower:bound}
Let $0\to\g_1\to\g\to\g_0\to0$ be an extension of abelian Lie algebras and let $X$ be any vector subspace of $\g_0$. 
If $\h$ is any Lie subalgebra of $\g$ such that $\mathrm{im}(\h\subset\g\to\g_0)=X$, then the $X$-structure 
$\mathbb{M}^{min}(X)$ defined in (\ref{def:M:min}) is a subquotient of the $X$-structure $\mathbb{M}(\h)$ defined in 
(\ref{def:M:h}). 
\end{proposition}
 
\subsection{Graded-filtered analogues} \label{subsection:GF}

In the preceding part of this section, the basic category is the monoidal category of vector spaces. One can replace it by the monoidal 
category $GF$ of {\it graded-filtered} vector spaces, whose objects are vector spaces $V$ equipped with a grading $V=\oplus_{n\geq 0}V[n]$ 
(the weight grading) and a decreasing filtration $V=F^0(V)\supset F^1(V)\supset\cdots$ (the depth filtration), such that: (i) each $V[n]$ is 
finite-dimensional, (ii) the filtration and the grading are compatible, i.e., for any $i\geq 0$, $F^i(V)=\oplus_{n\geq 0}F^i(V)\cap V[n]$.  

The construction of Subsection \ref{subsection:lower:bound} has the following analogue. To Lie algebras $\g,\g_0,\g_1$ in $GF$, 
fitting in an exact sequence $0\to\g_1\to\g\to\g_0\to0$ and such that $\g_0$, $\g_1$ are abelian, one associates
as before the $S(\g_0)$-module $\mathbb{M}(\g)$ and the map $\Lambda^2(\g_0)\to\mathbb{M}(\g)$, which 
make sense in the category $GF$. 

The analogue of Proposition \ref{proposition:lower:bound} is the following: {\it $0\to\g_1\to\g\to\g_0\to0$ is an extension of Lie 
algebras in $GF$ with $\g_0,\g_1$ abelian and $X$ is a subobject of $\g_0$ in $GF$. The $X$-structure $\mathbb{M}^{min}(X)$
then makes sense in $GF$. Then for any Lie subalgebra $\h$ of $\g$ in $GF$, such that $\mathrm{im}(\h\subset\g\to\g_0)=X$, 
$\mathbb{M}^{min}(X)$ is a subquotient of $\mathbb{M}(\h)$ as $X$-structures in $GF$.} 

To the object $V$ of $GF$, one associates its Hilbert-Poincar\'e series $P_V(t):=\sum_{n\geq 0}\mathrm{dim}V[n]t^n\in
\N[[t]]$ and the double series $P_V(t,u):=\sum_{n,i\geq 0}\mathrm{dim}\ \mathrm{gr}_F^i(V[n])t^nu^i\in
\N[u][[t]]$. We have $P_X(t)=P_X(t,1)$. As these assignments $GF\to\N[[t]]$, $GF\to\N[u][[t]]$ are non-decreasing,
one derives in the above situation: (1) $P_{\mathbb{M}^{min}(X)}(t,u)\leq P_{\mathbb{M}(\h)}(t,u)$, and 
(2) $P_{\mathbb{M}^{min}(X)}(t)\leq P_{\mathbb{M}(\h)}(t)$, $\leq$ meaning that the difference of two sides are series with 
non-negative coefficients. Since $P_{\mathbb{M}(\h)}(t)=P_{\h\cap\g_1}(t)$ and $P_{\mathbb{M}(\h)}(t,u)=P_{\h\cap\g_1}(t,u)$, we obtain: 
\begin{enumerate}
\item $P_{\mathbb{M}^{min}(X)}(t,u)\leq P_{\h\cap\g_1}(t,u)$, 
\item $P_{\mathbb{M}^{min}(X)}(t)\leq P_{\h\cap\g_1}(t)$.  
\end{enumerate}
  
\section{The $\Sigma$-structure $[\mathrm{gr}_{lcs}^1({\mathfrak{grt}}_1)]$ and the lower bound ${\mathbb M}^{min}(\Sigma)$}\label{Section5}

In this section, we give the consequences of the results of Section 4 for structures associated to the Lie algebra ${\mathfrak{grt}}_1$: 
we express a lower bound ${\mathbb M}^{min}(\Sigma)$ for the $\Sigma$-structure $[\mathrm{gr}_{lcs}^1({\mathfrak{grt}}_1)]$
(Subsections \ref{subsect:lower:bound}, \ref{constr:M:min}). We make explicit the ingredients of the construction of ${\mathbb M}^{min}(\Sigma)$
in Subsection \ref{subsection:ingredients} and relate this construction to some commutative algebra (Subsection \ref{subsection:comm:alg}). 

\subsection{The lower bound result} \label{subsect:lower:bound}

In Subsection \ref{subsec:Lquot}, we defined a Lie algebra ${\mathfrak{L}}_{quot}$ in the category $GF$, fitting in an exact sequence 
$0\to{\mathfrak{L}}_1\to{\mathfrak{L}}_{quot}\to{\mathfrak{L}}_0\to0$ of Lie algebras in $GF$, where both ${\mathfrak{L}}_0$ and ${\mathfrak{L}}_1$ are abelian. There is a sequence of 
morphisms of Lie algebras in $GF$, ${\mathfrak{grt}}_1\hookrightarrow\overline{\mathfrak{L}}\twoheadrightarrow{\mathfrak{L}}_{quot}$; composing them, 
we obtain a morphism ${\mathfrak{grt}}_1\to{\mathfrak{L}}_{quot}$. Set 
$$
{\mathfrak{H}}:=\mathrm{im}({\mathfrak{grt}}_1\to{\mathfrak{L}}_{quot}), 
$$
and $\Sigma:=\mathrm{im}({\mathfrak{H}}\subset{\mathfrak{L}}_{quot}\to{\mathfrak{L}}_0)$. Then $\Sigma=\mathrm{im}({\mathfrak{grt}}_1
\hookrightarrow\overline{\mathfrak{L}}\twoheadrightarrow{\mathfrak{L}}_0)$, therefore 
$$
\Sigma=\mathrm{gr}_{lcs}^0({\mathfrak{grt}}_1)
$$
The image of $\Sigma$ under the isomorphism ${\mathfrak{L}}_0\simeq AB\C[A,B]$ has been computed in \cite{En1}, namely 
$$
\Sigma\simeq \oplus_{k\mathrm{\ odd\ }\geq 3}\C\cdot\sigma_k, \mathrm{\ where\ }
\sigma_k(A,B):=A^k+B^k+(-A-B)^k\in AB\C[A,B]
$$
for any odd $k\geq 3$. The structure of $\Sigma$ as an object of $GF$ is described as follows: the weight degree $n$
part of $\Sigma$ is $\Sigma[n]=\C\cdot\sigma_n$ is $n$ is odd $\geq 3$, and is $0$ otherwise; the depth filtration is 
described by $F^0_{dpth}(\Sigma)=F^1_{dpth}(\Sigma)=\Sigma$, $F^2_{dpth}(\Sigma)=F^3_{dpth}(\Sigma)=\ldots=0$. 

One computes ${\mathfrak{H}}\cap{\mathfrak{L}}_1=\mathrm{ker}({\mathfrak{grt}}_1\hookrightarrow\overline{\mathfrak{L}}\twoheadrightarrow{\mathfrak{L}}_0)
=\mathrm{gr}_{lcs}^1({\mathfrak{grt}}_1)$, so there is a commutative diagram 
\begin{align} \label{morph:ex:seq} \xymatrix{
0\ar[r]&{\mathfrak{L}}_1\ar[r] & 
{\mathfrak{L}}_{quot}  \ar[r] & 
{\mathfrak{L}}_0 \ar[r]&0
\\ 0\ar[r]&\mathrm{gr}_{lcs}^1({\mathfrak{grt}}_1)  \ar[r]\ar@{^{(}->}[u] & 
{\mathfrak{H}} \ar[r]\ar@{^{(}->}[u]& \Sigma\ar@{^{(}->}[u]\ar[r]&0}
\end{align} 
The situation is therefore that of the $GF$ version of Proposition \ref{proposition:lower:bound} described in Subsection 
\ref{subsection:GF} with the following correspondence

\begin{tabular}{|l|l|l|l|}
  \hline
 $0\to\g_1\to\g\to\g_0\to0$ &  $0\to{\mathfrak{L}}_1\to{\mathfrak{L}}_{quot}\to{\mathfrak{L}}_0\to0$ \\
  \hline
$\mathit{Mod}$, ${\mathfrak{h}}$, $X$ & $\mathrm{gr}_{lcs}^1({\mathfrak{grt}}_1)$, ${\mathfrak{H}}$, $\Sigma$  \\
 \hline
\end{tabular}

Proposition \ref{proposition:lower:bound} and Subsection \ref{subsection:GF}  then yield the following statement. Let 
$[\mathrm{gr}_{lcs}^1({\mathfrak{grt}}_1)]:={\mathbb{M}}({\mathfrak{H}})$ be the $\Sigma$-structure associated with the bottom line of 
(\ref{morph:ex:seq}); 
it is given by (a) the $S(\Sigma)$-module ${\mathbb{M}}={\mathbb{M}}_0\oplus{\mathbb{M}}_1\oplus\ldots$, 
where ${\mathbb{M}}_i= S^i(\Sigma)\cdot\mathrm{gr}_{lcs}^1({\mathfrak{grt}})/ S^{i+1}(\Sigma)\cdot \mathrm{gr}_{lcs}^1({\mathfrak{grt}})$, 
(b) the linear map $\Lambda^2(\Sigma)\to{\mathbb{M}}_0$. 

\begin{proposition} \begin{enumerate}
\item The $\Sigma$-structure ${\mathbb{M}}^{min}(\Sigma)$ is a subquotient of the $\Sigma$-structure $[\mathrm{gr}_{lcs}^1({\mathfrak{grt}}_1)]$; 
\item $P_{{\mathbb{M}}^{min}(\Sigma)}(t)\leq P_{\mathrm{gr}_{lcs}^1({\mathfrak{grt}}_1)}(t)$ (generating series relative to the 
weight degree); 
\item $P_{{\mathbb{M}}^{min}(\Sigma)}(t,u)\leq P_{\mathrm{gr}_{lcs}^1({\mathfrak{grt}}_1)}(t,u)$ (generating series relative to the 
weight degree and the depth filtration). 
\end{enumerate}
\end{proposition}

\subsection{Construction of the $\Sigma$-structure ${\mathbb{M}}^{min}(\Sigma)$} \label{constr:M:min}

We now recall the construction of the $\Sigma$-structure ${\mathbb{M}}^{min}(\Sigma)$. 

1) The exact sequence $0\to{\mathfrak{L}}_1\to{\mathfrak{L}}_{quot}\to{\mathfrak{L}}_0\to0$ yields a $S({\mathfrak{L}}_0)$-module structure on ${\mathfrak{L}}_1$, and therefore 
a $S(\Sigma)$-module structure on the same space. One associates to it the graded module 
$\mathrm{gr}_\Sigma({\mathfrak{L}}_1)=\mathrm{gr}^0_\Sigma({\mathfrak{L}}_1)\oplus\mathrm{gr}^1_\Sigma({\mathfrak{L}}_1)\oplus\cdots$, 
where $\mathrm{gr}^i_\Sigma({\mathfrak{L}}_1)=S^i(\Sigma)\cdot{\mathfrak{L}}_1/S^{i+1}(\Sigma)\cdot{\mathfrak{L}}_1$ over the graded algebra 
$S(\Sigma)$ (the grading implied here is the $\Sigma$-grading, for which $\Sigma$ has degree 1 in $S(\Sigma)$). 

2) A linear map $\{,\}:\Lambda^2(\Sigma)\to\mathrm{gr}^0_\Sigma({\mathfrak{L}}_1)={\mathfrak{L}}_1/\Sigma\cdot{\mathfrak{L}}_1$ is defined by 
$\sigma\wedge\sigma'\mapsto($class of $c(\sigma,\sigma'))$, where $c\in Z^2({\mathfrak{L}}_0,{\mathfrak{L}}_1)$ is any cocycle 
representing the extension $0\to{\mathfrak{L}}_1\to{\mathfrak{L}}_{quot}\to{\mathfrak{L}}_0\to0$. 

3) One then sets 
$$
{\mathbb{M}}^{min}_0(\Sigma):=\mathrm{im}(\Lambda^2(\Sigma)\stackrel{\{,\}}{\to}\mathrm{gr}^0_\Sigma({\mathfrak{L}}_1))\subset 
\mathrm{gr}^0_\Sigma({\mathfrak{L}}_1), \quad {\mathbb{M}}^{min}_i(\Sigma):=S^i(\Sigma)\cdot{\mathbb{M}}^{min}_0(\Sigma)\subset 
\mathrm{gr}^i_\Sigma({\mathfrak{L}}_1)
$$
for $i\geq 1$ (the latter formula is also correct for $i=0$). 

4) Then ${\mathbb{M}}^{min}(\Sigma):={\mathbb{M}}^{min}_0(\Sigma)\oplus{\mathbb{M}}^{min}_1(\Sigma)\oplus\cdots$
is a graded $S(\Sigma)$-module, equipped with a map $\Lambda^2(\Sigma)\to{\mathbb{M}}^{min}_0(\Sigma)$ satisfying the axioms
of a $\Sigma$-structure. This $\Sigma$-structure will be also called ${\mathbb{M}}^{min}(\Sigma)$. 

\subsection{Ingredients of the construction of ${\mathbb{M}}^{min}(\Sigma)$}\label{subsection:ingredients}

The construction of ${\mathbb{M}}^{min}(\Sigma)$ in the previous Subsection uses the following ingredients: 1) the action of $S(\Sigma)$ 
on ${\mathfrak{L}}_1$, 2) the linear map $\{,\}:\Lambda^2(\Sigma)\to\mathrm{gr}^0_\Sigma({\mathfrak{L}}_1)$. In the present Subsection, we compute these objects 
explicitly. 

According to (\ref{def:V:M}) and (\ref{isos}), there are isomorphisms ${\mathfrak{L}}_1\simeq M=ABA'B'\C[A,B,A',B']^{as}$. For $k\geq 0$, set 
\begin{align}\label{def:sigmai}
\tilde\sigma_k(A,B,A',B'):=&A^k+B^k+(-A-B)^k+(A')^k+(B')^k+(-A'-B')^k+(-A-A')^k+(-B-B')^k\nonumber \\&+(A+B+A'+B')^k\in
\C[A,B,A',B']^{sym}. 
\end{align}
For any odd $k\geq 3$, we have
\begin{equation}\label{comp:delta}
\delta_{\sigma_k}=-\tilde\sigma_k, 
\end{equation} 
which together with Lemma \ref{lemme:actions} implies: 

\begin{lemma} \label{lemma:5:2} (Action of $\Sigma$ on ${\mathfrak{L}}_1$) The isomorphism ${\mathfrak{L}}_1\simeq M$ takes the restriction 
$\Sigma\otimes{\mathfrak{L}}_1\to{\mathfrak{L}}_1$ of the action map of $S(\Sigma)$ on ${\mathfrak{L}}_1$ to the map 
$$
\Sigma\otimes M\to M, \quad \sigma_k\otimes m\mapsto \tilde\sigma_k\cdot m, 
$$ 
where $\cdot$ is the product of symmetric and antisymmetric (under $(A,B)\leftrightarrow(A',B')$) polynomials.  
\end{lemma}

For any odd $k\geq 3$, set 
\begin{align} \label{lambda:i}
\lambda_k(A,B,A',B'):=&(A+B+A')^k-(A+A'+B')^k-(A+B+B')^k+(B+A'+B')^k\nonumber \\&+(A+B')^k-(A'+B)^k
\in\C[A,B,A',B']^{as}=\overline M. 
\end{align}
One checks that $\lambda_k=2\lambda_{\sigma_k}$, where $\sigma\mapsto\lambda_\sigma$ is the map defined in 
Subsection \ref{subsection:cocycle}. Together with explicit formula (\ref{def:c}) for the cocycle $c$, computation (\ref{comp:delta})
of the restriction of $v\mapsto\delta_v$ to $\Sigma$, and Proposition
\ref{prop:cocycle:1} on the functional properties of this cocycle, this implies that for $i,j$ odd $\geq 3$,
$c(\sigma_i,\sigma_j)=-2(\tilde\sigma_i\cdot\lambda_j-\tilde\sigma_j\cdot\lambda_i)\in M$. We will set 
\begin{equation} \label{tau:ij}
\tau_{ij}:=\tilde\sigma_i\cdot\lambda_j-\tilde\sigma_j\cdot\lambda_i\in M. 
\end{equation} 
Proposition \ref{prop:cocycle:2} relating the cocycle $c$ with the extension $0\to{\mathfrak{L}}_1\to{\mathfrak{L}}_{quot}\to{\mathfrak{L}}_0\to0$, 
relation $c(\sigma_i,\sigma_j)=-2\tau_{ij}$, and the relation between the bracket $\{,\}$ and cocycles (Subsection 
\ref{constr:M:min}, 2)) imply: 

\begin{lemma} (Map $\{,\}:\Lambda^2(\Sigma)\to\mathrm{gr}_\Sigma^0({\mathfrak{L}}_1)$) The map 
$\{,\}:\Lambda^2(\Sigma)\to\mathrm{gr}_\Sigma^0({\mathfrak{L}}_1)\simeq{\mathfrak{L}}_1/\Sigma\cdot{\mathfrak{L}}_1=M/\Sigma\cdot M$ is given by 
$\{\sigma_i,\sigma_j\}=($class of $-2\tau_{ij})$, for $i,j$ odd $\geq 3$. 
\end{lemma} 

\subsection{Expression of ${\mathbb{M}}^{min}(\Sigma)$ in terms of commutative algebra}\label{subsection:comm:alg}

\subsubsection{The ring $\mathbb{A}$} \label{A2}

The vector space $\C^2$ can be equipped with a structure of irreducible module over the 
symmetric group $S_3$. The external square of this module in a $(S_3)^2$-module with underlying vector space $\mathbf{V}:=
(\C^2)^{\otimes 2}$. The group $(S_3)^2$ is a subgroup of the wreath product $G:=S_3\wr S_2=(S_3)^2\rtimes S_2$ (where $S_2$
acts by permutation of factors); the $(S_3)^2$-module structure of $\mathbf{V}$ can be extended to this group by the condition 
that the non-trivial element of $S_2$ acts by permutation of the factors of the tensor square. This defines a $G$-module structure on 
$\mathbf{V}$. One then defines the corresponding invariant ring 
$$\mathbb{A}:=S(\mathbf{V})^G. $$
The grading of the symmetric algebra $S(\mathbf{V})$ (for which the elements of $\mathbf{V}$ have degree 1) restricts to a grading 
of $\mathbb{A}$, which will be called the {\it weight degree.}

The ring $\mathbb{A}$ has the following interpretation in terms of polynomial algebras. Recall that $A,B,A',B'$ are free commutative 
variables. We identify the ring $S(\mathbf{V})$ with the polynomial ring $\C[A,B,A',B']$. A generating family of $G=S_3\wr S_2$
is: 

$(12)\times 1$, $(123)\times 1$, the nontrivial element of $S_2\subset S_3\wr S_2$. 

These generators act on a polynomial in $\C[A,B,A',B']$ by replacing $(A,B,A',B')$ by the following values: 

$(B,A,B',A')$, $(B,-A-B,B',-A'-B')$, $(A,A',B,B')$. 

We then identify $\mathbb{A}$ with the subring of $\C[A,B,A',B']$ of all polynomials invariant under these replacements. 

Define the {\it depth grading} in $S(\mathbf{V})=\C[A,B,A',B']$ by assigning degree 0 to $A,A'$ and degree 1 to $B,B'$. 
The corresponding decreasing filtration is defined by 
$F^i_{dpth}(S(\mathbf{V})):=\oplus_{j|j\geq i}\{$part of $S(\mathbf{V})$ of depth degree $j\}$. We also call {\it depth filtration}
the induced filtration on ${\mathbb{A}}$, defined by $F^i_{dpth}({\mathbb{A}}):=F^i_{dpth}(S(\mathbf{V}))\cap{\mathbb{A}}$
for any $i\geq 0$. This is an algebra filtration on ${\mathbb{A}}$, compatible with the weight degree grading of ${\mathbb{A}}$. 

\subsubsection{The ring $\mathrm{gr}_\Sigma({\mathbb{A}})$}

The ring ${\mathbb{A}}$ is a subring of $\C[A,B,A',B']^{sym}$. There is an injective linear map 
$\Sigma\hookrightarrow\C[A,B,A',B']^{sym}$, given by $\sigma_i\mapsto\tilde\sigma_i$ for $i$ odd $\geq 3$. 
We use this map to identify $\Sigma$ with a subspace of $\C[A,B,A',B']^{sym}$. Then there is a double inclusion 
\begin{equation}\label{inclusions}
\Sigma\subset{\mathbb{A}}\subset\C[A,B,A',B']^{sym}. 
\end{equation}
Define ${\mathbb{I}}$ to be the ideal of  ${\mathbb{A}}$ generated by $\Sigma$, so 
$$
{\mathbb{I}}=(\Sigma)\subset{\mathbb{A}}. 
$$
One defines a decreasing ring filtration ${\mathbb{A}}=F^0({\mathbb{A}})\supset F^1({\mathbb{A}})\supset\cdots$
by $F^k({\mathbb{A}}):={\mathbb{I}}^k$ for $k\geq 0$, and $\mathrm{gr}_\Sigma({\mathbb{A}})$ as the associated graded ring.
So 
$$
\mathrm{gr}_\Sigma({\mathbb{A}})=\oplus_{k\geq 0}\mathrm{gr}_\Sigma^k({\mathbb{A}}), \quad \mathrm{where}\quad
\mathrm{gr}_\Sigma^k({\mathbb{A}})={\mathbb{I}}^k/{\mathbb{I}}^{k+1}. 
$$ 
The grading of ${\mathbb{A}}$ corresponding to this decomposition will be called the {\it $\Sigma$-grading}. So 
$\mathrm{gr}_\Sigma({\mathbb{A}})$ is bigraded by (weight degree, $\Sigma$-degree).

Define the {\it depth filtration} on each $\mathrm{gr}_\Sigma^k({\mathbb{A}})$, $k\geq 0$ by 
$$
F^i_{dpth}(\mathrm{gr}_\Sigma^k({\mathbb{A}})):=(F^i_{dpth}({\mathbb{A}})\cap{\mathbb{I}}^k)/(
F^i_{dpth}({\mathbb{A}})\cap{\mathbb{I}}^{k+1})\hookrightarrow \mathrm{gr}_\Sigma^k({\mathbb{A}})
$$
for any $i\geq 0$ and set $F^i_{dpth}(\mathrm{gr}_\Sigma^k({\mathbb{A}})):=\oplus_{k\geq 0}F^i_{dpth}(\mathrm{gr}_\Sigma^k({\mathbb{A}}))$. 
This defines an algebra filtration on $\mathrm{gr}_\Sigma({\mathbb{A}})$, compatible with its bigrading by (weight degree, $\Sigma$-degree).

\subsubsection{The ring ${\mathbb{S}}$} \label{section:S}\label{543}

Composing the linear map $\Sigma\to{\mathbb{I}}$ with the projection ${\mathbb{I}}\to{\mathbb{I}}/{\mathbb{I}}^2
=\mathrm{gr}_\Sigma^1({\mathbb{A}})$, 
one obtains a linear map $\Sigma\to\mathrm{gr}_\Sigma^1({\mathbb{A}})\hookrightarrow\mathrm{gr}_\Sigma({\mathbb{A}})$. Equip the 
symmetric algebra $S(\Sigma)$ of $\Sigma$ with the $\Sigma$-grading, for which the degree of $\Sigma$ is 1. Then $S(\Sigma)$ is bigraded 
by ($\Sigma$-degree, weight degree). The linear map $\Sigma\to\mathrm{gr}_\Sigma({\mathbb{A}})$ then induces a bigraded morphism 
$S(\Sigma)\to\mathrm{gr}_\Sigma({\mathbb{A}})$. We then define
$$
{\mathbb{S}}:=\mathrm{im}(S(\Sigma)\to\mathrm{gr}_\Sigma({\mathbb{A}})) ;  
$$
this is a bigraded subring of $\mathrm{gr}_\Sigma({\mathbb{A}})$. Its decomposition for the $\Sigma$-degree is denoted 
${\mathbb{S}}=\oplus_{k\geq 0}{\mathbb{S}}[k]$, where ${\mathbb{S}}[k]:=\mathrm{im}(S^k(\Sigma)\to\mathrm{gr}_\Sigma({\mathbb{A}}))$
(a subspace of $\mathrm{gr}_\Sigma^k({\mathbb{A}})$). We define the depth filtration on ${\mathbb{S}}[k]$ by $F^k_{dpth}({\mathbb{S}}[k])=
{\mathbb{S}}[k]$, $F^{k+1}_{dpth}({\mathbb{S}}[k])=0$, and the depth filtration on ${\mathbb{S}}$ by $F^i_{dpth}({\mathbb{S}})
:=\oplus_{k\geq 0}F^i_{dpth}({\mathbb{S}}[k])$. This filtration of ${\mathbb{S}}$ is compatible with the bigrading and with the 
morphisms $S(\Sigma)\twoheadrightarrow{\mathbb{S}}\hookrightarrow\mathrm{gr}_\Sigma({\mathbb{A}})$. 

\subsubsection{The ${\mathbb{A}}$-module structure of $M$}\label{544}

The linear map $\Sigma\to{\mathbb{A}}$ (see (\ref{inclusions})) induces a ring morphism $S(\Sigma)\to{\mathbb{A}}$. 
Recall also that $M$ is a $S(\Sigma)$-module (Subsection \ref{lemma:5:2}). 

As $M=ABA'B'\C[A,B,A',B']^{as}$ is a module over $\C[A,B,A',B']^{sym}$, and as ${\mathbb{A}}$ is a subring of $\C[A,B,A',B']^{sym}$, 
{\it the $S(\Sigma)$-module structure of $M$ lifts to a module structure over the ring ${\mathbb{A}}$.} 

Define the {\it weight grading} of $M$ by assigning degree 1 to $A,B,A',B'$, the {\it depth grading} of $M$ by assigning degree 0 to $A,A'$
and degree 1 to $B,B'$, and the {\it depth filtration} of $M$ as the corresponding decreasing filtration. Then the ${\mathbb{A}}$-module 
structure of $M$ is compatible with weights gradings and depth filtrations.  

\subsubsection{The $\mathrm{gr}_\Sigma({\mathbb{A}})$-module structure of $\mathrm{gr}_\Sigma(M)$}\label{545}

For each $k\geq 0$, we have ${\mathbb{I}}^k=\mathrm{im}(S^k(\Sigma)\to{\mathbb{A}})\cdot{\mathbb{A}}$, therefore 
${\mathbb{I}}^k\cdot M=\mathrm{im}(S^k(\Sigma)\to{\mathbb{A}})\cdot{\mathbb{A}}\cdot M=\mathrm{im}(S^k(\Sigma)\to{\mathbb{A}})
\cdot M=S^k(\Sigma)\cdot M$. It follows that the decreasing filtration $M\supset\Sigma\cdot M\supset S^2(\Sigma)\cdot M
\supset\cdots$ can be identified with the decreasing filtration $M\supset{\mathbb{I}}\cdot M\supset{\mathbb{I}}^2\cdot M\supset\cdots$. 
The associated graded space $\mathrm{gr}_\Sigma(M)\simeq\mathrm{gr}_\Sigma({\mathfrak{L}}_1)$ (see Subsection \ref{constr:M:min}) is therefore
given by 
$$
\mathrm{gr}_\Sigma(M)=\oplus_{k\geq 0}\mathrm{gr}_\Sigma^k(M), \quad \mathrm{where}\quad 
\mathrm{gr}_\Sigma^k(M)={\mathbb{I}}^k\cdot M/{\mathbb{I}}^{k+1}\cdot M. 
$$
We define the $\Sigma$-degree on $\mathrm{gr}_\Sigma(M)$ by setting the degree of $\mathrm{gr}_\Sigma^k(M)$ to be $k+2$. 
The above expression of $\mathrm{gr}_\Sigma(M)$ shows that {\it this space is naturally equipped with a module structure
over $\mathrm{gr}_\Sigma({\mathbb{A}})$.}  One checks that it is compatible with the weight degrees, the $\Sigma$-degrees and the 
depth filtration on both sides. 

Recall that there is a sequence of algebra morphisms $S(\Sigma)\twoheadrightarrow{\mathbb{S}}\hookrightarrow
\mathrm{gr}_\Sigma({\mathbb{A}})$ (see Subsection \ref{section:S}). The $S(\Sigma)$-module structure of 
$\mathrm{gr}_\Sigma(M)$ (Subsection \ref{constr:M:min}, 1)) is induced by the above $\mathrm{gr}_\Sigma({\mathbb{A}})$-module 
structure though this morphism. This module structure also induces a {\it ${\mathbb{S}}$-module structure on $\mathrm{gr}_\Sigma(M)$.} 

\subsubsection{Expression of ${\mathbb{M}}^{min}(\Sigma)$} \label{546}
According to Subsection \ref{constr:M:min}, 3), one has 
$$
{\mathbb{M}}^{min}_0(\Sigma)=\mathrm{im}(\Lambda^2(\Sigma)\stackrel{\{,\}}{\to}\mathrm{gr}^0_\Sigma({\mathfrak{L}}_1))\subset 
\mathrm{gr}^0_\Sigma(M), \quad {\mathbb{M}}^{min}_i(\Sigma)={\mathbb{S}}[i]\cdot{\mathbb{M}}^{min}_0(\Sigma)\subset 
\mathrm{gr}^i_\Sigma(M)\ \mathrm{for}\ i\geq 1, 
$$
so 
$$
{\mathbb{M}}^{min}(\Sigma)={\mathbb{S}}\cdot{\mathbb{M}}^{min}_0(\Sigma)\subset\mathrm{gr}_\Sigma(M).
$$  

\section{Presentation of some commutative rings}\label{Section:comm:alg}

In this section, we compute the presentation of some of the commutative rings arising in the discussion of 
Subsection \ref{subsection:comm:alg}. 

\subsection{Structure of $\mathbb{A}$}
Recall that the ring $\mathbb{A}$ is graded by the weight degree. For $k\geq 0$, the elements $\tilde\sigma_k$ defined by (\ref{def:sigmai}) 
belong to $\C[A,B,A',B']^{sym}$; one checks that they also belong to the subalgebra $\mathbb{A}\subset\C[A,B,A',B']^{as}$. 

\begin{proposition} \label{prop:présentation:AA} The graded ring $\mathbb{A}$ has the following presentation: generators are 
$\tilde\sigma_i$, $i=2,3,4,5,6$; the only relation is 
$$
(\tilde\sigma_5)^2-{25\over18}\tilde\sigma_2\tilde\sigma_3\tilde\sigma_5+\Big({275\over108}
\tilde\sigma_4-{25\over162}(\tilde\sigma_2)^2\Big)\cdot(\tilde\sigma_3)^2+\Big(\tilde\sigma_4-{1\over4}(\tilde\sigma_2)^2\Big)\cdot\Big(-{125\over432}(\tilde\sigma_2)^3+{175\over72}\tilde\sigma_2\tilde\sigma_4-{25\over6}
\tilde\sigma_6\Big)=0; 
$$
the degree of each $\tilde\sigma_i$ is $i$. 
\end{proposition}

\begin{proof} The Hilbert series of $\mathbb{A}$ can be computed by Molien's theorem (\cite{M,St});  
one obtains
\begin{equation} \label{eqn:P:AA}
P_{\mathbb{A}}(t)={{1+t^5}\over{(1-t^2)(1-t^3)(1-t^4)(1-t^6)}}.\end{equation} 
Moreover, it can be checked through Magma (\cite{Mgm}) that $\mathbb{A}$ is generated by the elements $\tilde\sigma_i$, $i=2,3,4,5,6$;
that the $\tilde\sigma_i$, $i=2,3,4,6$ generate a free polynomial ring, and that $\tilde\sigma_5$ is quadratic over this ring. 
The quadratic relation satisfied by $\tilde\sigma_5$ can be determined though a Maple computation (\cite{Mpl}); this the relation given in 
the statement of the proposition.\end{proof}

\subsection{Structure of $\mathrm{gr}_\Sigma^0(\mathbb{A})$} \label{titi}

\begin{lemma}\label{lemma:restriction:géns:idéal}
If $i$ is odd $\geq 3$, then the element $\tilde\sigma_i\in\mathbb{A}$ belongs to the ideal $(\tilde\sigma_3,\tilde\sigma_5)$
of $\mathbb{A}$ generated by $\tilde\sigma_3$ and $\tilde\sigma_5$. 
\end{lemma} 

\begin{proof} Recall that $\mathbb{A}$ is generated by $\tilde\sigma_i$, $i=2,\ldots,6$. If $i$ is an integer $\geq 0$, the degree $i$
part of $\mathbb{A}$ is linearly spanned by the monomials $(\tilde\sigma_2)^{i_2}\cdots(\tilde\sigma_6)^{i_6}$, where
$2i_2+3i_3+\cdots+6i_6=i$. It follows that the odd degree part of $\mathbb{A}$ is contained in its ideal $(\tilde\sigma_3,\tilde\sigma_5)$. 
If $i$ is odd, then $\tilde\sigma_i$ has degree $i$, so it belongs to this ideal. \end{proof}

Recall that the ring $\mathrm{gr}_\Sigma^0(\mathbb{A})$ is graded by the weight degree.

\begin{proposition} \label{prop43} For $i=2,4,6$, let $\dot\sigma_i$ be the image of $\tilde\sigma_i$ under $\mathbb{A}\to\mathbb{A}/\mathbb{I}=\mathrm{gr}_\Sigma^0(\mathbb{A})$. The ring $\mathrm{gr}_\Sigma^0(\mathbb{A})$ 
has the following presentation: generators $\dot\sigma_i$, with $i=2,4,6$; relation 
$$
\Big(\dot\sigma_4-{1\over4}(\dot\sigma_2)^2\Big)\cdot\Big(-{125\over432}(\dot\sigma_2)^3+{175\over72}\dot\sigma_2\dot\sigma_4-{25\over6}
\dot\sigma_6\Big)=0;
$$
the degree of each $\dot\sigma_i$ is $i$. 
\end{proposition}

\begin{proof} Lemma \ref{lemma:restriction:géns:idéal} implies that the ideals $\mathbb{I}=(\tilde\sigma_3,\tilde\sigma_5,
\tilde\sigma_7,\ldots)$ and $(\tilde\sigma_3,\tilde\sigma_5)$ of $\mathbb{A}$ coincide. As $\mathrm{gr}_\Sigma^0(\mathbb{A})
=\mathbb{A}/\mathbb{I}$, it follows that $\mathrm{gr}_\Sigma^0(\mathbb{A})=\mathbb{A}/(\tilde\sigma_3,\tilde\sigma_5)$. Combining 
this result with the presentation of $\mathbb{A}$ obtained in Proposition \ref{prop:présentation:AA}, one obtains the announced presentation 
of $\mathrm{gr}_\Sigma^0(\mathbb{A})$. \end{proof}

\subsection{Structure of $\mathrm{gr}_\Sigma(\mathbb{A})$}

Recall that the ring $\mathrm{gr}_\Sigma^0(\mathbb{A})$ is bigraded by (weight degree, $\Sigma$-degree).

Let $\mathrm{gr}_\Sigma^0(\mathbb{A})[\mathbf{x}_3,\mathbf{x}_5]$ be the polynomial algebra in two variables $\mathbf{x}_3,\mathbf{x}_5$ 
with coefficients in $\mathrm{gr}_\Sigma^0(\mathbb{A})$. It is equipped with a bigrading, such that $\dot\sigma_i$ has bidegree $(i,0)$ 
($i=2,4,6$) and $\mathbf{x}_i$ has bidegree $(i,1)$ ($i=3,5$). 

\begin{theorem} (Structure of $\mathrm{gr}_\Sigma(\mathbb{A})$) \label{StructureDAnneaux}
There is a unique isomorphism $\mathrm{gr}_\Sigma(\mathbb{A})\stackrel{\sim}{\to}\mathrm{gr}^0_\Sigma(\mathbb{A})
[\mathbf{x}_3,\mathbf{x}_5]$ of bigraded algebras, whose inverse can be characterized as follows: its restriction to 
$\mathrm{gr}^0_\Sigma(\mathbb{A})$ coincides with the canonical inclusion 
$\mathrm{gr}^0_\Sigma(\mathbb{A})\hookrightarrow\mathrm{gr}_\Sigma(\mathbb{A})$, and $\mathbf{x}_3,\mathbf{x}_5$ map to the 
classes $\dot\sigma_3,\dot\sigma_5$ of $\tilde\sigma_3,\tilde\sigma_5$ in $\mathbb{I}/\mathbb{I}^2$.  
\end{theorem}

\begin{proof} There is a unique morphism $\phi:\mathrm{gr}^0_\Sigma(\mathbb{A})[\mathbf{x}_3,\mathbf{x}_5]\to
\mathrm{gr}_\Sigma(\mathbb{A})$ of bigraded algebras, whose restriction to $\mathrm{gr}^0_\Sigma(\mathbb{A})$ is 
$\mathrm{gr}^0_\Sigma(\mathbb{A})\hookrightarrow\mathrm{gr}_\Sigma(\mathbb{A})$ and taking $\mathbf{x}_3,\mathbf{x}_5$ to 
$\dot\sigma_3,\dot\sigma_5$. 

The algebra $\mathrm{gr}_\Sigma(\mathbb{A})$ is known to be generated by its parts of $\Sigma$-degree 0 and 1, namely 
$\mathrm{gr}^0_\Sigma(\mathbb{A})$ and $\mathrm{gr}^1_\Sigma(\mathbb{A})$. According to Lemma \ref{lemma:restriction:géns:idéal}, 
the ideals $\mathbb{I}$ and $(\tilde\sigma_3,\tilde\sigma_5)=\mathbb{A}\cdot\tilde\sigma_3+\mathbb{A}\cdot\tilde\sigma_5$ of 
$\mathbb{A}$ coincide. Therefore 
$\mathrm{gr}^1_\Sigma(\mathbb{A})=\mathbb{I}/\mathbb{I}^2=\mathrm{im}(\mathbb{A}\cdot\tilde\sigma_3+\mathbb{A}\cdot\tilde\sigma_5
\to\mathbb{I}/\mathbb{I}^2)=\mathrm{gr}^0_\Sigma(\mathbb{A})\cdot\dot\sigma_3+\mathrm{gr}^0_\Sigma(\mathbb{A})\cdot\dot\sigma_5$. 
It follows that $\mathrm{gr}_\Sigma(\mathbb{A})$ is generated by $\mathrm{gr}_\Sigma^0(\mathbb{A}),\dot\sigma_3$ and $\dot\sigma_5$. 
The restriction of $\phi$ to the part of $\Sigma$-degree 0 is the identity of $\mathrm{gr}_\Sigma^0(\mathbb{A})$, therefore $\mathrm{im}(\phi)\supset\mathrm{gr}_\Sigma^0(\mathbb{A})$. The image of $\phi$ also contains the images of $\mathbf{x}_3$ and $\mathbf{x}_5$, 
which are $\dot\sigma_3$ and $\dot\sigma_5$. The image of $\phi$ therefore contains a generating part of $\mathrm{gr}_\Sigma(\mathbb{A})$, 
and so coincides with it. It follows that {\it $\phi$ is onto.} 

The algebra $\mathrm{gr}_\Sigma(\mathbb{A})$ is constructed out of $\mathbb{A}$ using a filtration which is compatible with the weight degree. 
It follows that the Hilbert series $P_{\mathrm{gr}_\Sigma(\mathbb{A})}(t)$ of $\mathrm{gr}_\Sigma(\mathbb{A})$ with respect to the weight 
degree coincides with that of $\mathbb{A}$ given by (\ref{eqn:P:AA}), so 
$$
P_{\mathrm{gr}_\Sigma(\mathbb{A})}(t)={{1+t^5}\over{(1-t^2)(1-t^3)(1-t^4)(1-t^6)}}. 
$$

Let $\C[\xi_2,\xi_4,\xi_6]/(\xi_4\cdot\xi_6)$ be the graded commutative algebra with generators $\xi_2,\xi_4,\xi_6$ and relation 
$\xi_4\cdot\xi_6=0$, with degree given by deg$(\xi_i)=i$. As basis of $\C[\xi_2,\xi_4,\xi_6]/(\xi_4\cdot\xi_6)$ is then given by the 
union of families $\xi_2^a\xi_4^b$ ($a\geq 0,b\geq 0$) and $\xi_2^c\xi_6^d$ ($c\geq 0,d>0$). Therefore the Hilbert series of 
 $\C[\xi_2,\xi_4,\xi_6]/(\xi_4\cdot\xi_6)$ is 
$$
P_{\C[\xi_2,\xi_4,\xi_6]/(\xi_4\cdot\xi_6)}(t)={1\over{(1-t^2)(1-t^4)}}+{1\over{(1-t^2)(1-t^6)}}-{1\over{1-t^2}}=
{{1-t^{10}}\over{(1-t^2)(1-t^4)(1-t^6)}}. 
$$
It follows from Proposition \ref{prop43} that there is a unique isomorphism 
$\mathrm{gr}_\Sigma^0(\mathbb{A})\simeq\C[\xi_2,\xi_4,\xi_6]/(\xi_4\cdot\xi_6)$ of graded algebras, given by 
$\xi_2\mapsto\dot\sigma_2$, $\xi_4\mapsto\dot\sigma_4-{1\over4}(\dot\sigma_2)^2$, $\xi_6\mapsto
-{125\over432}(\dot\sigma_2)^3+{175\over72}\dot\sigma_2\tilde\sigma_4-{25\over6}\dot\sigma_6$. 
Therefore the Hilbert series of $\mathrm{gr}_\Sigma^0(\mathbb{A})$ is 
$$
P_{\mathrm{gr}_\Sigma^0(\mathbb{A})}(t)=P_{\C[\xi_2,\xi_4,\xi_6]/(\xi_4\cdot\xi_6)}(t)={{1-t^{10}}\over{(1-t^2)(1-t^4)(1-t^6)}}. 
$$
The Hilbert series of the polynomial algebra $\C[x_3,x_5]$ is $P_{\C[x_3,x_5]}(t)={1\over{(1-t^3)(1-t^5)}}$. 
As $\mathrm{gr}_\Sigma^0(\mathbb{A})[\mathbf{x}_3,\mathbf{x}_5]$ is isomorphic to the tensor product of the graded algebras 
$\mathrm{gr}_\Sigma^0(\mathbb{A})$ and $\C[\mathbf{x}_3,\mathbf{x}_5]$, its Hilbert series is given by 
$$
P_{\mathrm{gr}_\Sigma^0(\mathbb{A})[\mathbf{x}_3,\mathbf{x}_5]}(t)=P_{\mathrm{gr}_\Sigma^0(\mathbb{A})}(t)\cdot 
P_{\C[\mathbf{x}_3,\mathbf{x}_5]}(t)={{1+t^5}\over{(1-t^2)(1-t^3)(1-t^4)(1-t^6)}}, 
$$
therefore $P_{\mathrm{gr}_\Sigma^0(\mathbb{A})[\mathbf{x}_3,\mathbf{x}_5]}(t)=P_{\mathrm{gr}_\Sigma(\mathbb{A})}(t)$: {\it the 
Hilbert series of $\mathrm{gr}_\Sigma^0(\mathbb{A})[\mathbf{x}_3,\mathbf{x}_5]$ and $\mathrm{gr}_\Sigma(\mathbb{A})$ with respect to 
the weight degree coincide.}  This result, combined with the fact that $\phi$ is onto, implies that $\phi$ is an isomorphism of graded vector 
spaces (for the weight degree). As $\phi$ is also compatible with the bigrading (weight degree, $\Sigma$-degree), it is an isomorphism of 
bigraded vector spaces. \end{proof}

\subsection{Results on $\mathbb{S}$} \label{results:S}

Lemma \ref{lemma:restriction:géns:idéal} says that for any odd $i\geq 3$, there exist elements 
$P_{i3},P_{i5}$ in $\mathbb{A}$, of respective weight degrees $i-3$, $i-5$, such that 
\begin{equation}\label{id:sigma:i}
\tilde\sigma_i=P_{i3}\cdot\tilde\sigma_3+P_{i5}\cdot\tilde\sigma_5 
\end{equation}
(equation in $\mathbb{A}$). 
We choose such a pair $(P_{i3},P_{i5})$ for each $i$ odd $\geq 3$. 

For each odd $i\geq 3$, let $\dot\sigma_i$ be the image of $\tilde\sigma_i\in{\mathbb{I}}$
in ${\mathbb{I}}/{\mathbb{I}}^2=\mathrm{gr}_\Sigma^1({\mathbb{A}})$, and let $\dot P_{i3},\dot P_{i5}$
be the images of $P_{i3},P_{i5}$ in ${\mathbb{A}}/{\mathbb{I}}=\mathrm{gr}_\Sigma^0({\mathbb{A}})$. Then (\ref{id:sigma:i})
implies
$$
\dot\sigma_i=\dot P_{i3}\cdot\dot\sigma_3+\dot P_{i5}\cdot\dot\sigma_5 
$$
(equation in $\mathrm{gr}_\Sigma^1({\mathbb{A}})$). Then $\mathbb{S}$ is the subring of $\mathrm{gr}_\Sigma({\mathbb{A}})$
generated by the $\dot\sigma_i$ for all odd $i\geq 3$. Under the isomorphism $\mathrm{gr}_\Sigma({\mathbb{A}})\simeq
\mathrm{gr}_\Sigma^0({\mathbb{A}})[\mathbf{x}_3,\mathbf{x}_5]$ (Theorem \ref{StructureDAnneaux}), {\it $\mathbb{S}$ is 
isomorphic to the subring of $\mathrm{gr}_\Sigma^0({\mathbb{A}})[\mathbf{x}_3,\mathbf{x}_5]$ generated by 
$\mathbf{x}_3,\mathbf{x}_5$ and the $\dot P_{i3}\cdot\mathbf{x}_3+\dot P_{i5}\cdot\mathbf{x}_5$ ($i\geq 7$).} 

\section{Computation of the lowest degree part ${\mathbb M}^{min}_0(\Sigma)$ of the lower bound $\Sigma$-structure}
\label{Section:comp:Mminlowest}

Recall that ${\mathbb M}^{min}_0(\Sigma)$ is a subspace of $\mathrm{gr}_\Sigma^0(M)=M/{\mathbb I}\cdot M$, graded for the weight 
grading. The space $\mathrm{gr}_\Sigma^0(M)$ is a module over the ring $\mathrm{gr}_\Sigma^0({\mathbb A})={\mathbb A}/{\mathbb I}$, 
which contains a particular element $\dot\tau_{35}$ (see (\ref{tau:ij}); we denote by $x\mapsto\dot x$ the projection map 
$M\to\mathrm{gr}_\Sigma^0(M)$). In this section, we show the equality, inside $\mathrm{gr}_\Sigma^0(M)$, of 
${\mathbb M}^{min}_0(\Sigma)$ and  $\mathrm{gr}_\Sigma^0({\mathbb A})\cdot\dot\tau_{35}$, the cyclic 
$\mathrm{gr}_\Sigma^0({\mathbb A})$-submodule generated by $\dot\tau_{35}$. 

We first prove the inclusion ${\mathbb M}^{min}_0(\Sigma)\subset\mathrm{gr}_\Sigma^0({\mathbb A})\cdot\dot\tau_{35}$
(Subsection \ref{subsect:incl:M:cycl}). We then compute the Hilbert-Poincar\'e series of $\mathrm{gr}_\Sigma^0({\mathbb A})
\cdot\dot\tau_{35}$ relative to the weight grading (Subsection \ref{subsect:hilb:cyclic}) and a lower bound for the series of 
${\mathbb M}^{min}_0(\Sigma)$ (Subsection \ref{subsect:lower:bound:series}). This enables us to prove the desired equality
and to compute explicitly the map $\Lambda^2(\Sigma)\to{\mathbb M}^{min}_0(\Sigma)$ (Subsection \ref{subsect:conclusions}). 

\subsection{The inclusion ${\mathbb M}^{min}_0(\Sigma)\subset\mathrm{gr}_\Sigma^0({\mathbb A})\cdot\dot\tau_{35}$}\label{C2}
\label{subsect:incl:M:cycl} The structure of the proof of this inclusion is as follows. Subsections (a), (b), (c) contain preparatory material for the 
definition of a family of polynomials $\mathbf{P}_{ij}$ and the proof of its property (Aux$_{ij}$) in Subsection (d). Subsection (e) contains some 
auxiliary results which enable us to prove in Subsection (f) that $\mathbf{P}_{ij}$ satisfies a property (Cond$_{ij}$). This implies the desired
inclusion (Subsection (g)). 

\medskip {\it(a) Divisibility of polynomials.} For $i$ odd $\geq 3$, we defined a polynomial $\lambda_i\in\C[A,B,A',B']^{as}$
(see (\ref{lambda:i})). One has $(\lambda_3)_{|B'=0}=3A'B(A'+2A+B)$, and 
for any $i$ odd $\geq 3$, $(\lambda_i)_{|B'=0}=(A+A'+B)^i-(A+A')^i-(A+B)^i+A^i$. This polynomial vanishes
under the substitutions $A'=0$ and $B=0$, and also (as $i$ is odd) under $A'=-2A-B$. It follows that 
{\it for any odd $i\geq 3$, $(\lambda_3)_{|B'=0}$ divides $(\lambda_i)_{|B'=0}$ in $\C[A,B,A']$.}
Degree considerations also imply that {\it the polynomial $(\lambda_i)_{|B'=0}/(\lambda_3)_{|B'=0}$ of $\C[A,B,A']$ in even, 
i.e., invariant under $(A,B,A')\mapsto(-A,-B,-A')$.} 

\medskip {\it(b) Decomposition of algebras.} 
Set $\sigma:=A^2+AB+B^2$, $p:=AB$; these are elements of $\C[A,B]$. Then  
$$
\C[A,B]_{\mathrm{invariant\ under\ }(A,B)\mapsto(-A,-B)\mathrm{\ and\ }(A,B)\mapsto(B,A)}=\C[\sigma,p]. 
$$
Let $\pi:=p^2\sigma+p^3$. Then $p$ is integral over $\C[\sigma,\pi]$, with minimal polynomial $X^3+\sigma X^2-\pi$. 
It follows that $\C[\sigma,p]$ is a free module 
over its subring $\C[\sigma,\pi]$, with 
\begin{equation} \label{decomposition}
\C[\sigma,p]=\C[\sigma,\pi]\oplus\C[\sigma,\pi]\cdot p\oplus\C[\sigma,\pi]\cdot p^2. 
\end{equation}

Now, we have $(\tilde\sigma_2)_{|A'=B'=0}=2(A^2+(A+B)^2+B^2)$, 
$(\tilde\sigma_6)_{|A'=B'=0}=2(A^6+(A+B)^6+B^6)$, therefore
\begin{equation} \label{expression:sigma:2:6}
(\tilde\sigma_2)_{|A'=B'=0}=4\sigma, \quad 
(\tilde\sigma_6)_{|A'=B'=0}=6\pi+4\sigma^3, 
\end{equation}
therefore 
$$
\C[(\tilde\sigma_2)_{|A'=B'=0},(\tilde\sigma_6)_{|A'=B'=0}]=\C[\sigma,\pi]\subset\C[\sigma,p].  
$$

\medskip {\it(c) Decomposition of generating series of polynomials.} Let $t$ be a formal variable, then one computes
\begin{equation} \label{égalité1}
\sum_{i\mathrm{\ odd\ }\geq 3}(\lambda_i)_{|B'=0}t^i
={{Num(A,B,A',t)}\over{(1-((A+B+A')t)^2)(1-((A+A')t)^2)(1-((A+B)t)^2)(1-(At)^2)}}
\end{equation}
where $Num(A,B,A',t)\in\C[A,B,A',t]$. One checks that $Num(A,B,A',t)$ is divisible by $\lambda_3(A,B,A',0)$ and one sets 
$Num_1(A,B,A',t):=Num(A,B,A',t)/\lambda_3(A,B,A',0)$. One computes 
\begin{align*}
Num_1(A,B,A',t) & ={1\over 3}t^3(3-(2A^2+B^2+A^{\prime 2}+2AB+BA'+2AA')t^2 \\ & 
-A(A^3+2A^2B+AB^2+2A^2A'+AA^{\prime 2}+3ABA'+B^2A'+BA^{\prime 2})t^4), 
\end{align*}
so 
$$
\sum_{i\mathrm{\ odd\ }\geq 3}{(\lambda_i)_{|B'=0}\over (\lambda_3)_{|B'=0}}\cdot t^i
={{Num_1(A,B,A',t)}\over{(1-((A+B+A')t)^2)(1-((A+A')t)^2)(1-((A+B)t)^2)(1-(At)^2)}}
$$
which upon specialization $A'=0$ and symmetrization in $(A,B)$ gives 
$$
\sum_{i\mathrm{\ odd\ }\geq 3}\Big(\Big({(\lambda_i)_{|B'=0}\over (\lambda_3)_{|B'=0}}\Big)_{A'=0}+(A\leftrightarrow B)\Big)\cdot t^i
={{Num_2(A,B,t)}\over{(1-((A+B)t)^2)^2(1-(At)^2)^2(1-(Bt)^2)^2}},  
$$
where $Num_2(A,B,t):=Num_1(A,B,0,t)(1-(Bt)^2)^2+(A\leftrightarrow B)$. 

Set 
\begin{equation} \label{déf:D}
D(\sigma,\pi,t):=(1-2\sigma t^2+\sigma^2 t^4-\pi t^6)^2, 
\end{equation}
then  
$$
(1-((A+B)t)^2)^2(1-(At)^2)^2(1-(Bt)^2)^2=D(\sigma,\pi,t), 
$$ 
so the l.h.s. belongs to $\C[\sigma,\pi][t]\subset\C[A,B][t]$ (and even to $1+t\C[\sigma,\pi][t]$). 
It follows from (a) that for $i$ odd $\geq 3$, $\mathrm{sym}_{A,B}((\lambda_i)_{|B'=0}/(\lambda_3)_{|B'=0})_{|A'=0}$
belongs to $$\C[A,B]_{\mathrm{invariant\ under\ }(A,B)\mapsto(-A,-B)\mathrm{\ and\ }(A,B)\mapsto(B,A)}=\C[\sigma,p].$$ It follows that 
$Num_2(A,B,t)$ belongs to $\C[\sigma,p]$. The decomposition of $Num_2(A,B,t)$ along (\ref{decomposition}) can be 
computed as follows: 
$$
Num_2(A,B,t)=X(\sigma,\pi,t)+Y(\sigma,\pi,t)\cdot p+0\cdot p^2, 
$$
where 
\begin{equation} \label{déf:X}
X(\sigma,\pi,t):=2t^3-3\sigma t^5+(4/3)\sigma^2t^7+(2\pi-(1/3)\sigma^3)t^9-(1/3)\pi\sigma t^{11}, \end{equation}
\begin{equation} \label{déf:Y}
Y(\sigma,\pi,t):=(5/3)t^5-2\sigma t^7+(1/3)\sigma^2t^9+(1/3)\pi t^{11}.
\end{equation}
It follows that {\it the decomposition according to (\ref{decomposition}) of the polynomials 
$\mathrm{sym}_{A,B}({(\lambda_i)_{|B'=0}\over (\lambda_3)_{|B'=0}})_{|A'=0}$ of $\C[\sigma,p]$ can be derived from the following identity 
\begin{equation} \label{décomposition:rapports:lambda}
\sum_{i\mathrm{\ odd\ }\geq 3}\Big(\Big({(\lambda_i)_{|B'=0}\over (\lambda_3)_{|B'=0}}\Big)_{A'=0}+(A\leftrightarrow B)\Big)\cdot t^i
={X(\sigma,\pi,t)\over D(\sigma,\pi,t)}+{Y(\sigma,\pi,t)\over D(\sigma,\pi,t)}\cdot p+0\cdot p^2, 
\end{equation}
where $X(\sigma,\pi,t),Y(\sigma,\pi,t),D(\sigma,\pi,t)$ are given by (\ref{déf:X}), (\ref{déf:Y}), (\ref{déf:D}).}

If $i$ is odd $\geq 3$, then $\tilde\sigma_i(A,B,0,0)=0$. Define $\tilde\sigma_i^{lin}(A,B,A',B')$ to be the part of 
$\tilde\sigma_i(A,B,A',B')$ which is linear in $A',B'$, so 
\begin{equation}\label{def:sigma:lin}
\tilde\sigma_i^{lin}(A,B,A',B')
={\partial \tilde\sigma_i\over\partial A'}(A,B,0,0)\cdot A'+{\partial \tilde\sigma_i\over\partial B'}(A,B,0,0)\cdot B'.
\end{equation} 

In Subsection \ref{results:S}, we introduced elements $P_{i3},P_{i5}$ of $\mathbb{A}\subset\C[A,B,A',B']$ satisfying identity 
(\ref{id:sigma:i}). 
One has $\tilde\sigma_i(A,B,A',B')=P_{i3}(A,B,A',B')\tilde\sigma_3(A,B,A',B')+P_{i5}(A,B,A',B')\tilde\sigma_5(A,B,A',B')$, so 
$$
\tilde\sigma_i^{lin}(A,B,A',B')=P_{i3}(A,B,0,0)\tilde\sigma_3^{lin}(A,B,A',B')+P_{i5}(A,B,0,0)\tilde\sigma_5^{lin}(A,B,A',B'), 
\quad (\mathrm{Eq}_i)
$$
One computes
$$
\tilde\sigma_i^{lin}(A,B,A',B')=i\big((A+B)^{i-1}-A^{i-1}\big)\cdot A'+i\big((A+B)^{i-1}-B^{i-1}\big)\cdot B'. 
$$
The combination $\sum_{i\mathrm{\ odd\ }\geq 3}t^i\cdot({\mathrm{Eq}_i})$ yields
\begin{align*}
& tA'\big({1+(t(A+B))^2\over(1-(t(A+B))^2)^2}-{1+(tA)^2\over(1-(tA)^2)^2}\big)+((A,A')\leftrightarrow(B,B')) 
\\ & =A'\Big(3\big((A+B)^2 - A^2\big)P_3(A,B,t) 
+ 5\big( (A+B)^4 - A^4 \big)P_5(A,B,t)\Big) +((A,A')\leftrightarrow(B,B')), 
\end{align*}
where $P_3(A,B,t):=\sum_{i\mathrm{\ odd\ }\geq 3}P_{i3}(A,B,0,0)t^i$, $P_5(A,B,t):=\sum_{i\mathrm{\ odd\ }\geq 3}P_{i5}(A,B,0,0)t^i$. 

Taking coefficients of $A'$ and $B'$ in this equation, we obtain the following system
$$3\big((A+B)^2 - A^2\big)P_3(A,B,t) 
+ 5\big( (A+B)^4 - A^4 \big)P_5(A,B,t) =t\big({1+(t(A+B))^2\over(1-(t(A+B))^2)^2}-{1+(tA)^2\over(1-(tA)^2)^2}\big), $$
$$
3\big((A+B)^2 - B^2\big)P_3(A,B,t) 
+ 5\big( (A+B)^4 - B^4 \big)P_5(A,B,t) =t\big({1+(t(A+B))^2\over(1-(t(A+B))^2)^2}-{1+(tB)^2\over(1-(tB)^2)^2}\big),
$$
leading to 
\begin{align*}
& 15 AB(A^2-B^2)(2A^2+5AB+2B^2)P_5(A,B,t)
\\ &
=  3(A^2-B^2)t{1+(t(A+B))^2\over(1-(t(A+B))^2)^2}+3(B^2-(A+B)^2)t{1+(tA)^2\over(1-(tA)^2)^2}+3((A+B)^2-A^2)t
{1+(tB)^2\over(1-(tB)^2)^2}, 
\end{align*}
which gives, after computation, 
$$
P_5(A,B,t)={t^5\over 5}{5-6\sigma t^2+\sigma^2 t^4+\pi t^6\over (1-2\sigma t^2+\sigma^2 t^4-\pi t^6)^2}={3\over 5}
{Y(\sigma,\pi,t)\over D(\sigma,\pi,t)}, 
$$
therefore 
\begin{equation} \label{expression:P5i}
\sum_{i\mathrm{\ odd\ }\geq 3}P_{i5}(A,B,0,0)t^i={3\over 5}
{Y(\sigma,\pi,t)\over D(\sigma,\pi,t)}. 
\end{equation}

\medskip {\it(d) A family $(\mathbf{P}_{ij})_{i,j}$ of polynomials and their properties $(\mathrm{Aux}_{ij})$.} 
Let $\mathbf{x}_2,\mathbf{x}_6$ be free commutative variables of weight degrees 2,6. For $i,j$ odd integers $\geq 3$, 
define the polynomial $\mathbf{P}_{ij}(\mathbf{x}_2,\mathbf{x}_6)$ of $\C[\mathbf{x}_2,\mathbf{x}_6]$ by 
\begin{align} \label{def:Pij}
\mathbf{P}_{ij}(\mathbf{x}_2,\mathbf{x}_6):={3\over10}\Big( & {X\over D}\Big(\sigma(\mathbf{x}_2,\mathbf{x}_6),\pi(\mathbf{x}_2,\mathbf{x}_6),t\Big)
{Y\over D}\Big(\sigma(\mathbf{x}_2,\mathbf{x}_6),\pi(\mathbf{x}_2,\mathbf{x}_6),u\Big) \\ & \nonumber 
-{Y\over D}\Big(\sigma(\mathbf{x}_2,\mathbf{x}_6),\pi(\mathbf{x}_2,\mathbf{x}_6),t\Big)
{X\over D}\Big(\sigma(\mathbf{x}_2,\mathbf{x}_6),\pi(\mathbf{x}_2,\mathbf{x}_6),u\Big)
\Big|t^iu^j\Big), 
\end{align}
where 
\begin{equation}\label{def:sigma:pi}
\sigma(\mathbf{x}_2,\mathbf{x}_6):={1\over 4}\mathbf{x}_2, \quad 
\pi(\mathbf{x}_2,\mathbf{x}_6):={1\over 6}\mathbf{x}_6-{1\over 96}\mathbf{x}_2^3,
\end{equation} 
and where $(\cdots|t^iu^j)$ means ``coefficient of $t^iu^j$ in $\cdots$''. 

We have 
$$
\mathbf{P}_{ij}((\tilde\sigma_2)_{|A'=B'=0},(\tilde\sigma_6)_{|A'=B'=0})=\mathbf{P}_{ij}(4\sigma,6\pi+4\sigma^3)
$$
by (\ref{expression:sigma:2:6}). Using $(\sigma(4\sigma,6\pi+4\sigma^3),\pi(4\sigma,6\pi+4\sigma^3))=(\sigma,\pi)$, 
and separating variables, one gets  
$$
\mathbf{P}_{ij}(4\sigma,3\pi+\sigma^3)=
{3\over10}\Big( \Big( {X\over D}(\sigma,\pi,t)\Big|t^i\Big)
\Big({Y\over D}(\sigma,\pi,t)\Big|t^j\Big)  -\Big({Y\over D}(\sigma,\pi,t)\Big|t^i\Big)
\Big({X\over D}(\sigma,\pi,t)\Big|t^j\Big)\Big),
$$
which is equal to 
$$
{3\over10}\Big( {Y\over D}(\sigma,\pi,t)\Big|t^j\Big)
\Big({X\over D}(\sigma,\pi,t)+p{Y\over D}(\sigma,\pi,t)\Big|t^i\Big)
-{3\over10}\Big( {Y\over D}(\sigma,\pi,t)\Big|t^i\Big)
\Big({X\over D}(\sigma,\pi,t)+p{Y\over D}(\sigma,\pi,t)\Big|t^j\Big),
$$
which by (\ref{décomposition:rapports:lambda}) and (\ref{expression:P5i}) is equal to 
$$
-{1\over2}\Big(
(P_{i5})_{|A'=B'=0}\Big(\Big({(\lambda_j)_{|B'=0}\over (\lambda_3)_{|B'=0}}\Big)_{A'=0}+(A\leftrightarrow B)\Big)-(i\leftrightarrow j)
\Big). 
$$
All this implies: 

\begin{lemma} \label{PropB} The polynomial $\mathbf{P}_{ij}(\mathbf{x}_2,\mathbf{x}_6)$ defined by (\ref{def:Pij}) satisfies
\begin{equation*}
\mathbf{P}_{ij}((\tilde\sigma_2)_{|A'=B'=0},(\tilde\sigma_6)_{|A'=B'=0})=-{1\over2}\Big(
(P_{i5})_{|A'=B'=0}\Big(\Big({(\lambda_j)_{|B'=0}\over (\lambda_3)_{|B'=0}}\Big)_{A'=0}+(A\leftrightarrow B)\Big)-(i\leftrightarrow j)\Big). 
\quad (\mathrm{Aux}_{ij})
\end{equation*}
\end{lemma}

\medskip {\it(e) Characterization of images of linear maps between spaces of polynomials.}

Set 
\begin{align*}
\C[A,B,A',B']^{as\times as}:=\{P\in\C[A,B,A',B']|&  P\mathrm{\ is\ antisymmetric\ under\ both\ exchanges\ }\\ & 
(A,B)\leftrightarrow(A',B')\mathrm{\ and\ }(A,A')\leftrightarrow(B,B').\}
\end{align*}

Let $r:\C[A,B,A',B']^{as\times as}\to\C[A,B,A']$ be the linear map given by $P\mapsto P_{|B'=0}$. 

\begin{lemma} \label{CaractérisationIm(r)}
Let $\Pi$ be an element of $\C[A,B,A']$. Then $\Pi$ belongs to $\mathrm{im}(r)$ iff: 
\begin{enumerate}
\item the polynomial $\Pi(0,B,A')$ is symmetric in $(B,A')$; 
\item the polynomial $\Pi(A,0,A')$ is antisymmetric in $(A,A')$; 
\item the polynomial $\Pi(A,B,0)$ is antisymmetric in $(A,B)$.  
\end{enumerate}
\end{lemma}

\begin{proof} Let $\Pi$ belong to $\mathrm{im}(r)$. Then $\Pi=r(F)$ for some $F\in\C[A,B,A',B']^{as\times as}$. 
Then $\Pi(0,B,A')=F(0,B,A',0)=-F(A',0,0,B)=F(0,A',B,0)=\Pi(0,A',B)$, so $\Pi$ satisfies (1). One proves in the same way that 
its satisfies (2) and (3). 

Let $\Pi$ be an element of $\C[A,B,A']$ satisfying (1), (2) and (3). Set 
\begin{align*}
& F(A,B,A',B'):= \Pi(A,B,A')-\Pi(A',B',A)-\Pi(B,A,B')+\Pi(B',A',B)\\ & +\Pi(A',0,A)-\Pi(B',0,B)+\Pi(B,A,0)-\Pi(B',A',0)-\Pi(0,A',B)+\Pi(0,B',A)
\\ & +\Pi(0,0,B)-\Pi(0,0,A)-\Pi(0,0,B')+\Pi(0,0,A'). 
\end{align*}
$F(A,B,A',B')$ is the sum of fourteen terms, which we denote $T_1,\ldots,T_{14}$. Then $T_1+\cdots+T_4$ and 
$T_{11}+\cdots+T_{14}$ are signed averages over 
the group $S_2\times S_2$, therefore they belong to $\C[A,B,A',B']^{as\times as}$. The fact that $\Pi$ satisfies (2) (resp., (3), (1)) implies that 
$T_5+T_6$ (resp., $T_7+T_8$, $T_9+T_{10}$) belongs to $\C[A,B,A',B']^{as\times as}$. It follows that $F(A,B,A',B')$ belongs to 
$\C[A,B,A',B']^{as\times as}$.

The image of $F$ by $r$ is
$$
r(F)(A,B,A')=F(A,B,A',0)=\Pi(A,B,A') -\Pi(0,A',0)-\Pi(0,0,0)+\Pi(0,0,A').  
$$
Since $\Pi$ satisfies (1), $\Pi(0,A',0)=\Pi(0,0,A')$ and since it satisfies (2), $\Pi(0,0,0)=0$. Therefore $r(F)=\Pi$. 
\end{proof}

\begin{lemma} \label{PropositionPi}
Let $\Pi$ be an element of $\C[A,B,A']$ and $\mathbf{P}$ be an element of $\C[\mathbf{x}_2,\mathbf{x}_6]$. The following conditions 
on $(\Pi,\mathbf{P})$ are equivalent: 
\begin{description}
\item[{\rm (a)}] $(\Pi,\mathbf{P})$ satisfies the following three conditions: 
\begin{description}
\item[{\rm (a1)}] $\Pi(0,B,A')=\Pi(0,A',B)$; 
\item[{\rm (a2)}] $\Pi(A,B,0)+\Pi(B,A,0)=\Pi(A,0,B)+\Pi(B,0,A)$; 
\item[{\rm (a3)}] $\mathbf{P}((\tilde\sigma_2)_{|A'=B'=0},(\tilde\sigma_6)_{|A'=B'=0})={1\over 2}(\Pi(A,B,0)+\Pi(B,A,0))$ (equality in $\C[A,B]$).   
\end{description}
\item[{\rm (b)}] there exists an element $X$ of $\C[A,B,A',B']^{as\times as}$ such that $(\mathbf{P}(\tilde\sigma_2,\tilde\sigma_6)+X)_{|B'=0}=\Pi$. 
\end{description}
\end{lemma}

\begin{proof} Condition (b) is equivalent to the statement``$\Pi-\mathbf{P}(\tilde\sigma_2,\tilde\sigma_6)_{|B'=0}$ belongs to 
$\mathrm{im}(r)$''. According to Lemma \ref{CaractérisationIm(r)}, this is equivalent to the condition that 
$\Pi-\mathbf{P}(\tilde\sigma_2,\tilde\sigma_6)_{|B'=0}$ satisfies conditions (1), (2) and (3) from that lemma. As 
$(\tilde\sigma_2)_{|A=B'=0}$ and $(\tilde\sigma_6)_{|A=B'=0}$ are symmetric in $(A',B)$,   
so is $\mathbf{P}(\tilde\sigma_2,\tilde\sigma_6)_{|A=B'=0}$; the condition that $\Pi-\mathbf{P}(\tilde\sigma_2,\tilde\sigma_6)_{|B'=0}$ 
satisfies (1) is therefore equivalent to (a1). Besides, the condition that $\Pi-\mathbf{P}(\tilde\sigma_2,\tilde\sigma_6)_{|B'=0}$ satisfies 
(2) is equivalent to the antisymmetry in $(A,A')$ of the polynomial $\Pi(A,0,A')-\mathbf{P}(2(A^2+A^{\prime2}+(A+A')^2),
2(A^6+A^{\prime6}+(A+A')^6))$, in other words 
\begin{equation} \label{2'}
\mathbf{P}(2(A^2+A^{\prime2}+(A+A')^2),2(A^6+A^{\prime6}+(A+A')^6))={1\over 2}(\Pi(A,0,A')+\Pi(A',0,A)). 
\end{equation}
In the same way, the condition that $\Pi-\mathbf{P}(\tilde\sigma_2,\tilde\sigma_6)_{|B'=0}$ satisfies (3) is equivalent to 
the antisymmetry in $(A,B)$ of the polynomial $\Pi(A,B,0)-\mathbf{P}(2(A^2+B^2+(A+B)^2),2(A^6+B^6+(A+B)^6))$, in other words 
\begin{equation} \label{3'}
\mathbf{P}(2(A^2+B^2+(A+B)^2),2(A^6+B^6+(A+B)^6))={1\over 2}(\Pi(A,B,0)+\Pi(B,A,0)). 
\end{equation}
The conjunction of (\ref{2'}) and (\ref{3'}) is equivalent to the conjunction of (a2) and (a3). So we have proved the chain of equivalences 
$(\mathrm{b})\Leftrightarrow((1), (2) \mathrm{\ and\ } (3) )\Leftrightarrow((\mathrm{a1}), (\ref{2'})\mathrm{\ and\ } (\ref{3'}))\Leftrightarrow((\mathrm{a1}), (\mathrm{a2})\mathrm{\ and\ }  (\mathrm{a3}))$. 
\end{proof}

\medskip {\it(f) The polynomials $\mathbf{P}_{ij}$ satisfy property $(\mathrm{Cond}_{ij})$.}

\begin{lemma} \label{LemmeUV}
Let $U,V$ belong to $\C[A,B,A',B']^{as\times as}$. Then there exists $X$ in $\C[A,B,A',B']^{as\times as}$ such that 
$(U-\tilde\sigma_5\cdot X)_{|B'=0}=(V+\tilde\sigma_3\cdot X)_{|B'=0}=0$
iff 
$$
\tilde\sigma_3\cdot U+\tilde\sigma_5\cdot V\equiv 0\quad mod\quad  
{\mathbb{I}}\cdot ABA'B'\cdot\C[A,B,A',B']^{as\times as}.
$$
\end{lemma}

\begin{proof} If $(U,V)$ satisfies the first condition, then $U\equiv\tilde\sigma_5\cdot X$ 
and $V\equiv-\tilde\sigma_3\cdot X$ mod 
$ABA'B'\cdot\C[A,B,A',B']^{as\times as}$. Let $U',V'$ in $ABA'B'\cdot\C[A,B,A',B']^{as\times as}$ be such that 
$U-\tilde\sigma_5\cdot X=U'$ and $V+\tilde\sigma_3\cdot X=V'$, then 
$
\tilde\sigma_3 U+\tilde\sigma_5 V=
\tilde\sigma_3 (\tilde\sigma_5 X+U')+\tilde\sigma_5 (-\tilde\sigma_3 X+V')=\tilde\sigma_3 U'+\tilde\sigma_5 V'
\equiv 0$ mod ${\mathbb{I}}\cdot ABA'B'\cdot\C[A,B,A',B']^{as\times as}$, so $(U,V)$ satisfies the second condition. 

Assume that $(U,V)$ satisfies the second condition. Since 
$${\mathbb{I}}\cdot\C[A,B,A',B']^{as\times as}=\tilde\sigma_3\cdot\C[A,B,A',B']^{as\times as}+\tilde\sigma_5\cdot\C[A,B,A',B']^{as
\times as},$$ there exist  
$U',V'$ in $ABA'B'\cdot\C[A,B,A',B']^{as\times as}$, such that $\tilde\sigma_3\cdot U+\tilde\sigma_5\cdot V
=\tilde\sigma_3\cdot U'+\tilde\sigma_5\cdot V'$, which one rewrites as follows $\tilde\sigma_3\cdot (U-U')
+\tilde\sigma_5\cdot (V-V')=0$ (equality in $\C[A,B,A',B']$). 

The ring $\C[A,B,A',B']$ is an UFD, in which $\tilde\sigma_3$ and $\tilde\sigma_5$ have no common factor. This last equation then 
implies the existence of $X$ in $\C[A,B,A',B']$, such that $U-U'=\tilde\sigma_5\cdot X$, $V-V'=-\tilde\sigma_3\cdot X$. 
Taking into account symmetries, one finds that $X$ in fact belongs to $\C[A,B,A',B']^{as\times as}$. Then 
$(U-\tilde\sigma_5\cdot X)_{|B'=0}=U'_{|B'=0}=0$ and $(V+\tilde\sigma_3\cdot X)_{|B'=0}=-V'_{|B'=0}=0$. 
This implies that $(U,V)$ satisfies the first condition. 
\end{proof}
 
For each odd $i\geq 3$, $\lambda_i$ belongs to $\C[A,B,A',B']^{as\times as}$ (see (\ref{lambda:i})), i.e., is 
antisymmetric under $(A,B)\leftrightarrow(A',B')$ and $(A,A')\leftrightarrow(B,B')$, while $\tilde\sigma_i$
is symmetric under the same transformations. Equation (\ref{tau:ij}) then implies that for any odd $i,j\geq 3$, 
$\tau_{ij}$ belongs to $\C[A,B,A',B']^{as\times as}$. Since this element also belongs to $M=ABA'B'\cdot\C[A,B,A',B']^{as}$, 
one obtains
$$
\tau_{ij}\in ABA'B'\cdot\C[A,B,A',B']^{as\times as}. 
$$
Since $\mathbb{A}$ is contained in the subalgebra of $\C[A,B,A',B']$ of invariants under $(A,B)\leftrightarrow(A',B')$ and 
$(A,A')\leftrightarrow(B,B')$, the action of  $\mathbb{A}$ on $M=ABA'B'\cdot\C[A,B,A',B']^{as}$ restricts to an action on 
$M=ABA'B'\cdot\C[A,B,A',B']^{as\times as}$. 

\begin{lemma} \label{PropC} Let $i,j$ be integers $\geq 3$ and let $\mathbf{P}(\mathbf{x}_2,\mathbf{x}_6)$ be an element of 
$\C[\mathbf{x}_2,\mathbf{x}_6]$. Then 
$\mathbf{P}(\mathbf{x}_2,\mathbf{x}_6)$ satisfies $(\mathrm{Aux}_{ij})$ iff it satisfies 
\begin{equation*}
\tau_{ij}=\mathbf{P}(\tilde\sigma_2,\tilde\sigma_6)\cdot\tau_{35} \mathrm{\ mod\ } 
\mathbb{I}\cdot ABA'B'\cdot\C[A,B,A',B']^{as\times as} \quad\quad\quad\quad
(\mathrm{Cond}_{ij})
\end{equation*} 
in $ABA'B'\cdot\C[A,B,A',B']^{as\times as}$. 
\end{lemma}

\begin{proof} Let $\Pi_{ij}:=(P_{j5})_{|B'=0}\cdot {(\lambda_i)_{B'=0}\over(\lambda_3)_{B'=0}}
-(P_{i5})_{|B'=0}\cdot {(\lambda_j)_{B'=0}\over(\lambda_3)_{B'=0}}$; this is an element of $\C[A,B,A']$. 
The polynomials $(P_{i5})_{|B'=0}$ and ${(\lambda_i)_{B'=0}\over(\lambda_3)_{B'=0}}$ of $\C[A,B,A']$ 
are invariant under the symmetry $(A,B,A')\leftrightarrow(A,A',B)$. It follows that $\Pi_{ij}$ is invariant under 
the same symmetry, so $\Pi_{ij}(A,B,A')=\Pi_{ij}(A,A',B)$. Substituting $A=0$, one derives $\Pi_{ij}(0,B,A')=\Pi_{ij}(0,A',B)$, 
so {\it $\Pi_{ij}$ satisfies condition (a1) in Lemma \ref{PropositionPi}}. Substituting $(A,B,0)$ and then $(A,0,B)$ to $(A,B,A')$, 
and summing up the resulting identities, one obtains $\Pi_{ij}(A,B,0)+\Pi_{ij}(B,A,0)=\Pi_{ij}(A,0,B)+\Pi_{ij}(B,0,A)$, so 
{\it $\Pi_{ij}$ satisfies condition (a2) in Lemma \ref{PropositionPi}.} 

One checks that {\it $\mathbf{P}(\mathbf{x}_2,\mathbf{x}_6)$ satisfies $(\mathrm{Aux}_{ij})$} iff $(\Pi_{ij},\mathbf{P})$ satisfies 
condition (a3) in Lemma \ref{PropositionPi}; as conditions (a1) and (a2) are automatically satisfied, this {\it is 
equivalent to $(\Pi_{ij},\mathbf{P})$ satisfying condition (a) in Lemma \ref{PropositionPi}.} Taking into account
 Lemma \ref{PropositionPi}, this is equivalent to $(\Pi_{ij},\mathbf{P})$ satisfying condition (b) in Lemma \ref{PropositionPi}, 
i.e., taking into account the definition of $\Pi_{ij}$, to the {\it existence of $X$ in $\C[A,B,A',B']^{as\times as}$ such that} 
\begin{equation} \label{eq3}
(\mathbf{P}(\tilde\sigma_2,\tilde\sigma_6)+X)_{|B'=0}=(P_{j5})_{|B'=0}\cdot {(\lambda_i)_{B'=0}\over(\lambda_3)_{B'=0}}
-(P_{i5})_{|B'=0}\cdot {(\lambda_j)_{B'=0}\over(\lambda_3)_{B'=0}}
\end{equation}
(equality in $\C[A,B,A']$). Since $(\lambda_3)_{B'=0}=(\tilde\sigma_3)_{|B'=0}$, since this element of $\C[A,B,A']$ is nonzero and 
since $\C[A,B,A']$ is a domain, {\it (\ref{eq3}) is equivalent to} 
\begin{equation} \label{eq2'}
(\mathbf{P}(\tilde\sigma_2,\tilde\sigma_6)\cdot\lambda_3+X\cdot\tilde\sigma_3-P_{j5}\cdot\lambda_i+P_{i5}\cdot\lambda_j)_{|B'=0}=0
\end{equation}
Since $(\tau_{ij})_{|B'=0}=(\tau_{35})_{|B'=0}=0$, $(\tau_{ij}-\mathbf{P}(\tilde\sigma_2,\tilde\sigma_6)\cdot\tau_{35})_{|B'=0}=0$. 
The equality 
$$
\tau_{ij}-\mathbf{P}(\tilde\sigma_2,\tilde\sigma_6)\cdot\tau_{35}
=\tilde\sigma_3\cdot(P_{i3}\lambda_j-P_{j3}\lambda_i-\mathbf{P}(\tilde\sigma_2,\tilde\sigma_6)\lambda_5
-X\tilde\sigma_5)
+\tilde\sigma_5\cdot(P_{i5}\lambda_j-P_{j5}\lambda_i+\mathbf{P}(\tilde\sigma_2,\tilde\sigma_6)\lambda_3
+X\tilde\sigma_3) 
$$
then implies  
$(\tilde\sigma_3)_{|B'=0}\cdot(P_{i3}\lambda_j-P_{j3}\lambda_i-\mathbf{P}(\tilde\sigma_2,\tilde\sigma_6)\lambda_5
-X\tilde\sigma_5)_{|B'=0}
+(\tilde\sigma_5)_{|B'=0}\cdot(P_{i5}\lambda_j-P_{j5}\lambda_i+\mathbf{P}(\tilde\sigma_2,\tilde\sigma_6)\lambda_3
+X\tilde\sigma_3)_{|B'=0}=0$. 
Condition (\ref{eq2'}) implies the vanishing of the second part of the left-hand side of this identity, which, taking into account the 
non-vanishing of $(\tilde\sigma_3)_{|B'=0}$ and the fact that $\C[A,B,A']$ is a domain, implies the following identity
\begin{equation} \label{eq1'}
(P_{i3}\lambda_j-P_{j3}\lambda_i-\mathbf{P}(\tilde\sigma_2,\tilde\sigma_6)\lambda_5
-X\tilde\sigma_5)_{|B'=0}=0
\end{equation}
It follows that {\it there exists an $X$ in $\C[A,B,A',B']^{as\times as}$ satisfying (\ref{eq2'}) iff there exists 
$X$ in $\C[A,B,A',B']^{as\times as}$ satisfying the conjunction of (\ref{eq2'}) and (\ref{eq1'}).} 
This existence condition is exactly the first side of the equivalence of Lemma \ref{LemmeUV}, with 
$U=P_{i3}\lambda_j-P_{j3}\lambda_i-\mathbf{P}(\tilde\sigma_2,\tilde\sigma_6)\lambda_5$, 
$V=P_{i5}\lambda_j-P_{j5}\lambda_i+\mathbf{P}(\tilde\sigma_2,\tilde\sigma_6)\lambda_3$. 
By virtue of the equivalence of Lemma \ref{LemmeUV}, {\it the existence of $X$ in $\C[A,B,A',B']^{as\times as}$ satisfying 
the conjunction of (\ref{eq2'}) and (\ref{eq1'}) is equivalent to}  
$$
\tilde\sigma_3\cdot (P_{i3}\lambda_j-P_{j3}\lambda_i-\mathbf{P}(\tilde\sigma_2,\tilde\sigma_6)\lambda_5)
+\tilde\sigma_5\cdot (P_{i5}\lambda_j-P_{j5}\lambda_i+\mathbf{P}(\tilde\sigma_2,\tilde\sigma_6)\lambda_3)
\equiv 0\quad \mathrm{mod}\quad  
{\mathbb{I}}\cdot ABA'B'\cdot\C[A,B,A',B']^{as\times as}, 
$$
i.e., 
$$
\tau_{ij}-\mathbf{P}(\tilde\sigma_2,\tilde\sigma_6)\tau_{35}\equiv 0\quad \mathrm{mod}\quad  
{\mathbb{I}}\cdot ABA'B'\cdot\C[A,B,A',B']^{as\times as}, 
$$
which is $(\mathrm{Cond}_{ij})$. \end{proof}

Combining Lemma \ref{PropB} and Lemma \ref{PropC}, we obtain: 
\begin{proposition}\label{prop:Pij:Condij}
The poynomial $\mathbf{P}_{ij}$ given by (\ref{def:Pij}) satisfies condition $(\mathrm{Cond}_{ij})$ from Lemma \ref{PropC}.  
\end{proposition}

\medskip {\it(g) The inclusion ${\mathbb M}^{min}_0(\Sigma)\subset\mathrm{gr}_\Sigma^0({\mathbb A})\cdot\dot\tau_{35}$.}
For any odd $i,j\geq 3$, Proposition \ref{prop:Pij:Condij} implies the relation
$$
\tau_{ij}\equiv \mathbf{P}_{ij}(\tilde\sigma_2,\tilde\sigma_6)\cdot \tau_{35}\quad \mathrm{mod}\quad  {\mathbb{I}}\cdot M
$$
in $M$, where $\cdot$ is the action of $\mathbb{A}$ on $M$. The image of this relation in $\mathrm{gr}_\Sigma^0(M)
=M/{\mathbb{I}}\cdot M$ is 
\begin{equation}\label{id:tau:ij}
\dot\tau_{ij}=\mathbf{P}_{ij}(\dot\sigma_2,\dot\sigma_6)\cdot \dot\tau_{35}, 
\end{equation}
where $m\mapsto\dot m$ is the projection map $M\to M/{\mathbb{I}}\cdot M=\mathrm{gr}_\Sigma^0(M)$ and $\cdot$ is the action of 
$\mathrm{gr}_\Sigma^0(\mathbb{A})$ on $\mathrm{gr}_\Sigma^0(M)$. As ${\mathbb M}^{min}_0(\Sigma)$ is the subspace of 
$\mathrm{gr}_\Sigma^0(M)$ spanned by all the $\dot\tau_{ij}$ for odd $i,j\geq 3$, we derive:

\begin{proposition} \label{prop:incl}
${\mathbb M}^{min}_0(\Sigma)$ is contained in $\mathrm{gr}_\Sigma^0({\mathbb A})\cdot\dot\tau_{35}$. 
\end{proposition}

\subsection{Computation of the Hilbert-Poincar\'e series of $\mathrm{gr}_\Sigma^0({\mathbb{A}})\cdot\dot\tau_{35}$}\label{C1}
\label{subsect:hilb:cyclic}

Recall that $\mathrm{gr}_\Sigma^0(M)$ is a module over the ring $\mathrm{gr}_\Sigma^0({\mathbb{A}})$, containing a vector 
$\dot\tau_{35}$. Recall also that $\mathrm{gr}_\Sigma^0({\mathbb{A}})$ contains the images $\dot\sigma_2,\dot\sigma_4$
of $\tilde\sigma_2,\tilde\sigma_4$ under ${\mathbb{A}}\to{\mathbb{A}}/{\mathbb{I}}=\mathrm{gr}_\Sigma^0({\mathbb{A}})$. 

\begin{lemma} \label{Prop:act:sigma4}
The element $\dot\sigma_4-{1\over 4}(\dot\sigma_2)^2$ of $\mathrm{gr}_\Sigma^0({\mathbb{A}})$ annihilates the vector
$\dot\tau_{35}$: $(\dot\sigma_4-{1\over 4}(\dot\sigma_2)^2)\cdot\dot\tau_{35}=0$ (equality in $\mathrm{gr}_\Sigma^0(M)$). 
\end{lemma}
 
\begin{proof} Recall that the $\tilde\sigma_i$, as well as $\lambda_3,\lambda_5$ all belong to $\C[A,B,A',B']$. One checks that 
\begin{equation} \label{id:sigma:lambda}
(\tilde\sigma_3)_{|B'=0}=(\lambda_3)_{|B'=0}, \quad (\tilde\sigma_5)_{|B'=0}=(\lambda_5)_{|B'=0}.
\end{equation}
Set $P:=\tilde\sigma_4-{1\over 4}(\tilde\sigma_2)^2$ and $Q:=3(A^2B^{\prime 2}-A^{\prime 2}B^2)\in\C[A,B,A',B']^{as}$. 
One has $P=-3(AB'-A'B)^2$, so 
\begin{equation} \label{constat}
P_{|B'=0}=Q_{|B'=0}.
\end{equation}
Then
$(P\lambda_5)_{|B'=0}=(P\tilde\sigma_5)_{|B'=0}=(Q\tilde\sigma_5)_{|B'=0}$, using the second part of (\ref{id:sigma:lambda}) 
and then (\ref{constat}). Since both $P\lambda_5$ and $Q\tilde\sigma_5$ are anti-invariant under the exchanges
$(A,A')\leftrightarrow(B,B')$ and $(A,B)\leftrightarrow(A',B')$, it follows that $P\lambda_5=Q\tilde\sigma_5$
modulo $ABA'B'\cdot\C[A,B,A',B']^{as}$. In the same way, one proves (using the first part of (\ref{id:sigma:lambda}) that 
$P\lambda_3=Q\tilde\sigma_3$ modulo $ABA'B'\cdot\C[A,B,A',B']^{as}$.

As $\tau_{35}=\tilde\sigma_3\lambda_5-\tilde\sigma_5\lambda_3$, one has $P\cdot \tau_{35}=P\cdot(\tilde\sigma_3\lambda_5-\tilde\sigma_5\lambda_3)=\tilde\sigma_3\cdot P\lambda_5
-\tilde\sigma_5\cdot P\lambda_3\equiv\tilde\sigma_3\cdot Q\sigma_5-\tilde\sigma_5\cdot Q\sigma_3=0$ 
modulo $\sum_{i=3,5}\tilde\sigma_i\cdot ABA'B'\cdot\C[A,B,A',B']^{as}$. Therefore 
$P\cdot\tau_{35}\equiv0$ in $M=ABA'B'\cdot\C[A,B,A',B']^{as}$ modulo 
${\mathbb{I}}\cdot M=\sum_{i\mathrm{\ odd\ }\geq 3}\tilde\sigma_i\cdot ABA'B'\cdot\C[A,B,A',B']^{as}$. 
So $P\cdot\dot\tau_{35}=0$ in $\mathrm{gr}_\Sigma^0(M)$, as claimed. \end{proof}

Define $\overline{\mathrm{gr}_\Sigma^0({\mathbb{A}})}$ as the quotient of $\mathrm{gr}_\Sigma^0({\mathbb{A}})$
by the ideal generated by $\dot\sigma_4-{1\over 4}(\dot\sigma_2)^2$. Taking into account Proposition \ref{prop43}, 
this ring is the commutative ring freely generated by the images $\overline\sigma_2$ and  $\overline\sigma_6$
of $\dot\sigma_2$ and  $\dot\sigma_6$ under the projection map $\mathrm{gr}_\Sigma^0({\mathbb{A}})\to
\overline{\mathrm{gr}_\Sigma^0({\mathbb{A}})}$. Therefore: 

\begin{lemma}\label{str:gr:bar} There is a unique isomorphism of graded rings $\C[{\mathbf{x}}_2,{\mathbf{x}}_6]\simeq
\overline{\mathrm{gr}_\Sigma^0({\mathbb{A}})}$, given by ${\mathbf{x}}_2\mapsto\dot\sigma_2$, ${\mathbf{x}}_6\mapsto\dot\sigma_6$.  
 \end{lemma}

Lemma \ref{Prop:act:sigma4} implies that the ideal generated by $\dot\sigma_4-{1\over 4}(\dot\sigma_2)^2$
annihilates the cyclic submodule $\mathrm{gr}_\Sigma^0({\mathbb{A}})\cdot\dot\tau_{35}$ of $\mathrm{gr}_\Sigma^0(M)$ 
generated by $\dot\tau_{35}$, therefore that $\mathrm{gr}_\Sigma^0({\mathbb{A}})\cdot\dot\tau_{35}$ has a 
structure of module over $\overline{\mathrm{gr}_\Sigma^0({\mathbb{A}})}$, of which the $\mathrm{gr}_\Sigma^0({\mathbb{A}})$-module 
$\mathrm{gr}_\Sigma^0({\mathbb{A}})\cdot\dot\tau_{35}$ is the pullback under the projection $\mathrm{gr}_\Sigma^0({\mathbb{A}})\to
\overline{\mathrm{gr}_\Sigma^0({\mathbb{A}})}$. 
 
\begin{lemma} \label{LemmeIndAlg}
The morphism $\C[{\mathbf{x}}_2,{\mathbf{x}}_6]\to\C[A,A']$, ${\mathbf{x}}_2\mapsto(\tilde\sigma_2)_{|B=B'=0}$, ${\mathbf{x}}_6\mapsto(\tilde\sigma_6)_{|B=B'=0}$, is injective. 
\end{lemma}

\begin{proof} One checks that the pair $(A^2+AA'+A^{\prime 2},AA')$ is algebraically independent in $\C[A,A']$, and that
if $u,v$ are free commutative variables, then the pair $(4u,2(2u^3+3uv^2+3v^3))$ is algebraically independent in $\C[u,v]$. 
So the morphism $\C[{\mathbf{x}}_2,{\mathbf{x}}_6]\to\C[u,v]$, ${\mathbf{x}}_2\mapsto 4u$, ${\mathbf{x}}_6\mapsto 
2(2u^3+3uv^2+3v^3)$ is injective, as well as the morphism $\C[u,v]\to\C[A,A']$, $u\mapsto A^2+AA'+A^{\prime 2}$, $v\mapsto AA'$. 
As $(\tilde\sigma_2)_{|B=B'=0}=2(A^2+A^{\prime 2}+(A+A')^2)$ and $(\tilde\sigma_6)_{|B=B'=0}=2(A^6+A^{\prime 6}+(A+A')^6)$, 
the morphism $\C[{\mathbf{x}}_2,{\mathbf{x}}_6]\to\C[A,A']$ is the composition $\C[{\mathbf{x}}_2,{\mathbf{x}}_6]\to\C[u,v]\to\C[A,A']$, 
and is therefore injective. \end{proof}

Composing the isomorphism $\C[{\mathbf{x}}_2,{\mathbf{x}}_6]\simeq\overline{\mathrm{gr}_\Sigma^0({\mathbb{A}})}$ 
with the action map $\overline{\mathrm{gr}_\Sigma^0({\mathbb{A}})}\to{\mathrm{gr}}_\Sigma^0({\mathbb{A}})\cdot\dot\tau_{35}$ of $\overline{\mathrm{gr}_\Sigma^0({\mathbb{A}})}$ on $\mathrm{gr}_\Sigma^0({\mathbb{A}})\cdot\dot\tau_{35}$, 
one obtains a linear map $\C[{\mathbf{x}}_2,{\mathbf{x}}_6]\to\mathrm{gr}_\Sigma^0({\mathbb{A}})\cdot\dot\tau_{35}$. 

\begin{proposition} \label{PropInjectivité}
This linear map $\C[{\mathbf{x}}_2,{\mathbf{x}}_6]\to\mathrm{gr}_\Sigma^0({\mathbb{A}})\cdot\dot\tau_{35}$ induces an isomorphism 
of graded vector spaces $\C[{\mathbf{x}}_2,{\mathbf{x}}_6]\{8\}\simeq\mathrm{gr}_\Sigma^0({\mathbb{A}})\cdot\dot\tau_{35}$, where 
the right-hand side is graded by the weight degree and in the left-hand side $\{8\}$ means shifting the degree by 8. 
\end{proposition}

\begin{proof} As the element $\dot\tau_{35}\in\mathrm{gr}_\Sigma^0({\mathbb{A}})\cdot\dot\tau_{35}$ has degree 8, the map $\C[{\mathbf{x}}_2,{\mathbf{x}}_6]\to\mathrm{gr}_\Sigma^0({\mathbb{A}})\cdot\dot\tau_{35}$ induced by the action on $\dot\tau_{35}$ 
shifts the degrees by $8$. Since the action of $\mathrm{gr}_\Sigma^0({\mathbb{A}})$ on $\dot\tau_{35}$ factors through $\overline{\mathrm{gr}_\Sigma^0({\mathbb{A}})}$, we have $\mathrm{gr}_\Sigma^0({\mathbb{A}})\cdot\dot\tau_{35}
=\overline{\mathrm{gr}_\Sigma^0({\mathbb{A}})}\cdot\dot\tau_{35}$. Therefore the map $\C[{\mathbf{x}}_2,{\mathbf{x}}_6]\to
\mathrm{gr}_\Sigma^0({\mathbb{A}})\cdot\dot\tau_{35}$ is injective. 

Multiplying by $AA'BB'$ the composite map $\C[A,A',B,B']^{as}\hookrightarrow\C[A,A',B,B']\to\C[A,A']$, where the second map is 
the specialization for $B=B'=0$, we obtain a map 
\begin{equation} \label{spéc}
M=AA'BB'\cdot\C[A,A',B,B']^{as}\to AA'BB'\cdot\C[A,A'].  
\end{equation}
For each odd integer $i\geq 3$, there holds $(\tilde\sigma_i)_{|B=B'=0}=0$. It follows that the image of 
$\sum_{i\mathrm{\ odd\ }\geq 3}\tilde\sigma_i\cdot M={\mathbb{I}}\cdot M$ under (\ref{spéc}) is zero. 
It follows that there is a unique map 
$$
\mathrm{gr}_\Sigma^0(M)\to AA'BB'\cdot\C[A,A'],
$$ through which (\ref{spéc}) factors. 

If $k$ is an integer $\geq 2$, then this map intertwines the action of $\tilde\sigma_k$ on $\mathrm{gr}_\Sigma^0(M)$ with the 
product by $(\tilde\sigma_k)_{|B=B'=0}$ on $ABA'B'\cdot\C[A,A']$. It follows that the map 
%
$$\C[{\mathbf{x}}_2,{\mathbf{x}}_6]\to AA'BB'\C[A,A']$$
obtained as $\C[{\mathbf{x}}_2,{\mathbf{x}}_6]\to\mathrm{gr}_\Sigma^0({\mathbb{A}})\cdot\dot\tau_{35}\hookrightarrow
\mathrm{gr}_\Sigma^0(M)\to AA'BB'\C[A,A']$ is given by 
\begin{equation} \label{image:P}
P({\mathbf{x}}_2,{\mathbf{x}}_6)\mapsto P((\tilde\sigma_2)_{|B=B'=0},(\tilde\sigma_6)_{|B=B'=0})\cdot\mathrm{im}(\dot\tau_{35}\in\mathrm{gr}_\Sigma^0(M)\to AA'BB'\cdot \C[A,A']). 
\end{equation}
One computes 
$$
\mathrm{im}(\dot\tau_{35}\in\mathrm{gr}_\Sigma^0(M)\to AA'BB'\cdot \C[A,A'])=30AA'BB'(A^2-A^{\prime 2})(2A^2+5AA'+2A^{\prime 2}). 
$$
If $P\in\C[{\mathbf{x}}_2,{\mathbf{x}}_6]$ is in $\mathrm{ker}(\C[{\mathbf{x}}_2,{\mathbf{x}}_6]\to\mathrm{gr}_\Sigma^0(M))$, then the
 element (\ref{image:P}) is zero. As $\mathrm{im}(\dot\tau_{35}\in\mathrm{gr}_\Sigma^0(M)\to AA'BB'\cdot \C[A,A'])$ is nonzero and as 
$\C[A,A']$ has no zero divisors, $$P((\tilde\sigma_2)_{|B=B'=0},(\tilde\sigma_6)_{|B=B'=0})=0.$$ Then Lemma \ref{LemmeIndAlg} implies that 
$P=0$. This shows that the linear map $\C[{\mathbf{x}}_2,{\mathbf{x}}_6]\to AA'BB'\C[A,A']$ is injective. As this linear map decomposes as 
$\C[{\mathbf{x}}_2,{\mathbf{x}}_6]\to\mathrm{gr}_\Sigma^0({\mathbb{A}})\cdot\dot\tau_{35}\to AA'BB'\C[A,A']$, its injectivity implies that of 
$\C[{\mathbf{x}}_2,{\mathbf{x}}_6]\to\overline{\mathbb{A}}$.  
\end{proof}

Proposition \ref{PropInjectivité} enables one to compute the Hilbert series of $\mathrm{gr}_\Sigma^0({\mathbb{A}})\cdot\dot\tau_{35}
\simeq\C[{\mathbf{x}}_2,{\mathbf{x}}_6]\{8\}$ by $P_{\mathrm{gr}_\Sigma^0({\mathbb{A}})\cdot\dot\tau_{35}}(t)
=t^8\cdot P_{\C[{\mathbf{x}}_2,{\mathbf{x}}_6]}(t)$, therefore
\begin{equation} \label{SérieHilbMajorante}
P_{\mathrm{gr}_\Sigma^0({\mathbb{A}})\cdot\dot\tau_{35}}(t)={t^8\over(1-t^2)(1-t^6)}. 
\end{equation}

\subsection{A lower bound for the Hilbert-Poincar\'e series of $\mathbb{M}_0^{min}(\Sigma)$}\label{C3}\label{subsect:lower:bound:series}

In the proof of Proposition \ref{PropInjectivité}, we constructed a linear map 
$$\mathrm{gr}_\Sigma^0(M)\to AA'BB'\cdot\C[A,A'],
$$ such that 
the composed linear map $ABA'B'\cdot\C[A,B,A',B']^{as}\simeq
M\to M/{\mathbb{I}}\cdot M\simeq\mathrm{gr}_\Sigma^0(M)\to AA'BB'\cdot\C[A,A']$ is 
$ABA'B'\cdot P\mapsto ABA'B'\cdot P_{|B=B'=0}$. Recall that $\mathrm{gr}_\Sigma^0(M)$ contains 
elements $\dot\tau_{ij}$ defined for any odd $i,j\geq3$ (see introduction of Section \ref{Section:comp:Mminlowest}), and define 
$\overline\tau_{ij}$ as their images under $\mathrm{gr}_\Sigma^0(M)\to AA'BB'\cdot\C[A,A']$. 

As $\dot\tau_{ij}$ is itself the image of $\tau_{ij}\in M$ under $M\to\mathrm{gr}_\Sigma^0(M)$, $\overline\tau_{ij}$ is the image of 
$\tau_{ij}$ (defined in (\ref{tau:ij})) under the composed map $ABA'B'\cdot\C[A,B,A',B']^{as}\to AA'BB'\cdot\C[A,A']$. 

If $i$ is odd, then $\tilde\sigma_i(A,0,A',0)=\lambda_i(A,0,A',0)=0$, therefore 
$\tilde\sigma_i(A,B,A',B')\equiv{\partial\tilde\sigma_i\over\partial B}(A,0,A',0)\cdot B+
{\partial\tilde\sigma_i\over\partial B'}(A,0,A',0)\cdot B'$ and $\lambda_i(A,B,A',B')\equiv{\partial\lambda_i\over\partial B}(A,0,A',0)
\cdot B+{\partial\lambda_i\over\partial B'}(A,0,A',0)\cdot B'$ mod $(B,B')^2$ (relations in $\C[A,B,A',B']$), where $(B,B')$ is the ideal 
of $\C[A,B,A',B']$ generated by $B,B'$.  Plugging this in (\ref{tau:ij}), one obtains an expression of the class of $\tau_{ij}$ in 
$\C[A,B,A',B']/(B,B')^3$ in terms of ${\partial\tilde\sigma_i\over\partial B}(A,0,A',0),\ldots,{\partial\lambda_i\over\partial B'}(A,0,A',0)$. 
One checks that ${\partial\lambda_i\over\partial B}(A,0,A',0)={\partial\tilde\sigma_i\over\partial B}(A,0,A',0)$,  
${\partial\lambda_i\over\partial B'}(A,0,A',0)=-{\partial\tilde\sigma_i\over\partial B'}(A,0,A',0)$, therefore 
$\tau_{ij}\equiv-2BB'({\partial\tilde\sigma_i\over\partial B}{\partial\tilde\sigma_j\over\partial B'}
-{\partial\tilde\sigma_j\over\partial B}{\partial\tilde\sigma_i\over\partial B'})(A,0,A',0)$ mod $(B,B')^3$. 
One also computes ${\partial\tilde\sigma_i\over\partial B}(A,0,A',0)=i((A+A')^{i-1}-A^{i-1})$, 
${\partial\tilde\sigma_i\over\partial B'}(A,0,A',0)=i((A+A')^{i-1}-A^{\prime i-1})$. Together with the previous formula, this implies
$$
{\overline\tau}_{ij}=2BB'\cdot P_{ij}(A,A')\in ABA'B'\cdot\C[A,A'], 
$$
where 
\begin{align*}
P_{ij}(A,A'):= & ij(A^{\prime j-1}-A^{j-1})(A+A')^{i-1}+ij(A^{i-1}-A^{\prime i-1})(A+A')^{j-1}-ijA^{i-1}A^{\prime j-1}
+ijA^{j-1}A^{\prime i-1}\\ & \in  AA'\cdot\C[A,A'], 
\end{align*}
Recall that the subspace $\mathbb{M}_0^{min}(\Sigma)$ of $\mathrm{gr}_\Sigma^0(M)$ is defined as the linear span of all $\dot\tau_{ij}$, 
for $i,j$ odd $\geq 3$. Define  $\overline{\mathbb{M}_0^{min}(\Sigma)}\subset AA'BB'\cdot\C[A,A']$ to be the image of this 
subspace under $\mathrm{gr}_\Sigma^0(M)\to AA'BB'\cdot\C[A,A']$, so 
$$
\overline{\mathbb{M}_0^{min}(\Sigma)}=\mathrm{im}(\mathbb{M}_0^{min}(\Sigma)
\subset\mathrm{gr}_\Sigma^0(M)\to AA'BB'\cdot\C[A,A']).
$$
Then $\overline{\mathbb{M}_0^{min}(\Sigma)}$ is the linear span of all $\overline\tau_{ij}$, for $i,j$ odd $\geq 3$. 
Set $\mathcal{A}:=\{(a_{ij})_{i,j\mathrm{\ odd\ }\geq 3}, a_{ij}\in\C, a_{ij}+a_{ji}=0\mathrm{\ for\ any\ }i,j\}$. Taking into 
account the identity $\overline\tau_{ij}=-\overline\tau_{ji}$, there is a surjective linear map $\mathcal{A}\to
\overline{\mathbb{M}_0^{min}(\Sigma)}$, given by 
$$
(a_{ij})_{i,j\mathrm{\ odd\ }\geq 3}\mapsto \sum_{i,j\mathrm{\ odd\ }\geq 3}a_{ij}{\overline\tau}_{ij}. 
$$
Denote by $R$ the kernel of this map; it identifies with the vector space of relations between the polynomials ${\overline\tau}_{ij}$.
The vector spaces $\mathcal{A}$ and $\overline{\mathbb{M}_0^{min}(\Sigma)}$ are graded: one has 
$\mathcal{A}=\oplus_{n\geq 0}\mathcal{A}_n$ and $\overline{\mathbb{M}_0^{min}(\Sigma)}=\oplus_{n\geq 0}
\overline{\mathbb{M}_0^{min}(\Sigma)}_n$, where $\mathcal{A}_n:=\{(a_{ij})_{i,j\mathrm{\ odd\ }\geq 3}|a_{ij}=0$ if $i+j\neq n\}$ and 
$\overline{\mathbb{M}_0^{min}(\Sigma)}_n:=$Span$\{\overline\tau_{ij}|i,j$ are odd $\geq 3$ and $i+j=n\}$. The linear map $\mathcal{A}\to\overline{\mathbb{M}_0^{min}(\Sigma)}$ is then graded, which implies that 
$R=\mathrm{ker}(\mathcal{A}\to\overline{\mathbb{M}_0^{min}(\Sigma)})$ is a graded vector subspace of $\mathcal{A}$. Therefore 
$R=\oplus_{n\geq 0}R_n$, where $R_n:=R\cap\mathcal{A}_n$. 

Let $n$ be an integer $\geq 0$. If $n$ is odd, then $\mathcal{A}_n=0$, therefore $R_n=0$. Assume that $n$ is even. Let 
$(a_{ij})_{i,j\mathrm{\ odd\ }\geq 3,i+j=n}$ belong to $\mathcal{A}_n$. Its image under 
the map $\mathcal{A}\to\overline{\mathbb{M}_0^{min}(\Sigma)}$ is equal to 
$$
4BB'(A+A')^{n-2}\big(P(x)-P(1-x)-(1-x)^{n-2}P({x\over 1-x})\big), 
$$
where $x:=A/(A+A')$ and $P(x):=\sum_{i,j|i,j\mathrm{\ odd\ }\geq 3,i+j=n}ija_{ij}x^{j-1}$. The map 
$(a_{ij})_{i,j\mathrm{\ odd\ }\geq 3,i+j=n}\mapsto P$ sets up an isomorphism between $\mathcal{A}_n$
and $\{P\in\C[X]|P$ is even, has degree $\leq n-2$, $P(0)=0$, and $X^{n-2}P(1/X)=-P(X)\}$. This isomorphism
induces an isomorphism between $R_n$ and 
\begin{align*}
\Sigma_n:=\{P\in\C[X]| & P\mathrm{\ is\ even,\ has\ degree\ }\leq n-2,\ P(0)=0,\ X^{n-2}P(1/X)=-P(X),\mathrm{\ and\ } \\ & 
P(X)-P(1-X)-(1-X)^{n-2}P({X\over 1-X})=0\}. 
\end{align*}
The two last equations in the definition of $\Sigma_n$ can be rewritten as follows $P+P|S=P+P|U+P|U^2=0$, where 
$P|\gamma:=(cX+d)^{n-2}P({aX+b\over cX+d})$ for $\gamma=\begin{pmatrix} a & b \\ c & d\end{pmatrix}\in\mathrm{SL}_2(\Z)$,
and $S:=\begin{pmatrix} 0 & -1 \\ 1 & 0\end{pmatrix}$, $U:=\begin{pmatrix} 1 & -1 \\ 1 & 0\end{pmatrix}$.  

Let $W_n^+$ be the space of polynomials $P\in\C[X]$ which are even, of degree $\leq n-2$, and satisfy $P+P|S=P+P|U+P|U^2=0$. 
According to \cite{Za}, there is an injective map $S_n\to W_n^+$, injective and with image of codimension 1, where $S_n$
is the space of cusp modular forms for the group $\mathrm{SL}_2(\Z)$. It follows that $\mathrm{dim}(W_n^+)=\mathrm{dim}(S_n)+1$. 

On the other hand, $\Sigma_n$ is the kernel of the linear map $W_n^+\to\C$, $P\mapsto P(0)$. As $W_n^+$ contains the polynomial
$X^{n-2}-1$, the linear map is nonzero. Therefore $\mathrm{dim}(\Sigma_n)=\mathrm{dim}(W_n^+)-1$, which implies that 
$\mathrm{dim}(\Sigma_n)=\mathrm{dim}(S_n)$ and therefore $\mathrm{dim}(R_n)=\mathrm{dim}(S_n)$. 

The commutative ring of modular forms under $\mathrm{SL}_2(\Z)$ is known to be freely generated in degrees 4 and 6; 
its Hilbert series is therefore $1/((1-t^4)(1-t^6))$.  On the other hand, the graded vector space of cusp forms under $\mathrm{SL}_2(\Z)$
is a free module of rank one over this group, with generator in degree 12. The Hilbert series of $S:=\oplus_{n\geq 0}S_n$
is therefore $P_S(t)=t^{12}/((1-t^4)(1-t^6))$, which implies that the Hilbert series of $R=\oplus_{n\geq 0}R_n$ is 
$$
P_R(t)={t^{12}\over{(1-t^4)(1-t^6)}}. 
$$

One also computes $\mathrm{dim}(\mathcal{A}_n)=[n/4]-1$ ($[x]$ meaning the integral part of $x$); from where one derives the Hilbert series 
of $\mathcal{A}=\oplus_{n\geq 0}\mathcal{A}_n$:
$$
P_{\mathcal{A}}(t)={t^8\over{(1-t^2)(1-t^4)}}. 
$$
There is an exact sequence $0\to R\to \mathcal{A}\to\overline{\mathbb{M}_0^{min}(\Sigma)}\to 0$, therefore
\begin{equation} \label{CalculSérieBarbarLambda}
P_{\overline{\mathbb{M}_0^{min}(\Sigma)}}(t)=P_{\mathcal{A}}(t)-P_R(t)={t^8\over{(1-t^2)(1-t^6)}}. 
\end{equation}
The fact that $\overline{\mathbb{M}_0^{min}(\Sigma)}$ is a quotient of $\mathbb{M}_0^{min}(\Sigma)$ implies the inequality  
$P_{\mathbb{M}_0^{min}(\Sigma)}(t)\geq P_{\overline{\mathbb{M}_0^{min}(\Sigma)}}(t)$ (meaning that the difference of these series has 
$\geq 0$ terms), 
which given (\ref{CalculSérieBarbarLambda}) yields 
\begin{equation}\label{lower:bound:hilb:M2}
P_{\mathbb{M}_0^{min}(\Sigma)}(t)\geq {t^8\over (1-t^2)(1-t^6)}.
\end{equation}

\subsection{Computation of ${\mathbb M}^{min}_0(\Sigma)$} \label{subsect:conclusions}

Comparing (\ref{SérieHilbMajorante}) and (\ref{lower:bound:hilb:M2}), one obtains
$P_{\mathbb{M}_0^{min}(\Sigma)}(t)\geq P_{\mathrm{gr}_\Sigma^0(\mathbb{A})\cdot\dot\tau_{35}}(t)$. 
Combining this inequality with the opposite inequality following from the inclusion 
$\mathbb{M}_0^{min}(\Sigma)\subset\mathrm{gr}_\Sigma^0(\mathbb{A})\cdot\dot\tau_{35}$ (see Proposition \ref{prop:incl}), one
obtains the equality 
$$
P_{\mathbb{M}_0^{min}(\Sigma)}(t)=P_{\mathrm{gr}_\Sigma^0(\mathbb{A})\cdot\dot\tau_{35}}(t)(= {t^8\over (1-t^2)(1-t^6)})
$$
and therefore: 
\begin{proposition} \label{prop:7:11}
The subspaces $\mathbb{M}_0^{min}(\Sigma)$ and $\mathrm{gr}_\Sigma^0(\mathbb{A})\cdot\dot\tau_{35}$ of 
$\mathrm{gr}_\Sigma^0(M)$ are equal. 
\end{proposition}

Combining this result with the isomorphism $\mathrm{gr}_\Sigma^0(\mathbb{A})\cdot\dot\tau_{35}\simeq
\C[{\mathbf{x}}_2,{\mathbf{x}}_6]\{8\}$ (Proposition \ref{PropInjectivité}), we obtain a isomorphism of graded vector spaces
\begin{equation}\label{iso:M:min}
\mathbb{M}_0^{min}(\Sigma)\simeq\C[{\mathbf{x}}_2,{\mathbf{x}}_6]\{8\}. 
\end{equation}
The map $\{,\}:\Lambda^2(\Sigma)\to\mathbb{M}_0^{min}(\Sigma)$ is given by $\{\sigma_i,\sigma_j\}=($class of $c(\sigma_i,\sigma_j))$
for any odd $i,j\geq 3$. Taking into account the relation $c(\sigma_i,\sigma_j)=-2\tau_{ij}$ (Subsection \ref{subsection:ingredients}),
we obtain $\{\sigma_i,\sigma_j\}=-2\dot\tau_{ij}$. Then (\ref{id:tau:ij}) yields $\{\sigma_i,\sigma_j\}
=-2{\mathbf{P}}_{ij}(\dot\sigma_2,\dot\sigma_6)\dot\tau_{35}$. It follows that the map 
$$
\{,\}:\Lambda^2(\Sigma)\to\C[{\mathbf{x}}_2,{\mathbf{x}}_6]\{8\}
$$
obtained by composing $\Lambda^2(\Sigma)\stackrel{\{,\}}{\to}\mathbb{M}_0^{min}(\Sigma)$ with the isomorphism (\ref{iso:M:min})
is given by 
\begin{equation}\label{form:of:bracket}
\{\sigma_i,\sigma_j\}=-2{\mathbf{P}}_{ij}({\mathbf{x}}_2,{\mathbf{x}}_6) \quad \mathrm{for\ any\ odd\ }i,j\geq 3, 
\end{equation}
where ${\mathbf{P}}_{ij}$ is given by (\ref{def:Pij}). Introducing the generating series 
$$
\sigma(t):=\sum_{i\mathrm{\ odd\ }\geq 3}\sigma_i\cdot t^i\in\Sigma[[t]], 
$$
We may rewrite (\ref{form:of:bracket}) as follows
$$
\{\sigma(t),\sigma(u)\}=-{3\over5}\Big({X\over D}\big(\sigma({\mathbf{x}}_2,{\mathbf{x}}_6),\pi({\mathbf{x}}_2,{\mathbf{x}}_6),t\big){Y\over D}\big(\sigma({\mathbf{x}}_2,{\mathbf{x}}_6),\pi({\mathbf{x}}_2,{\mathbf{x}}_6),u\big)-(t\leftrightarrow u)\Big), 
$$  
where $X,Y,D,\sigma({\mathbf{x}}_2,{\mathbf{x}}_6),\pi({\mathbf{x}}_2,{\mathbf{x}}_6)$ are as in (\ref{déf:X}), (\ref{déf:Y}), (\ref{déf:D}), 
and (\ref{def:sigma:pi}). 
Define 
\begin{equation}\label{def:d}
d(\mathbf{x}_2,\mathbf{x}_6,t):=(1-{\mathbf{x}_2\over 4}t^2)^2+(-{1\over 6}\mathbf{x}_6
+{1\over96}(\mathbf{x}_2)^3)t^6,  
\end{equation}
then one checks that 
$$
t{d\over dt}({-{1\over 3}t^3\over d(\mathbf{x}_2,\mathbf{x}_6,t)})=-{1\over2}\Big(
{X\over D}\big(\sigma({\mathbf{x}}_2,{\mathbf{x}}_6),\pi({\mathbf{x}}_2,{\mathbf{x}}_6),t\big)
+\sigma({\mathbf{x}}_2,{\mathbf{x}}_6)
{Y\over D}\big(\sigma({\mathbf{x}}_2,{\mathbf{x}}_6),\pi({\mathbf{x}}_2,{\mathbf{x}}_6),t\big)\Big), 
$$
$$
t{d\over dt}({-{1\over 5}t^5\over d(\mathbf{x}_2,\mathbf{x}_6,t)})=-{3\over5}
{Y\over D}\big(\sigma({\mathbf{x}}_2,{\mathbf{x}}_6),\pi({\mathbf{x}}_2,{\mathbf{x}}_6),t\big). 
$$
It follows that if  
\begin{equation}\label{xi3:xi5}
\xi_3({\mathbf{x}}_2,{\mathbf{x}}_6,t):=t{d\over dt}({{1\over 3}t^3\over d(\mathbf{x}_2,\mathbf{x}_6,t)}), \quad 
\xi_5({\mathbf{x}}_2,{\mathbf{x}}_6,t):=t{d\over dt}({{1\over 5}t^5\over d(\mathbf{x}_2,\mathbf{x}_6,t)}),
\end{equation}
then 
\begin{equation}\label{bracket:sigma}
\{\sigma(t),\sigma(u)\}=-2(\xi_3({\mathbf{x}}_2,{\mathbf{x}}_6,t)\xi_5({\mathbf{x}}_2,{\mathbf{x}}_6,u)
-\xi_3({\mathbf{x}}_2,{\mathbf{x}}_6,u)\xi_5({\mathbf{x}}_2,{\mathbf{x}}_6,t)). 
\end{equation}

\section{Computation of the lower bound ${\mathbb M}^{min}(\Sigma)$}\label{Section:comp:Mminfull}

To compute the lower bound  ${\mathbb M}^{min}(\Sigma)$ as a (weight degree, $\Sigma$-degree)-bigraded vector space, we first construct a 
surjective morphism $\C[\mathbf{x}_2,\mathbf{x}_6,\mathbf{x}_3,\mathbf{x}_5]\{8,2\}\to\mathbb{M}^{min}(\Sigma)$ between this space and 
a polynomial ring with shifted degrees (Subsection \ref{subsect:surj}). We then show that this morphism is injective (Subsection \ref{subsect:inj}), 
and is therefore an isomorphism. We then compute the action (Subsection \ref{subsect:SSigma:str}) and the filtration (Subsection 
\ref{subsect:depth:M:min}) obtained from the action of $S(\Sigma)$ on ${\mathbb M}^{min}(\Sigma)$ and from the depth filtration of 
${\mathbb M}^{min}(\Sigma)$ by transport of structure through this isomorphism. We summarize these results in Subsection 
\ref{subsect:summary}. 

\subsection{A surjective morphism $\phi:\C[\mathbf{x}_2,\mathbf{x}_6,\mathbf{x}_3,\mathbf{x}_5]\{8,2\}\to\mathbb{M}^{min}(\Sigma)$} 
\label{subsect:surj}

By Subsection \ref{545}, the $\Sigma$-graded space $\mathrm{gr}_\Sigma(M)=\oplus_{k\geq 0}\mathrm{gr}_\Sigma^k(M)$ is equipped with 
an action of the $\Sigma$-graded ring $\mathrm{gr}_\Sigma(\mathbb{A})=\oplus_{k\geq 0}\mathrm{gr}_\Sigma^k(\mathbb{A})$. Then 
$\mathrm{gr}_\Sigma(\mathbb{A})$ contains $\mathrm{gr}_\Sigma^0(\mathbb{A})$ as a subring, and by Subsection \ref{543}, it also 
contains a $\Sigma$-graded subring $\mathbb{S}$. By Subsection \ref{546}, the space $\mathrm{gr}_\Sigma^0(M)$ contains a vector subspace 
$\mathbb{M}_0^{min}(\Sigma)$, and the $\Sigma$-graded subspace $\mathbb{M}^{min}(\Sigma)$ of $\mathrm{gr}_\Sigma(M)$ is equal to 
$\mathbb{S}\cdot\mathbb{M}_0^{min}(\Sigma)$. In Proposition \ref{prop:7:11}, we showed that $\mathbb{M}_0^{min}(\Sigma)$ coincides with $\mathrm{gr}_\Sigma^0(\mathbb{A})\cdot\dot\tau_{35}$, where the element $\dot\tau_{35}\in\mathrm{gr}_\Sigma^0(M)$ 
has been defined in the introduction of Section \ref{Section:comp:Mminlowest}. It follows that 
\begin{equation}\label{formula:M:min}
\mathbb{M}^{min}(\Sigma)=(\mathbb{S}\cdot\mathrm{gr}_\Sigma^0(\mathbb{A}))\cdot\dot\tau_{35},
\end{equation} 
where $\mathbb{S}\cdot\mathrm{gr}_\Sigma^0(\mathbb{A})$ is the subalgebra of $\mathrm{gr}_\Sigma(\mathbb{A})$, 
image of the product morphism $\mathbb{S}\otimes\mathrm{gr}_\Sigma^0(\mathbb{A})\to\mathrm{gr}_\Sigma(\mathbb{A})$. 

Under the isomorphism $\mathrm{gr}_\Sigma(\mathbb{A})\simeq\mathrm{gr}_\Sigma^0(\mathbb{A})[\mathbf{x}_3,\mathbf{x}_5]$
(Theorem \ref{StructureDAnneaux}), the subalgebra $\mathrm{gr}_\Sigma^0(\mathbb{A})$ of $\mathrm{gr}_\Sigma(\mathbb{A})$ is taken 
to the subalgebra $\mathrm{gr}_\Sigma^0(\mathbb{A})$ of $\mathrm{gr}_\Sigma^0(\mathbb{A})[\mathbf{x}_3,\mathbf{x}_5]$, while the 
subalgebra $\mathbb{S}$ of $\mathrm{gr}_\Sigma(\mathbb{A})$ is taken to a subalgebra of 
$\mathrm{gr}_\Sigma^0(\mathbb{A})[\mathbf{x}_3,\mathbf{x}_5]$ containing $\C[\mathbf{x}_3,\mathbf{x}_5]$ (see Subsection \ref{results:S}). 
As the product of these two subalgebras of $\mathrm{gr}_\Sigma^0(\mathbb{A})[\mathbf{x}_3,\mathbf{x}_5]$ is equal to 
$\mathrm{gr}_\Sigma^0(\mathbb{A})[\mathbf{x}_3,\mathbf{x}_5]$ itself, the product $\mathbb{S}\cdot\mathrm{gr}_\Sigma^0(\mathbb{A})$ 
coincides with $\mathrm{gr}_\Sigma(\mathbb{A})$. Equation (\ref{formula:M:min}) then implies that 
\begin{equation} \label{id:M:min}
\mathbb{M}^{min}(\Sigma)=\mathrm{gr}_\Sigma(\mathbb{A})\cdot\dot\tau_{35}. 
\end{equation}
Let $\overline{\mathrm{gr}_\Sigma(\mathbb{A})}$ be the quotient of the ring $\mathrm{gr}_\Sigma(\mathbb{A})$ by the principal ideal 
generated by the element $\dot\sigma_4-{1\over4}(\dot\sigma_2)^2\in\mathrm{gr}_\Sigma^0(\mathbb{A})\subset 
\mathrm{gr}_\Sigma(\mathbb{A})$. By Lemma \ref{Prop:act:sigma4}, this element annihilates $\dot\tau_{35}$. Together with (\ref{id:M:min}), 
this implies that the action of $\mathrm{gr}_\Sigma(\mathbb{A})$ on $\mathbb{M}^{min}(\Sigma)$ factors through 
$\overline{\mathrm{gr}_\Sigma(\mathbb{A})}$ and that 
\begin{equation}\label{eq:M:min}
\mathbb{M}^{min}(\Sigma)=\overline{\mathrm{gr}_\Sigma(\mathbb{A})}\cdot\dot\tau_{35}. 
\end{equation} 
As $\dot\sigma_4-{1\over4}(\dot\sigma_2)^2$ belongs to $\mathrm{gr}_\Sigma^0(\mathbb{A})$, the isomorphism 
$\mathrm{gr}_\Sigma(\mathbb{A})\simeq\mathrm{gr}_\Sigma^0(\mathbb{A})[\mathbf{x}_3,\mathbf{x}_5]$ from Theorem
\ref{StructureDAnneaux} induces an isomorphism 
\begin{equation}\label{iso:gr:A}
\overline{\mathrm{gr}_\Sigma(\mathbb{A})}\simeq\overline{\mathrm{gr}_\Sigma^0(\mathbb{A})}[\mathbf{x}_3,\mathbf{x}_5], 
\end{equation}
where we recall that $\overline{\mathrm{gr}_\Sigma^0(\mathbb{A})}$ is the quotient of $\mathrm{gr}_\Sigma^0(\mathbb{A})$
by its ideal generated by $\dot\sigma_4-{1\over4}(\dot\sigma_2)^2$ (see Subsection \ref{subsect:hilb:cyclic}). Combining the 
isomorphism (\ref{iso:gr:A}) with the isomorphism $\overline{\mathrm{gr}_\Sigma^0(\mathbb{A})}\simeq\C[\mathbf{x}_2,\mathbf{x}_6]$
(Lemma \ref{str:gr:bar}), we obtain an isomorphism 
\begin{equation}\label{iso:A}
\overline{\mathrm{gr}_\Sigma(\mathbb{A})}\simeq\C[\mathbf{x}_2,\mathbf{x}_6,\mathbf{x}_3,\mathbf{x}_5]
\end{equation}
of (weight degree, $\Sigma$-degree)-bigraded rings, where the bidegrees of $\mathbf{x}_2,\mathbf{x}_6,\mathbf{x}_3,\mathbf{x}_5$
are respectively (2,0), (6,0), (3,1), (5,1). 

The element $\dot\tau_{35}$ of $\mathbb{M}^{min}(\Sigma)$ has weight degree 8 and $\Sigma$-degree 2 (as it belongs to 
$\mathbb{M}^{min}_0(\Sigma)$). Combining the isomorphism (\ref{iso:gr:A}) with the action map of 
$\overline{\mathrm{gr}_\Sigma(\mathbb{A})}$ on $\dot\tau_{35}$ and taking (\ref{eq:M:min}) into account, {\it we obtain a surjective 
bigraded linear map 
$$
\phi:\C[\mathbf{x}_2,\mathbf{x}_6,\mathbf{x}_3,\mathbf{x}_5]\{8,2\}\to\mathbb{M}^{min}(\Sigma), 
$$ 
where $\cdots\{8,2\}$ means shifting the bidegree by (8,2).} This linear map is compatible with the structures of modules over both sides of 
the inverse of isomorphism (\ref{iso:A}).

\subsection{Injectivity of $\phi:\C[\mathbf{x}_2,\mathbf{x}_6,\mathbf{x}_3,\mathbf{x}_5]\{8,2\}\to\mathbb{M}^{min}(\Sigma)$}
\label{subsect:inj}

For each odd $i\geq 3$, the element $\tilde\sigma_i$ of $\mathbb{A}\subset\C[A,A',B,B']$ is contained in $(B,B')$, the ideal of $\C[A,A',B,B']$
generated by $B,B'$. As the ideal $\mathbb{I}$ of $\mathbb{A}$ is generated by the $\tilde\sigma_i$, $i$ odd $\geq 3$, it follows that 
\begin{equation}\label{I:BB'}
\mathbb{I}\mathrm{\ is\ contained\ in\ }(B,B').
\end{equation}
Recall that $M$ is a vector subspace of $ABA'B'\cdot\C[A,B,A',B']$. It then follows from (\ref{I:BB'}) that for any $k\geq 0$, $\mathbb{I}^k\cdot M$
is contained in $ABA'B'\cdot(B,B')^k$. The collection of maps $\mathbb{I}^k\cdot M\hookrightarrow ABA'B'\cdot(B,B')^k$ induces quotient maps
\begin{equation}\label{quot:map}
\mathrm{gr}_\Sigma^k(M)={\mathbb{I}^k\cdot M\over\mathbb{I}^{k+1}\cdot M}\to
ABA'B'\cdot{(B,B')^k\over(B,B')^{k+1}}\simeq ABA'B'\cdot\C[A,B,A',B']_{B\mathrm{-deg}+B'\mathrm{-deg}=k},
\end{equation}
where the index means the linear space of monomials whose sum of partial degrees in $B$ and $B'$ equals $k$. The direct sum of 
all maps (\ref{quot:map}) is a bigraded linear map 
\begin{equation}\label{mod:map}
\mathrm{gr}_\Sigma(M)\to ABA'B'\cdot\C[A,B,A',B'], 
\end{equation}
where on the right side the weight degree is the total degree and the $\Sigma$-degree is the sum of partial degrees in $B$ and $B'$. 
 
Similarly, the morphism $(\mathbb{A},\mathbb{I})\to(\C[A,B,A',B'],(B,B'))$ of pairs (a graded algebra, a graded ideal of this algebra) induces a 
morphism of bigraded rings $\oplus_{k\geq 0}\mathbb{I}^k/\mathbb{I}^{k+1}\to\oplus_{k\geq 0}(B,B')^k/(B,B')^{k+1}$, which identifies
with a morphism 
\begin{equation}\label{alg:map}
\mathrm{gr}_\Sigma(\mathbb{A})\to\C[A,A',B,B']
\end{equation}
of bigraded rings, where on the r.h.s. the pair (weight degree, $\Sigma$-degree) is again (total degree, sum of partial degrees in $B$ and $B'$). 

The morphism (\ref{mod:map}) is then compatible with the morphism (\ref{alg:map}) and with the action of both sides of (\ref{alg:map}) on both 
sides of (\ref{mod:map}).  

The presentation results Lemma \ref{str:gr:bar} and Proposition \ref{prop43} imply that there is a morphism of graded algebras 
$sec:\overline{\mathrm{gr}_\Sigma^0(\mathbb{A})}\to\mathrm{gr}_\Sigma^0(\mathbb{A})$, uniquely defined by 
$(\mathrm{class\ of\ }\dot\sigma_2)\mapsto\dot\sigma_2$, $(\mathrm{class\ of\ }\dot\sigma_6)\mapsto\dot\sigma_6$. Tensoring 
this morphism with the identity morphism of the polynomial ring $\C[\mathbf{x}_3,\mathbf{x}_5]$, and composing the resulting morphism 
with the isomorphisms $\overline{\mathrm{gr}_\Sigma(\mathbb{A})}\simeq\overline{\mathrm{gr}_\Sigma^0(\mathbb{A})}[\mathbf{x}_3,
\mathbf{x}_5]$ (see (\ref{iso:gr:A})) and $\mathrm{gr}_\Sigma(\mathbb{A})\simeq\mathrm{gr}_\Sigma^0(\mathbb{A})[\mathbf{x}_3,\mathbf{x}_5]$
(Theorem \ref{StructureDAnneaux}), one obtains a morphism of bigraded algebras 
$\overline{\mathrm{gr}_\Sigma(\mathbb{A})}\to\mathrm{gr}_\Sigma(\mathbb{A})$, extending 
$sec:\overline{\mathrm{gr}_\Sigma^0(\mathbb{A})}\to\mathrm{gr}_\Sigma^0(\mathbb{A})$ by $(\mathrm{class\ of\  }\dot\sigma_3)
\mapsto\dot\sigma_3$, $(\mathrm{class\ of\ }\dot\sigma_5)\mapsto\dot\sigma_5$.

Since this morphism is a section of the quotient map $\mathrm{gr}_\Sigma(\mathbb{A})\to\overline{\mathrm{gr}_\Sigma(\mathbb{A})}$, 
the identity morphism of $\mathbb{M}^{min}(\Sigma)$ is compatible with the algebra morphism 
$\overline{\mathrm{gr}_\Sigma(\mathbb{A})}\to\mathrm{gr}_\Sigma(\mathbb{A})$ and with the module structure of 
$\mathbb{M}^{min}(\Sigma)$ with respect to both sides of this morphism. 

Then there is a sequence of morphisms of bigraded vector spaces
\begin{equation}\label{map:d}
\C[\mathbf{x}_2,\mathbf{x}_6,\mathbf{x}_3,\mathbf{x}_5]\{8,2\}\stackrel{\phi}{\twoheadrightarrow}\mathbb{M}^{min}(\Sigma)
\stackrel{\mathrm{id}}{\to}\mathbb{M}^{min}(\Sigma)\hookrightarrow\mathrm{gr}_\Sigma(M)\to ABA'B'\C[A,B,A',B'], 
\end{equation}
compatible with the following sequence of morphisms of algebras, and their module structures over these algebras
\begin{equation}\label{map:a}
\C[\mathbf{x}_2,\mathbf{x}_6,\mathbf{x}_3,\mathbf{x}_5]\stackrel{\simeq}{\to}\overline{\mathrm{gr}_\Sigma(\mathbb{A})}
\stackrel{sec}{\to}\mathrm{gr}_\Sigma(\mathbb{A})\stackrel{\mathrm{id}}{\to}\mathrm{gr}_\Sigma(\mathbb{A})\to\C[A,B,A',B']. 
\end{equation}
The resulting composite morphism $\C[\mathbf{x}_2,\mathbf{x}_6,\mathbf{x}_3,\mathbf{x}_5]\to\C[A,B,A',B']$ is given by 
$$
\mathbf{x}_2\mapsto(\mathrm{class\ of\ }\dot\sigma_2)\mapsto\dot\sigma_2\mapsto\sigma_2(A,0,A',0), \quad 
\mathbf{x}_6\mapsto(\mathrm{class\ of\ }\dot\sigma_6)\mapsto\dot\sigma_6\mapsto\sigma_6(A,0,A',0),
$$
$$
\mathbf{x}_3\mapsto(\mathrm{class\ of\ }\dot\sigma_3)\mapsto\dot\sigma_3\mapsto 
{\partial\tilde\sigma_3\over\partial B}_{|B=B'=0}B+{\partial\tilde\sigma_3\over\partial B'}_{|B=B'=0}B'=\tilde\sigma_3^{lin}(A,A',B,B'), 
$$
$$
\mathbf{x}_5\mapsto(\mathrm{class\ of\ }\dot\sigma_5)\mapsto\dot\sigma_5\mapsto 
{\partial\tilde\sigma_5\over\partial B}_{|B=B'=0}B+{\partial\tilde\sigma_5\over\partial B'}_{|B=B'=0}B'=\tilde\sigma_5^{lin}(A,A',B,B'),  
$$
where $\tilde\sigma_i^{lin}$ are as in (\ref{def:sigma:lin}) and the equalities follow from the invariance of $\sigma_i(A,B,A',B')$ under the 
exchange $B\leftrightarrow A'$. 
Let $\mathbf{1}$ be the element of $\C[\mathbf{x}_2,\mathbf{x}_6,\mathbf{x}_3,\mathbf{x}_5]\{8,2\}$ corresponding to 
$1\in\C[\mathbf{x}_2,\mathbf{x}_6,\mathbf{x}_3,\mathbf{x}_5]$. There is a commutative diagram 
\begin{equation}\label{diag:pour:inject}
\xymatrix{\C[\mathbf{x}_2,\mathbf{x}_6,\mathbf{x}_3,\mathbf{x}_5]\ar[r]^{(a)} \ar[d]_{(b)}& 
\C[A,A',B,B'] \ar[d]^{(c)} \\ 
\C[\mathbf{x}_2,\mathbf{x}_6,\mathbf{x}_3,\mathbf{x}_5]\{8,2\}\ar[r]^{(d)} & AA'BB'\cdot \C[A,A',B,B'] }
\end{equation}
where (a), (d) are the composite maps (\ref{map:a}), (\ref{map:d}), (b) is the action map on $\mathbf{1}$ and (c) is the action map 
on the image of $\mathbf{1}$ by (d).  

We have $\mathrm{im}(\mathbf{1}\in\C[\mathbf{x}_2,\mathbf{x}_6,\mathbf{x}_3,\mathbf{x}_5]\{8,2\}\stackrel{(d)}{\to}
AA'BB'\C[A,A',B,B'])=\mathrm{im}(\dot\tau_{35}\in\mathrm{gr}_\Sigma^0(M)\to AA'BB'\C[A,A']\hookrightarrow AA'BB'\C[A,A',B,B'])$. One 
computes $\mathrm{im}(\dot\tau_{35}\in\mathrm{gr}_\Sigma^0(M)\to AA'BB'\C[A,A'])=30AA'BB'(A^2-A^{\prime2})(2A+A')(A+2A')$, 
therefore 
$$
\mathrm{im}(\mathbf{1}\in\C[\mathbf{x}_2,\mathbf{x}_6,\mathbf{x}_3,\mathbf{x}_5]\{8,2\}\stackrel{(d)}{\to}
AA'BB'\C[A,A',B,B'])=30AA'BB'(A^2-A^{\prime2})(2A+A')(A+2A'). 
$$
Since this element of $AA'BB'\C[A,A',B,B']$ is nonzero and since $\C[A,A',B,B']$ is a domain, 
\begin{equation}\label{inject:c}
\mathrm{the\ map\ (c)\ is\ injective.}  
\end{equation}
We also record: 
\begin{equation}\label{iso:b}
\mathrm{the\ map\ (b)\ is\ an\ isomorphism\ of\ bigraded\ vector\ spaces\ (of\ bidegree\ (8,2)).} 
\end{equation}

One computes 
$$
\tilde\sigma_2(A,0,A',0)=2(A^2+A^{\prime2}+(A+A')^2), \quad\tilde\sigma_6(A,0,A',0)=2(A^6+A^{\prime6}+(A+A')^6),  
$$
$$
\tilde\sigma_3^{lin}(A,A',B,B')=3((A+A')^2-A^2)B+3((A+A')^2-A^{\prime 2})B', 
$$
$$
\tilde\sigma_5^{lin}(A,A',B,B')=5((A+A')^4-A^4)B+5((A+A')^4-A^{\prime 4})B'. 
$$
The Jacobian matrix of the map 
$$
(A,A',B,B')\mapsto(\tilde\sigma_2(A,0,A',0),\tilde\sigma_6(A,0,A',0),\tilde\sigma_3^{lin}(A,A',B,B'),
\tilde\sigma_5^{lin}(A,A',B,B'))
$$ 
is then $\begin{pmatrix} {\partial(\tilde\sigma_2(A,0,A',0),\tilde\sigma_6(A,0,A',0))\over \partial(A,A')} 
& 0 \\ * & {\partial(\tilde\sigma_3^{lin}(A,A',B,B'),\tilde\sigma_5^{lin}(A,A',B,B'))\over\partial(B,B')} \end{pmatrix}$. 

One computes the determinants of the diagonal submatrices: 

$|{\partial(\tilde\sigma_2(A,0,A',0),\tilde\sigma_6(A,0,A',0))\over \partial(A,A')}|
=48AA'(A+A')(A-A')(2A+A')(A+2A')$, 

$|{\partial(\tilde\sigma_3^{lin}(A,A',B,B'),\tilde\sigma_5^{lin}(A,A',B,B'))\over\partial(B,B')}|=
-15AA'(A+A')(A-A')(2A+A')(A+2A')$, 

\noindent which implies that the Jacobian determinant of the above map is nonzero. This implies that the family 
$(\tilde\sigma_2(A,0,A',0),\tilde\sigma_6(A,0,A',0),\tilde\sigma_3^{lin}(A,B,A',B'),\tilde\sigma_5^{lin}(A,B,A',B'))$ 
of $\C[A,A',B,B']$ is algebraically independent. 

It follows that the map (a) is injective. Combining this with the isomorphism result (\ref{iso:b}) and the injectivity result (\ref{inject:c}), 
we obtain the injectivity of map (d) in (\ref{diag:pour:inject}).     

There is a factorisation of (d) starting with the morphism $\C[\mathbf{x}_2,\mathbf{x}_6,\mathbf{x}_3,\mathbf{x}_5]\{8,2\}
\stackrel{\phi}{\to}\mathbb{M}^{min}(\Sigma)$ (see (\ref{map:d})), therefore {\it $\phi$ is injective.}

Combining this result with the surjectivity of $\phi$ proved in Subsection \ref{subsect:inj}, we obtain: 

\begin{proposition}\label{prop:unique:iso}
There is a unique isomorphism of bigraded vector spaces 
$$
\phi:\C[\mathbf{x}_2,\mathbf{x}_6,\mathbf{x}_3,\mathbf{x}_5]\{8,2\}\to\mathbb{M}^{min}(\Sigma), 
$$
such that $\phi(\mathbf{1})=\dot\tau_{35}\in\mathbb{M}^{min}_0(\Sigma)$, which intertwines the actions of 
$\mathbf{x}_2,\mathbf{x}_6,\mathbf{x}_3,\mathbf{x}_5$ with the actions of $\dot\sigma_2,\dot\sigma_6,\dot\sigma_3,\dot\sigma_5$.

The bidegree (weight degree, $\Sigma$-degree) is defined on the left side as follows: 
$\mathbf{x}_2,\mathbf{x}_6,\mathbf{x}_3,\mathbf{x}_5$ have (weight degree, $\Sigma$-degree)
respectively $(2,0)$, $(6,0)$, $(3,1)$, $(5,1)$ and $\{8,2\}$ means shifting the bidegree by $(8,2)$. 
\end{proposition}

\subsection{The $S(\Sigma)$-module structure of ${\mathbb M}^{min}(\Sigma)$}\label{subsect:SSigma:str}

Lemma \ref{lemma:5:2} describes the action of $\Sigma(\hookrightarrow{\mathfrak{L}}_0)$ on $M\simeq{\mathfrak{L}}_1$. 
According to Subsection \ref{543}, there is a sequence of linear maps $\Sigma\to\mathbb{I}\hookrightarrow\mathbb{A}$, 
inducing an algebra morphism $S(\Sigma)\to\mathbb{A}$, and such that the $S(\Sigma)$-module structure of 
$M$ is the pullback of its $\mathbb{A}$-module structure. 

The linear map $\Sigma\to\mathbb{I}$ induces a linear map $\Sigma\to\mathbb{I}/\mathbb{I}^2=\mathrm{gr}_\Sigma^1(\mathbb{A})$
and therefore an algebra morphism $S(\Sigma)\to\mathrm{gr}_\Sigma(\mathbb{A})$. The $S(\Sigma)$-module structure on 
$\mathrm{gr}_\Sigma(M)$ is then the pullback of its $\mathrm{gr}_\Sigma(\mathbb{A})$-module structure. 
The $\mathrm{gr}_\Sigma(\mathbb{A})$-module structure on the subspace $\mathbb{M}^{min}(\Sigma)\subset\mathrm{gr}_\Sigma(M)$ 
is the pullback of a $\overline{\mathrm{gr}_\Sigma(\mathbb{A})}$-module structure. The $S(\Sigma)$-module structure of this space is 
therefore a pullback of this module structure under the algebra morphism $S(\Sigma)\to\overline{\mathrm{gr}_\Sigma(\mathbb{A})}$. 
The latter morphism is induced by the linear map $\Sigma\to\overline{\mathrm{gr}_\Sigma^1(\mathbb{A})}=\mathbb{I}/
(\mathbb{I}^2+\mathbb{I}\cap\mathbb{J})$, where $\mathbb{J}$ is the principal ideal of $\mathbb{A}$ generated by 
$\tilde\sigma_4-{1\over 4}(\tilde\sigma_2)^2$. 

Recall that $\sigma(t)=\sum_{k\mathrm{\ odd\ }\geq 3}\sigma_k t^k\in\Sigma[[t]]$ is a generating series for a basis of $\Sigma$. 
The map $\Sigma\to\mathbb{I}$ is given by $\sigma(t)\mapsto\sum_{k\mathrm{\ odd\ }\geq 3}\tilde\sigma_k t^k$. One computes 
\begin{align}\label{début:sigma}
& \nonumber \sum_{k\mathrm{\ odd\ }\geq 3}\tilde\sigma_k t^k=\sum_{k\mathrm{\ odd\ }\geq 3}(\sum_{\alpha,\beta=1}^3
X_{\alpha\beta}^k)t^k \quad (\mathrm{where\ }(X_{\alpha\beta})_{1\leq\alpha,\beta\leq 3} \mathrm{\ is\ defined\ by} 
\\& \nonumber(X_{11},X_{12},X_{21},X_{22}):=(A,B,A',B'), \quad  \sum_{\beta'=1}^3X_{\alpha\beta'}=\sum_{\alpha'=1}^3X_{\alpha'\beta}=0 
\mathrm{\ for\ } 1\leq\alpha,\beta\leq 3)
\\ & \nonumber=\sum_{\alpha,\beta=1}^3{1\over2}({tX_{\alpha\beta}\over1+tX_{\alpha\beta}}+{tX_{\alpha\beta}\over1-tX_{\alpha\beta}})
={1\over 2}t{d\over dt}\mathrm{log}{\prod_{\alpha,\beta=1}^3(1+tX_{\alpha\beta})\over\prod_{\alpha,\beta=1}^3(1-tX_{\alpha\beta})}
\\ &={1\over 2}t{d\over dt}\mathrm{log}{D(t)\over D(-t)}
\quad(\mathrm{where\ }D(t):=\prod_{\alpha,\beta=1}^3(1+tX_{\alpha\beta})). 
\end{align}
Since $\mathrm{log}D(t)=t\tilde\sigma_1-{t^2\over 2}\tilde\sigma_2+{t^3\over3}\tilde\sigma_3+\cdots$ and $D(t)$ has degree $9$, 
there holds $D(t)=\mathrm{exp}(t\tilde\sigma_1-{t^2\over 2}\tilde\sigma_2+\cdots-{t^9\over 9}\tilde\sigma_9)_{\leq 9}$, where the 
index $\leq 9$ denotes the projection $t^i\mapsto t^i$ if $1\leq i\leq 9$, $t^i\mapsto 0$ if $i>9$. Expressing $\tilde\sigma_7$,
$\tilde\sigma_8$ and $\tilde\sigma_9$ as polynomials in the generators $\tilde\sigma_2,\tilde\sigma_3,\ldots,\tilde\sigma_6$ 
of $\mathbb{A}$ using Maple, one obtains an expression of $D(t)$ terms of the generators of $\mathbb{A}$. 

Set $D_{ev}(t):={1\over 2}(D(t)+D(-t))$, $D_{odd}(t):={1\over 2}(D(t)-D(-t))$. Let $\mathbb{J}$ be the ideal of $\mathbb{A}$ generated
by $\tilde\sigma_4-{1\over 4}(\tilde\sigma_2)^2$. Set 
$$
D_{ev}^0(t):=(1-{\tilde\sigma_2\over 4}t^2)^2+(-{1\over 6}\tilde\sigma_6+{1\over96}(\tilde\sigma_2)^3)t^6, \quad
D_{odd}^0(t):={1\over3}\tilde\sigma_3 t^3+({1\over5}\tilde\sigma_5-{1\over 6}\tilde\sigma_2\tilde\sigma_3)t^5, 
$$
Then $D_{odd}^0(t)\in \mathbb{I}[t]$ and one checks that  
$$
D_{ev}(t)\equiv D_{ev}^0(t) \quad \mathrm{mod}\quad (\mathbb{J}+\mathbb{I}^2)[t], 
\quad
D_{odd}(t)\equiv D_{odd}^0(t) \quad \mathrm{mod}\quad (\mathbb{IJ}+\mathbb{I}^2)[t]. 
$$
The first statement, together with the invertibility of $D_{ev}^0(t)$ in $\mathbb{A}[[t]]$, implies that 
${D_{ev}(t)\over D_{ev}^0(t)}\in 1+(\mathbb{J}+\mathbb{I}^2)[[t]]$. The second statement, together with 
$D_{odd}^0(t)\in \mathbb{I}[t]$, then implies that 
$$
{D_{odd}(t)\over D_{ev}(t)}\equiv{D_{odd}^0(t)\over D_{ev}^0(t)} \quad \mathrm{mod}\quad(\mathbb{IJ}+\mathbb{I}^2)[[t]]. 
$$
This implies 
$$
{D(t)\over D(-t)}={D_{ev}(t)+D_{odd}(t)\over D_{ev}(t)-D_{odd}(t)}\equiv
1+2{D_{odd}^0(t)\over D_{ev}^0(t)}
 \quad \mathrm{mod}\quad(\mathbb{IJ}+\mathbb{I}^2)[[t]]. 
$$
By (\ref{début:sigma}), $\sum_{k\mathrm{\ odd\ }\geq 3}\tilde\sigma_k t^k={1\over 2}t{d\over dt}\mathrm{log}
{D(t)\over D(-t)}$, therefore 
$$
\sum_{k\mathrm{\ odd\ }\geq 3}\tilde\sigma_k t^k\equiv
t{d\over dt}({D_{odd}^0(t)\over D_{ev}^0(t)})
 \quad \mathrm{mod}\quad(\mathbb{IJ}+\mathbb{I}^2)[[t]]. 
$$

The map $\Sigma\to\mathbb{I}/(\mathbb{I}^2+\mathbb{I}\cap\mathbb{J})=\overline{\mathrm{gr}_\Sigma^1(\mathbb{A})}$ 
is then given by 
$$
\sum_{k\mathrm{\ odd\ }\geq 3}\sigma_k t^k\mapsto (\mathrm{class\ of\ }t{d\over dt}({D_{odd}^0(t)\over D_{ev}^0(t)}))
=(\mathrm{class\ of\ }\tilde\sigma_3\cdot t{d\over dt}({{1\over 3}t^3-{1\over 6}\tilde\sigma_2t^5\over D_{ev}^0(t)})
+\tilde\sigma_5\cdot t{d\over dt}({{1\over 5}t^5\over D_{ev}^0(t)})). 
$$
Composing it with the isomorphism $\overline{\mathrm{gr}_\Sigma^1(\mathbb{A})}\simeq\overline{\mathrm{gr}_\Sigma^0(\mathbb{A})}\otimes
(\C\mathbf{x}_3\oplus\C\mathbf{x}_5)$, we obtain the map 
$$
\Sigma\to\overline{\mathrm{gr}_\Sigma^0(\mathbb{A})}\otimes(\C\mathbf{x}_3\oplus\C\mathbf{x}_5)
$$
$$
\sum_{k\mathrm{\ odd\ }\geq 3}\sigma_k t^k\mapsto  (\mathrm{class\ of\ }t{d\over dt}({{1\over 3}t^3-{1\over 6}\tilde\sigma_2t^5\over D_{ev}^0(t)}))\otimes\mathbf{x}_3+(\mathrm{class\ of\ }t{d\over dt}({{1\over 5}t^5\over D_{ev}^0(t)}))\otimes\mathbf{x}_5
$$
Composing the latter map with the isomorphism induced by $\overline{\mathrm{gr}_\Sigma^0(\mathbb{A})}\simeq
\C[\mathbf{x}_2,\mathbf{x}_6]$, and using (class of $D_{ev}^0(t))=d(\mathbf{x}_2,\mathbf{x}_6,t)$ (see (\ref{def:d})), 
we obtain the map 
$$
\Sigma\to\C[\mathbf{x}_2,\mathbf{x}_6]\otimes(\C\mathbf{x}_3\oplus\C\mathbf{x}_5), 
$$
$$
\sigma(t)\mapsto t{d\over dt}({{1\over 3}t^3-{1\over 6}\mathbf{x}_2t^5\over d(\mathbf{x}_2,\mathbf{x}_6,t)})\otimes\mathbf{x}_3
+t{d\over dt}({{1\over 5}t^5\over d(\mathbf{x}_2,\mathbf{x}_6,t)})\otimes\mathbf{x}_5. 
$$
We derive from there: 
\begin{proposition}\label{prop:comp:module}
Under the inverse of the isomorphism $\phi:\C[\mathbf{x}_2,\mathbf{x}_6,\mathbf{x}_3,\mathbf{x}_5]\{8,2\}\to\mathbb{M}^{min}(\Sigma)$, 
the action of $\sigma(t)\in\Sigma[[t]]$ on $\mathbb{M}^{min}(\Sigma)$ corresponds to multiplication by 
$$
\xi_3(\mathbf{x}_2,\mathbf{x}_6,t)\cdot\mathbf{x}_3+\xi_5(\mathbf{x}_2,\mathbf{x}_6,t)
\cdot(\mathbf{x}_5-{5\over 6}\mathbf{x}_2\mathbf{x}_3)
$$
on $\C[\mathbf{x}_2,\mathbf{x}_6,\mathbf{x}_3,\mathbf{x}_5]\{8,2\}$, where $\xi_3(\mathbf{x}_2,\mathbf{x}_6,t),
\xi_5(\mathbf{x}_2,\mathbf{x}_6,t)$ are given by (\ref{xi3:xi5}). 
\end{proposition} 

\subsection{The depth filtration of ${\mathbb M}^{min}(\Sigma)$}\label{subsect:depth:M:min}

It follows from Subsection \ref{subsect:depth} that the depth filtration on $M$ is given by 
$M=F^2_{dpth}(M)\supset F^3_{dpth}(M)\supset\cdots$, where $F^{k+2}_{dpth}(M)=ABA'B'\cdot\{(B,B')^k\cap\C[A,B,A',B']^{as}\}$
and $(B,B')\subset\C[A,B,A',B']$ is the ideal generated by $B,B'$. This filtration induces filtrations on the subspaces 
$F^k_\Sigma(M):=S^k(\Sigma)\cdot M$, on the quotients $\mathrm{gr}_\Sigma^k(M)=F^k_\Sigma(M)/F^{k+1}_\Sigma(M)$, and
on the subspaces $\mathbb{M}_k^{min}(\Sigma)\subset\mathrm{gr}_\Sigma^k(M)$. 

Since $\Sigma$ is contained in the intersection $(B,B')\cap \C[A,B,A',B']^{sym}$, we have an inclusion $F^k_\Sigma(M)\subset 
F^{k+2}_{dpth}(M)$ for each $k\geq 0$. The collection of these inclusions gives rise to a linear map 
\begin{equation}\label{lin:map:k}
\mathrm{gr}_\Sigma^k(M)\to \mathrm{gr}_{dpth}^{k+2}(M)
\end{equation}
for any $k\geq 0$. 

The depth filtration of $F^k_\Sigma(M)$ is induced by the depth filtration of $M$. The inclusion $F^{k+2}_{dpth}(M)\supset F^k_\Sigma(M)$
implies that $F^k_\Sigma(M)$ coincides with its part of depth degree $\geq k+2$, so $F^k_\Sigma(M)=F^{k+2}_{dpth}(F^k_\Sigma(M))$.
It follows that the same is true of $\mathrm{gr}^k_\Sigma(M)$, so 
\begin{equation}\label{grkM:depth:k+2}
\mathrm{gr}^k_\Sigma(M)=F^{k+2}_{dpth}(\mathrm{gr}^k_\Sigma(M)).
\end{equation}
The part of $\mathrm{gr}^k_\Sigma(M)$ of depth degree $\geq k+3$ is defined as the image of 
$F^{k+3}_{dpth}(F^k_\Sigma(M))\hookrightarrow F^k_\Sigma(M)\to\mathrm{gr}^k_\Sigma(M)$. 
The commutative diagram 
$$
\xymatrix{
F^{k+3}_{dpth}(\mathrm{gr}_\Sigma^k(M)) \ar@{^{(}->}[r] & \mathrm{gr}_\Sigma^k(M) \ar[r] & \mathrm{gr}_{dpth}^{k+2}(M)
 \\ 
F^{k+3}_{dpth}(F_\Sigma^k(M))
\ar@{>>}[u]\ar@{^{(}->}[rr] &  & F^{k+3}_{dpth}(M)\ar[u]^{0}}
$$
implies that the composition $F^{k+3}_{dpth}(\mathrm{gr}_\Sigma^k(M)) \hookrightarrow\mathrm{gr}_\Sigma^k(M) \to
\mathrm{gr}_{dpth}^{k+2}(M)$ is zero, so that 
\begin{equation} \label{inclusion:kernel}
F^{k+3}_{dpth}(\mathrm{gr}_\Sigma^k(M)) \mathrm{\ is\ contained\ in\ the\ kernel\ of\ } 
\mathrm{gr}_\Sigma^k(M) \to\mathrm{gr}_{dpth}^{k+2}(M).
\end{equation}  

The direct sum over $k\geq 0$ of the linear maps (\ref{lin:map:k}) is a linear map 
$$
\mathrm{gr}_\Sigma(M)\simeq\oplus_{k\geq 0}\mathrm{gr}^k_\Sigma(M)\to\oplus_{k\geq 0}\mathrm{gr}^{k+2}_{depth}(M). 
$$
Composing it with the inclusion $\mathbb{M}^{min}(\Sigma)\hookrightarrow\mathrm{gr}_\Sigma(M)$ and with the isomorphism 
$\mathbb{M}^{min}(\Sigma)\simeq\C[\mathbf{x}_2,\mathbf{x}_6,\mathbf{x}_3,\mathbf{x}_5]\{8,2\}$, with the isomorphism 
$\oplus_{k\geq 0}\mathrm{gr}^{k+2}_{depth}(M)\simeq M$ arising from the fact that the depth filtration of $M$ comes from a grading, 
and with the inclusion $M\hookrightarrow ABA'B'\cdot \C[A,B,A',B']$, we obtain a linear map 
$$
\C[\mathbf{x}_2,\mathbf{x}_6,\mathbf{x}_3,\mathbf{x}_5]\{8,2\}\to ABA'B'\cdot \C[A,B,A',B']. 
$$
One can check that this map coincides with map (\ref{map:a}), which has been shown (proof of Proposition \ref{prop:unique:iso})  
to be injective. It follows that for any $k\geq 0$, the map $\mathbb{M}^{min}_k(\Sigma)\hookrightarrow\mathrm{gr}^k_\Sigma(M)
\to\mathrm{gr}^{k+2}_{depth}(M)$ is injective. Comparing this with (\ref{inclusion:kernel}), we derive that {\it the intersection 
$\mathbb{M}^{min}_k(\Sigma)\cap F^{k+3}_{dpth}(\mathrm{gr}_\Sigma^k(M))$ is zero.}  On the other hand, (\ref{grkM:depth:k+2}) 
implies that $\mathbb{M}^{min}_k(\Sigma)$ is contained in $F^{k+2}_{dpth}(\mathrm{gr}_\Sigma^k(M))$. All this implies: 

\begin{proposition}
For each $k\geq 0$, $\mathbb{M}^{min}_k(\Sigma)$ is pure for the depth filtration, of depth $k+2$. 
\end{proposition}

\subsection{Summary of the results on lower bounds for $\Sigma$-structures}\label{subsect:summary}

\subsubsection{} Recall that $\mathfrak{grt}_1$ is equipped with a weight grading and a compatible depth filtration (see Subsection \ref{subsect:depth}). The spaces derived from $\mathfrak{grt}_1$ are likewise equipped with compatible degrees and filtrations. 

\subsubsection{} There is an isomorphism of graded spaces $\mathrm{gr}_{lcs}^0(\mathfrak{grt}_1)\simeq
\oplus_{k\mathrm{\ odd\ }\geq 3}\C\sigma_k(=\Sigma)$, and $\Sigma$ is pure of depth degree 1. The Lie bracket of $\mathfrak{grt}_1$ gives 
rise to a $S(\Sigma)$-module structure over $\mathrm{gr}_{lcs}^1(\mathfrak{grt}_1)$ and to a linear map 
$\Lambda^2(\Sigma)\to\mathrm{gr}_{lcs}^1(\mathfrak{grt}_1)/\Sigma\cdot\mathrm{gr}_{lcs}^1(\mathfrak{grt}_1)$. 

\subsubsection{} Set $\mathrm{gr}_\Sigma^i(\mathrm{gr}_{lcs}^1(\mathfrak{grt}_1)):=S^i(\Sigma)\cdot\mathrm{gr}_{lcs}^1(\mathfrak{grt}_1)/
S^{i+1}(\Sigma)\cdot\mathrm{gr}_{lcs}^1(\mathfrak{grt}_1)$, then 
$$
\mathrm{gr}_\Sigma(\mathrm{gr}_{lcs}^1(\mathfrak{grt}_1)):=
\oplus_{i\geq 0}\mathrm{gr}_\Sigma^i(\mathrm{gr}_{lcs}^1(\mathfrak{grt}_1))
$$ 
is a $\Sigma$-module in the sense of the Introduction. In addition to its grading corresponding to this decomposition (the $\Sigma$-grading), 
this module is equipped with a weight grading and a depth filtration, which are all compatible. 

\subsubsection{} The variables $\mathbf{x}_2,\mathbf{x}_3,\mathbf{x}_5,\mathbf{x}_6$ are free commutative variables of (weight degree, 
$\Sigma$-degree) equal to (2,0), (3,1), (5,1), (6,0), respectively. Polynomials and formal series are given by 
$$
d(\mathbf{x}_2,\mathbf{x}_6,t):=(1-{\mathbf{x}_2\over 4}t^2)^2+(-{1\over 6}\mathbf{x}_6
+{1\over96}(\mathbf{x}_2)^3)t^6, 
$$ 
$$
\xi_3({\mathbf{x}}_2,{\mathbf{x}}_6,t):=t{d\over dt}({-{1\over 3}t^3\over d(\mathbf{x}_2,\mathbf{x}_6,t)}), \quad 
\xi_5({\mathbf{x}}_2,{\mathbf{x}}_6,t):=t{d\over dt}({-{1\over 5}t^5\over d(\mathbf{x}_2,\mathbf{x}_6,t)}). 
$$
The (weight degree, $\Sigma$-degree)-bigraded vector space $\mathbb{M}^{min}(\Sigma)$ is equal to 
$\C[\mathbf{x}_2,\mathbf{x}_3,\mathbf{x}_5,\mathbf{x}_6]\{8,2\}$, where $\{8,2\}$ denotes the shift of bidegrees 
by $(8,2)$; so $\mathbb{M}_0^{min}(\Sigma)=\C[\mathbf{x}_2,\mathbf{x}_6]\{8\}$, 
$\mathbb{M}_1^{min}(\Sigma)=\C[\mathbf{x}_2,\mathbf{x}_6]\{8\}\otimes(\C\mathbf{x}_3\oplus\C\mathbf{x}_5)$, etc. 
A $S(\Sigma)$-module structure is defined on $\mathbb{M}^{min}(\Sigma)$ by the condition that $\sigma(t)=\sum_{k
\mathrm{\ odd\ }\geq 3}\sigma_kt^k$ acts by multiplication by 
$$
\xi_3(\mathbf{x}_2,\mathbf{x}_6,t)\cdot\mathbf{x}_3+\xi_5(\mathbf{x}_2,\mathbf{x}_6,t)
\cdot(\mathbf{x}_5-{5\over 6}\mathbf{x}_2\mathbf{x}_3). 
$$
A linear map $\Lambda^2(\Sigma)\to\mathbb{M}^{min}_0(\Sigma)$
is defined by  
$$
\{\sigma(t),\sigma(u)\}=-2(\xi_3({\mathbf{x}}_2,{\mathbf{x}}_6,t)\xi_5({\mathbf{x}}_2,{\mathbf{x}}_6,u)
-\xi_3({\mathbf{x}}_2,{\mathbf{x}}_6,u)\xi_5({\mathbf{x}}_2,{\mathbf{x}}_6,t)). 
$$
Equipped with these structures, $\mathbb{M}^{min}(\Sigma)$ is a $\Sigma$-module. Its depth filtration 
is such that 
$$
F^i_{dpth}(\mathbb{M}^{min}(\Sigma))=\mathrm{(part\ of\ }\mathbb{M}^{min}(\Sigma)\mathrm{\ of\ }\Sigma\text{-degree\ }
\geq i)=\mathbb{M}^{min}_{i-2}(\Sigma)\oplus\mathbb{M}^{min}_{i-1}(\Sigma)\oplus\cdots.
$$ 
\begin{theorem}\label{thm:Mmin} 
The depth-filtered $\Sigma$-module $\mathbb{M}^{min}(\Sigma)$ is a subquotient of the depth-filtered 
$\Sigma$-module $\mathrm{gr}_\Sigma(\mathrm{gr}_{lcs}^1(\mathfrak{grt}_1))$. 
\end{theorem}

\section{Lower bounds for Hilbert-Poincar\'e series of subquotients of $\mathfrak{grt}_1$}\label{Section:comp:hilbert}

As remarked in Subsection \ref{subsection:GF}, Theorem \ref{thm:Mmin} implies lower bound results for Hilbert-Poincar\'e series. 
For $X$ a vector space with weight grading $X=\oplus_{n\in\N}X[n]$ and compatible depth filtration, we have 
\begin{equation}\label{formula:for:P}
P_X(t,u)=\sum_{n,i\geq 0}\mathrm{dim}\ \mathrm{gr}_{dpth}^iX[n]t^nu^i, \quad
P_X(t)=\sum_{n\geq 0}\mathrm{dim}\ X[n]t^n(=P_X(t,1)). 
\end{equation}
Then 
$$
P_{\mathbb{M}^{min}(\Sigma)}(t,u)={t^8u^2\over{(1-t^2)(1-t^6)(1-ut^3)(1-ut^5)}}, \quad
P_{\mathbb{M}^{min}(\Sigma)}(t)={t^8\over{(1-t^2)(1-t^6)(1-t^3)(1-t^5)}}. 
$$
On the other hand, the 2-variable series of $\mathrm{gr}_{lcs}^1(\mathfrak{grt}_1)$ and of  $\mathrm{gr}_\Sigma(
\mathrm{gr}_{lcs}^1(\mathfrak{grt}_1))$ coincide, as well as the 1-variable series of the same spaces. Theorem \ref{thm:Mmin}
then implies: 

\begin{theorem}\label{thm:lower:bound}
The following inequalities hold
$$
P_{\mathrm{gr}_{lcs}^1(\mathfrak{grt}_1)}(t,u)\geq {t^8u^2\over{(1-t^2)(1-t^6)(1-ut^3)(1-ut^5)}}, \quad
P_{\mathrm{gr}_{lcs}^1(\mathfrak{grt}_1)}(t)\geq {t^8\over{(1-t^2)(1-t^6)(1-t^3)(1-t^5)}},
$$
where an inequality means that the difference of both sides is a series with nonnegative coefficients, and the left-hand sides 
are defined by (\ref{formula:for:P}). 
\end{theorem}

\section{On the explicit form of $\mathrm{gr}_{dpth}(\mathfrak{grt}_1)$ in depths 2 and 3} \label{Section:depth3}

The depth filtration of $\overline{\mathfrak{L}}$ (Section \ref{subsect:depth}) induces a decreasing filtration on $\mathfrak{grt}_1$, also called the depth 
filtration. The associated graded Lie algebra $\mathrm{gr}_{dpth}(\mathfrak{grt}_1)$ is a (weight, depth)-bigraded Lie algebra. 
The Broadhurst-Kreimer conjecture (\cite{BK}) predicts the dimensions of the components of this bigraded vector space. 
It has been proved for depths 1,2,3 (\cite{G}). We recall this material in Subsection \ref{subsect:material}.  

The explicit form of $\mathrm{gr}_{dpth}^k(\mathfrak{grt}_1)$ is known for depths $k=1,2$ (\cite{Ec}). Using the known results on 
$\mathbb{M}^{min}(\Sigma)$, we recover this explicit form (Subsection \ref{subsect:depth2}) and prove some results on the explicit form 
of $\mathrm{gr}_{dpth}^k(\mathfrak{grt}_1)$ when $k=3$ (Subsection \ref{subsect:depth3}). 

\subsection{\ } \label{subsect:material}

In Subsection \ref{sec:structure}, we introduced the free Lie algebra $\tilde{\mathfrak{L}}={\mathbb{L}}({\mathbb{C}}x\oplus{\mathbb{C}}y)$, 
equipped with the bracket $\langle,\rangle$. As remarked in Subsection \ref{subsect:depth}, this bracket (as well as the free Lie algebra 
bracket $[,]$) is homogeneous for the depth degree (i.e., the degree in $y$), therefore $\underline{\mathfrak{L}}:=
\oplus_{i>0}(\text{part\ of\ }\tilde{\mathfrak{L}}\text{\ of\ }y\text{-degree\ }i)$ is a graded Lie subalgebra of $\tilde{\mathfrak{L}}$ for 
the brackets $\langle,\rangle$ and $[,]$. 

For $a\geq 0$, set $\xi[a]:=(\mathrm{ad}x)^a(y)$ ($\in\underline{\mathfrak{L}}$), and 
$$
\underline{V}:=\oplus_{a\geq 0}\mathbb{C}\xi[a]\quad(\subset\underline{\mathfrak{L}}). 
$$
Then $\underline{V}\simeq\underline{\mathfrak{L}}[1]$ (where for $k\geq 1$, $\underline{\mathfrak{L}}[k]$ denotes the part of $\underline{\mathfrak{L}}$
of $y$-degree $k$), and it follows from Lazard's elimination theorem that the Lie bracket $[,]$ induces an isomorphism 
$$
{\mathbb{L}}(\underline{V})\simeq\underline{\mathfrak{L}}
$$ 
of depth-graded Lie algebras between $(\underline{\mathfrak{L}},[,])$ and the free Lie algebra generated by $\underline{V}$, where the  
degree on ${\mathbb{L}}(\underline{V})$ is defined by the condition that $\underline{V}$ has depth degree 1.   

The depth filtration on $\underline{\mathfrak{L}}$ is defined by $F^i_{dpth}(\underline{\mathfrak{L}}):=\oplus_{j|j\geq i}
\underline{\mathfrak{L}}[j]$; since this filtration is induced by a grading, $\underline{\mathfrak{L}}$ coincides with its associated 
graded Lie algebra for this filtration. This filtration is also compatible with the Lie algebra inclusion 
$\underline{\mathfrak{L}}\subset\tilde{\mathfrak{L}}$. 

The Lie algebra inclusion ${\mathfrak{grt}}_1\subset\underline{{\mathfrak{L}}}$ and the depth filtration on $\underline{{\mathfrak{L}}}$ induces 
a filtration on ${\mathfrak{grt}}_1$, which is again called the depth filtration. The associated graded Lie algebra is $\mathrm{gr}_{dpth}
({\mathfrak{grt}}_1)=\oplus_{k\geq 1}\mathrm{gr}_{dpth}^k({\mathfrak{grt}}_1)$, where 
$$
\mathrm{gr}_{dpth}^k({\mathfrak{grt}}_1):={{\mathfrak{grt}}_1\cap F_{dpth}^k(\underline{{\mathfrak{L}}})
\over {\mathfrak{grt}}_1\cap F_{dpth}^{k+1}(\underline{{\mathfrak{L}}})}\hookrightarrow \underline{\mathfrak{L}}[k]
\simeq\mathbb{L}_k(\underline{V}). 
$$
 
The map $\mathrm{gr}_{dpth}^k({\mathfrak{grt}}_1)\hookrightarrow\mathbb{L}_k(\underline{V})$ is an inclusion of graded (for the $x$-degree) vector spaces, and 
as mentioned above, the Broadhurst-Kreimer conjecture predicts the graded dimension of these spaces for each $k\geq1$.  They have been established for $k=1,2,3$. 
More precisely, the following results have been shown (\cite{G}). Set 
$$
\underline{W}:=\oplus_{a\mathrm{\ even\ }>0}{\mathbb{C}}\xi[a]\subset\underline{V}.
$$
Let $\mathrm{Lie}(\underline{W})\subset{\mathbb{L}}(\underline{V})$ be the Lie subalgebra (for the bracket $\langle,\rangle$) 
generated by $\underline W$, and let 
$\mathrm{Lie}(\underline W)[k]\subset{\mathbb{L}}_k(\underline{V})$ be its part of depth degree $k$. Then: 
\begin{enumerate}
\item for $k=1,2,3$, $\mathrm{Lie}(\underline W)[k]=\mathrm{gr}_{dpth}^k({\mathfrak{grt}_1})$, 
\item the Hilbert-Poincar\'e series of $\mathrm{gr}_{dpth}^k({\mathfrak{grt}_1})$ is equal to 
${t^3\over 1-t^2}$ for $k=1$, to ${t^8\over(1-t^2)(1-t^6)}$ for $k=2$, and to ${t^{11}(1+t^2-t^4)\over(1-t^2)(1-t^4)(1-t^6)}$
for $k=3$. 
\end{enumerate}

\subsection{Computation of $\mathrm{gr}_{dpth}^2(\mathfrak{grt}_1)$}\label{subsect:depth2}
There is a commutative diagram 
$$
\xymatrix{
 & \Sigma^{\otimes2}\ar@{^{(}->}[r]\ar[d]\ar[dl]& {\mathfrak{L}}_0^{\otimes2}\ar@{->>}[r]\ar[d]&\mathrm{gr}^1_{dpth}({\mathfrak{L}}_0)^{\otimes2}\ar[d]& 
\mathrm{gr}^1_{dpth}(\overline{\mathfrak{L}})^{\otimes2}\ar[d]\ar[l]& 
\ar[l]\ar[d] \mathrm{gr}_{dpth}^1(\mathfrak{grt}_1)^{\otimes2}& 
\ar@{->>}[l]\ar[d]\mathfrak{grt}_1^{\otimes2}
\\
\mathbb{M}_0\ar@{^{(}->}[r]& {{\mathfrak{L}}_1\over{\Sigma\cdot{\mathfrak{L}}_1}}\ar@{->>}[r]&  {{\mathfrak{L}}_1\over{{\mathfrak{L}}_0\cdot{\mathfrak{L}}_1}}\ar[r]& \mathrm{gr}^2_{dpth}(\mathfrak{L}_1)& 
\mathrm{gr}^2_{dpth}(\overline{\mathfrak{L}}^1)\ar[l]
&
\mathrm{gr}^2_{dpth}(\mathfrak{grt}_1)\ar[l]& 
F^2_{dpth}(\mathfrak{grt}_1)\ar@{->>}[l]}
$$
where the first square arises from the inclusion $\Sigma\subset{\mathfrak{L}}_0$, the second square from the compatiblity of the operations of ${\mathfrak{L}}$ with 
depth, the fifth square from compatiblity of the Lie bracket of $\mathfrak{grt}_1$ with depth, the fourth square from the Lie algebra inclusion 
$\mathfrak{grt}_1\subset\overline{\mathfrak{L}}$, and the third square from taking associated graded for the l.c.s. filtration. 

The isomorphism ${\mathfrak{L}}_0\simeq AB\C[A,B]$ is given by $(\mathrm{class\ of\ }(\mathrm{ad}x)^k(\mathrm{ad}y)^l([x,y]))\leftrightarrow
A^{k+1}B^{l+1}$. It gives rise to an isomorphism $\mathrm{gr}_{dpth}({\mathfrak{L}}_0)\simeq A\C[A]B$. This space is isomorphic to the subspace
$\underline{V}_+:=\oplus_{a>0}\C\xi[a]$ of $\underline{V}$, where the bijections are $(\mathrm{class\ of\ }(\mathrm{ad}x)^{k+1}(y))
\leftrightarrow A^{k+1}B\leftrightarrow\xi[k+1]$ for $k\geq 0$. We have therefore 
\begin{equation} \label{iso:gr:L0}
\mathrm{gr}_{dpth}({\mathfrak{L}}_0)\simeq A\C[A]B\simeq\underline{V}_+. 
\end{equation}
The isomorphisms ${\mathfrak{L}}_1\simeq\mathbb{L}_2({\mathfrak{L}}_0)\simeq\mathbb{L}_2(AB\C[A,B])$ gives rise to an isomorphism 
\begin{equation} \label{iso:gr:L1}
\mathrm{gr}_{dpth}^2({\mathfrak{L}}_1)\simeq\mathbb{L}_2(A\C[A]B)\simeq ABA'B'\C[A,A']^{as}\simeq\mathbb{L}_2(\underline{V}_+), 
\end{equation}
where the exponent $as$ means antisymmetry in $A,A'$, and where the bijections are 
$$
(\text{class of }[(\mathrm{ad}x)^{k+1}(y),(\mathrm{ad}x)^{l+1}(y)])\leftrightarrow[A^{l+1}B,A^{l+1}B]\leftrightarrow
ABA'B'(A^k(A')^l-A^l(A')^k)\leftrightarrow[\xi[k+1],\xi[l+1]]
$$
for $k,l\geq 0$. 
One checks that the natural maps $\Sigma\to\mathrm{gr}_{dpth}^1({\mathfrak{L}}_0)$ and $\mathbb{M}_0^{min}(\Sigma)\to\mathrm{gr}_{dpth}^1({\mathfrak{L}}_1)$
are injective. Moreover, there is a decomposition $\mathbb{L}_2(\underline{V})=[\xi[0],\underline{V}_+]\oplus\mathbb{L}_2(\underline{V}_+)$. 
As $F^2_{dpth}(\mathfrak{grt}_1)$ is contained in $\overline{\mathfrak{L}}^1$, its projection on $[\xi[0],\underline{V}_+]$ is zero, which 
implies that $\mathrm{gr}_{dpth}^2(\mathfrak{grt}_1)$ is contained in $\mathbb{L}_2(\underline{V}_+)\simeq\mathrm{gr}_{dpth}^2({\mathfrak{L}}_1)$. 
The above diagram then gives rise to a diagram 
$$
\xymatrix{
\Sigma^{\otimes 2}\ar@/^1.5pc/[rr]^{\simeq}
\ar@{^{(}->}[r]\ar_{\{,\}}@{->>}[d]&\mathrm{gr}_{dpth}^1({\mathfrak{L}}_0)^{\otimes2}\ar[d] 
& \mathrm{gr}_{dpth}^1(\mathfrak{grt}_1)^{\otimes2}  \ar@{_{(}->}[l]\ar^{bracket}[d]\\
\mathbb{M}_0^{min}(\Sigma)\ar@{^{(}->}[r]& \mathrm{gr}_{dpth}^2({\mathfrak{L}}_1)& \mathrm{gr}_{dpth}^2(\mathfrak{grt}_1)\ar@{_{(}->}[l]
}
$$
where the surjectivity of $\Sigma^{\otimes 2}\to\mathbb{M}_0^{min}(\Sigma)$ follows from the construction of $\mathbb{M}^{min}(\Sigma)$
and the surjectivity of $\mathrm{gr}_{dpth}^1(\mathfrak{grt}_1)^{\otimes2}\to\mathrm{gr}_{dpth}^2(\mathfrak{grt}_1)$ follows from the 
results recalled in Subsection \ref{subsect:material}. This diagram implies the equality 
\begin{equation}\label{iso:depth:2}
\mathbb{M}_0^{min}(\Sigma)=\mathrm{gr}_{dpth}^2(\mathfrak{grt}_1). 
\end{equation}
It follows that {\it the subspace $\mathrm{gr}_{dpth}^2(\mathfrak{grt}_1)$ of $\mathbb{L}_2(\underline{V}_+)$ coincides, under the identification 
$\mathbb{L}_2(\underline{V}_+)\simeq ABA'B'\C[A,A']^{as}$ induced by $[\xi[k+1],\xi[l+1]]\leftrightarrow 
ABA'B'(A^kA^{\prime l}-A^lA^{\prime k})$, with the subspace $ABA'B'(A-A')(A+A')(A+2A')(2A+A')\C[A^2+A^{\prime 2}+(A+A')^2,
A^6+A^{\prime 6}+(A+A')^6]$} (see \cite{Ec}; the kernel of the bracket $\mathrm{gr}^1_{dpth}(\mathfrak{grt}_1)^{\otimes 2}\to\mathrm{gr}^2_{dpth}(\mathfrak{grt}_1)$ was studied first in \cite{Sc}). 

\subsection{Computation of $\mathrm{gr}_{dpth}^3(\mathfrak{grt}_1)$}\label{subsect:depth3}

\subsubsection{} The isomorphism ${\mathfrak{L}}_1\simeq\mathbb{L}(AB\C[A,B])$ gives rise to an isomorphism
$\mathrm{gr}_{dpth}^3({\mathfrak{L}}_1)\simeq(\text{the part}$ $\text{of }\mathbb{L}(AB\C[A,B])\text{ of degree 3 in }B)$, therefore 
\begin{equation} \label{iso:gr:L13}
\mathrm{gr}_{dpth}^3({\mathfrak{L}}_1)\simeq[AB\C[A],AB^2\C[A]]\simeq ABA'(B')^2\C[A,A']\simeq [\underline{V}_+,[\xi[0],\underline{V}_+]], 
\end{equation}
where the third space is viewed as a subspace of $\mathbb{L}_2(\oplus_{i\geq0}(\mathrm{ad}\xi[0])^i(\underline{V}_+))$; the bijections 
are given by $(\text{class of }[(\mathrm{ad}x)^{k+1}(y),(\mathrm{ad}y)(\mathrm{ad}x)^{l+1}(y)])
\leftrightarrow[A^{k+1}B,A^{l+1}B^2]\leftrightarrow A^{k+1}B(A')^{l+1}(B')^2\leftrightarrow[\xi[k+1],[\xi[0],\xi[l+1]]]$ for $k,l\geq 0$.

The adjoint action of ${\mathfrak{L}}_0$ on ${\mathfrak{L}}_1$ is compatible with depth, therefore it gives rise to a linear map 
$\mathrm{gr}_{dpth}^1({\mathfrak{L}}_0)\otimes\mathrm{gr}_{dpth}^2({\mathfrak{L}}_1)\to\mathrm{gr}_{dpth}^3({\mathfrak{L}}_1)$. By transport of structure via the 
isomorphisms (\ref{iso:gr:L0}), (\ref{iso:gr:L1}) and (\ref{iso:gr:L13}), we obtain a map 
$$
\mu:A\C[A]B\otimes ABA'B'\C[A,A']^{as}\to ABA'(B')^2\C[A,A']. 
$$
Lemma \ref{lemme:actions} then implies that this map is given by 
$$
f(A)B\otimes ABA'B'\cdot g(A,A')\mapsto ABA'(B')^2 (f(A')-f(A+A'))g(A,A'). 
$$

\subsubsection{} For any integer $k\geq 0$, there is a natural sequence of linear maps 
$$
\mathbb{M}_k^{min}(\Sigma)\hookrightarrow {S^k(\Sigma)\cdot{\mathfrak{L}}_1\over S^{k+1}(\Sigma)\cdot{\mathfrak{L}}_1}\to
{S^k({\mathfrak{L}}_0)\cdot{\mathfrak{L}}_1\over S^{k+1}({\mathfrak{L}}_0)\cdot{\mathfrak{L}}_1}\to\mathrm{gr}_{dpth}^{k+2}({\mathfrak{L}}_1), 
$$
where the last map is induced by the inclusions $S^k({\mathfrak{L}}_0)\cdot{\mathfrak{L}}_1\subset F^{k+2}_{dpth}({\mathfrak{L}}_1)$, 
$S^{k+1}({\mathfrak{L}}_0)\cdot{\mathfrak{L}}_1\subset F^{k+3}_{dpth}({\mathfrak{L}}_1)$. The direct sum over $k\geq 0$ of this sequence of maps is a 
sequence of graded module morphisms over the sequence of morphisms of symmetric algebras (graded by the symmetric
algebra degree) $S(\Sigma)\simeq S(\Sigma)\to S({\mathfrak{L}}_0)\to S(\mathrm{gr}_{dpth}^1({\mathfrak{L}}_0))$. One derives from there a 
commutative diagram 
\begin{equation}\label{CommDiag}
\xymatrix{
\Sigma\otimes\mathbb{M}_0^{min}(\Sigma) \ar[r]\ar[d]_{act}& \mathrm{gr}_{dpth}^1({\mathfrak{L}}_0)\otimes\mathrm{gr}_{dpth}^2({\mathfrak{L}}_1)\ar[d]\\
\mathbb{M}_1^{min}(\Sigma)\ar^{bot}[r]& \mathrm{gr}_{dpth}^3({\mathfrak{L}}_1)}
\end{equation}
It follows from the computation of the $S(\Sigma)$-module $\mathbb{M}^{min}(\Sigma)$ in Proposition \ref{prop:comp:module}
that {\it the map $act$ is surjective}.   

Let $\Sigma_0:=\C\sigma_3\oplus\C\sigma_5\subset\Sigma$. Proposition \ref{prop:comp:module} also implies that the restriction 
of $act$ to a map $\Sigma_0\otimes\mathbb{M}_0^{min}(\Sigma)\to\mathbb{M}_1^{min}(\Sigma)$ is a linear isomorphism. 
We have therefore a diagram 
\begin{equation}\label{diag:M:interm}
\xymatrix{
\Sigma_0\otimes\mathbb{M}_0^{min}(\Sigma)\ar_{\simeq}[d]\ar@{^{(}->}[r]& \mathrm{gr}_{dpth}^1({\mathfrak{L}}_0)\otimes\mathrm{gr}_{dpth}^2({\mathfrak{L}}_1)
\ar[d]\ar^{\!\!\!\!\!\!\!\!\!\!\!\!\simeq}[r]& 
A\C[A]B\otimes ABA'B'\C[A,A']^{as}\ar^{\mu}[d]\\
\mathbb{M}_1^{min}(\Sigma)\ar^{bot}[r]& \mathrm{gr}_{dpth}^3({\mathfrak{L}}_1)\ar^{\simeq}[r]& ABA'(B')^2\C[A,A']
}\end{equation}
Composing the resulting map $\Sigma_0\otimes\mathbb{M}_0^{min}(\Sigma)\to ABA'(B')^2\C[A,A']$ with the isomorphism of 
$\Sigma_0\otimes\mathbb{M}_0^{min}(\Sigma)$ with a subspace of $\Sigma_0\otimes ABA'B'\C[A,A']^{as}$ (see end of Subsection 
\ref{subsect:depth2}), we obtain the map 
\begin{align}\label{test:map}
& \Sigma_0\otimes ABA'B'(A-A')(A+A')(2A+A')(A+2A')\C[A^2+(A')^2+(A+A')^2,A^6+(A')^6+(A+A')^6]\nonumber
\\ & \to ABA'(B')^2\C[A,A'], 
\end{align}
$$
\sigma_3\otimes ABA'B'g(A,A')\mapsto 3ABA'(B')^2((A')^2-(A+A')^2)g(A,A'),$$
$$
\sigma_5\otimes ABA'B'g(A,A')\mapsto 5ABA'(B')^2((A')^4-(A+A')^4)g(A,A')
$$
As the ratio $((A')^4-(A+A')^4)/((A')^2-(A+A')^2)=(A')^2+(A+A')^2$ is not invariant under the permutation $A\leftrightarrow A'$, 
the map (\ref{test:map}) is injective. It follows that {\it the map $bot$ from diagram (\ref{CommDiag}) is injective.} Diagram 
(\ref{diag:M:interm}) also implies that {\it the image of $bot$ coincides with the images of (\ref{test:map}).}

\subsubsection{} There is a commutative diagram 
$$
\xymatrix{
F^1_{dpth}(\mathfrak{grt}_1)\otimes F^2_{dpth}(\mathfrak{grt}_1)\ar@{^{(}->}[r]\ar[d] & \overline{\mathfrak{L}}\otimes\overline{\mathfrak{L}}^1\ar[d]\ar[r]& {\mathfrak{L}}_0\otimes{\mathfrak{L}}_1\ar[d]\\
F^3_{dpth}(\mathfrak{grt}_1\ar@{^{(}->}[r])& \overline{\mathfrak{L}}^1\ar[r]&{\mathfrak{L}}^1 
}
$$
where the first square is induced by the Lie algebra inclusion $\mathfrak{grt}_1\subset\overline{\mathfrak{L}}$, and the second square is induced by 
taking associated graded for the l.c.s. filtration. The resulting diagram is compatible with the depth gradings and therefore gives rise to 
a diagram 
\begin{equation}\label{diag:grt:L}
\xymatrix{
\mathrm{gr}^1_{dpth}(\mathfrak{grt}_1)\otimes \mathrm{gr}^2_{dpth}(\mathfrak{grt}_1)\ar[r]\ar[d]& \mathrm{gr}^1_{dpth}({\mathfrak{L}}_0)\otimes \mathrm{gr}^2_{dpth}({\mathfrak{L}}_1)\ar[d]\\
\mathrm{gr}^3_{dpth}(\mathfrak{grt}_1)\ar[r]& \mathrm{gr}^3_{dpth}({\mathfrak{L}}_1)}
\end{equation}
The space $\mathbb{L}_3(\underline{V})$ decomposes as follows 
$$
\mathbb{L}_3(\underline{V})=[\xi[0],[\xi[0],\underline{V}_+]]\oplus[\underline{V}_+,[\xi[0],\underline{V}_+]]
\oplus\mathbb{L}_3(\underline{V}_+). 
$$
As $F_{dpth}^3(\mathfrak{grt}_1)\subset\overline{\mathfrak{L}}^1$, the subspace $\mathrm{gr}_{dpth}^3(\mathfrak{grt}_1)\subset 
\mathbb{L}_3(\underline{V})$ is such that 
\begin{equation}\label{incl:grt:V}
\mathrm{gr}_{dpth}^3(\mathfrak{grt}_1)\subset [\underline{V}_+,[\xi[0],\underline{V}_+]]\oplus\mathbb{L}_3(\underline{V}_+). 
\end{equation}
Recalling the isomorphism (\ref{iso:gr:L13}) of $\mathrm{gr}^3_{dpth}({\mathfrak{L}}_1)$ with the first summand, the bottom map of (\ref{diag:grt:L}) 
then can be identified with the composition of (\ref{incl:grt:V}) with the projection on the first summand of 
$[\underline{V}_+,[\xi[0],\underline{V}_+]]\oplus\mathbb{L}_3(\underline{V}_+)$. 

\subsubsection{}
Combining diagrams (\ref{diag:grt:L}) and (\ref{CommDiag}) and using isomorphism (\ref{iso:depth:2}), we obtain a commutative diagram 
$$
\xymatrix{
\Sigma\otimes\mathbb{M}_0^{min}(\Sigma)\ar@{^{(}->}[r]\ar@{->>}[d]
\ar@/^1.5pc/[rr]^{\simeq}&\mathrm{gr}_{dpth}^1({\mathfrak{L}}_0)\otimes\mathrm{gr}_{dpth}^2({\mathfrak{L}}_1)\ar[d] & 
\mathrm{gr}_{dpth}^1(\mathfrak{grt}_1)\otimes\mathrm{gr}_{dpth}^2(\mathfrak{grt}_1)\ar@{_{(}->}[l]\ar[d]^{\langle,\rangle}\\
\mathbb{M}_1^{min}(\Sigma)\ar@{^{(}->}[r]&\mathrm{gr}_{dpth}^3(\L_1) & \mathrm{gr}_{dpth}^3(\mathfrak{grt}_1)\ar[l]
}
$$
which, using isomorphism (\ref{iso:gr:L13}), gives rise to the diagram 
$$
\xymatrix{
\Sigma\otimes\mathbb{M}_0^{min}(\Sigma)\ar@{->>}[d]\ar^{\langle,\rangle\circ\simeq}[r]&\mathrm{gr}_{dpth}^3(\mathfrak{grt}_1)
\ar[d]\ar@{^{(}->}[r] &
[\underline{V}_+,[\xi[0],\underline{V}_+]]
\oplus\mathbb{L}_3(\underline{V}_+)\ar@{->>}[ld] \\
\mathbb{M}_1^{min}(\Sigma)\ar@{^{(}->}[r]& [\underline{V}_+,[\xi[0],\underline{V}_+]]& }
$$
From this diagram, one derives the equality 
$$
\mathrm{im}(\mathrm{gr}_{dpth}^3(\mathfrak{grt}_1\to[\underline{V}_+,[\xi[0],\underline{V}_+]]))
=\mathrm{im}(\mathbb{M}_1^{min}(\Sigma)\hookrightarrow[\underline{V}_+,[\xi[0],\underline{V}_+]])
=\text{image\ of\ map\ }(\ref{test:map}). 
$$
The vertical kernel in the square of this diagram is a linear map from the kernel of the left vertical map to the kernel of the right vertical
map. As the latter kernel coincides with $\mathrm{gr}_{dpth}^3(\mathfrak{grt}_1)\cap\mathbb{L}_3(\underline{V}_+)$, we obtain a linear map 
$$
\mathrm{ker}(\Sigma\otimes\mathbb{M}_0^{min}(\Sigma)\to\mathbb{M}_1^{min}(\Sigma))\to
\mathrm{gr}_{dpth}^3(\mathfrak{grt}_1)\cap\mathbb{L}_3(\underline{V}_+), 
$$
such that the following diagram commutes 
\begin{equation}\label{first:diag:concl}
\xymatrix{
\mathrm{ker}(\Sigma\otimes\mathbb{M}_0^{min}(\Sigma)\to\mathbb{M}_1^{min}(\Sigma))\ar[r]\ar@{^{(}->}[d]& 
\mathrm{gr}_{dpth}^3(\mathfrak{grt}_1)\cap\mathbb{L}_3(\underline{V}_+)\ar@{^{(}->}[d]\\
\Sigma\otimes\mathbb{M}_0^{min}(\Sigma)\ar[r]& \mathrm{gr}_{dpth}^3(\mathfrak{grt}_1)
}
\end{equation}
In Subsection \ref{subsect:depth2}, we constructed a commutative diagram 
$$
\xymatrix{
\Lambda^2(\Sigma)\ar^{\sim}[r]\ar_{\mathrm{id}\otimes\{,\}}@{->>}[d]& \Lambda^2(\mathrm{gr}_{dpth}^1(\mathfrak{grt}_1))\ar^{\mathrm{id}\otimes
\langle,\rangle}@{->>}[d] \\
\mathbb{M}_0^{min}(\Sigma)\ar^{\sim}[r]&\mathrm{gr}_{dpth}^2(\mathfrak{grt}_1)
}
$$
Taking the tensor product of this diagram with the isomorphism $\Sigma\stackrel{\sim}{\to}\mathrm{gr}_{dpth}^1(\mathfrak{grt}_1)$
and combining it with the Jacobi identity 
$$
(\Lambda^3(\mathrm{gr}_{dpth}^1(\mathfrak{grt}_1))\stackrel{\mathrm{id}\otimes \langle,\rangle}{\longrightarrow}\mathrm{gr}_{dpth}^1(\mathfrak{grt}_1)\otimes\mathrm{gr}_{dpth}^2(\mathfrak{grt}_1)\stackrel{\langle,\rangle}{\longrightarrow}
\mathrm{gr}_{dpth}^3(\mathfrak{grt}_1))=0,
$$
we obtain the diagram 
$$
\xymatrix{
\Lambda^3(\Sigma)\ar^{\!\!\!\!\!\!\!\!\!\sim}[r]\ar_{\mathrm{id}\otimes\{,\}}[d]&\Lambda^3(\mathrm{gr}_{dpth}^1(\mathfrak{grt}_1))\ar_{\mathrm{id}
\otimes \langle,\rangle}[d]\ar^{0}[dr] & \\
\Sigma\otimes\mathbb{M}_0^{min}(\Sigma)\ar^{\!\!\!\!\!\!\!\!\!\!\!\!\!\!\!\!\!\!\!\!\sim}[r]& 
\mathrm{gr}_{dpth}^1(\mathfrak{grt}_1)\otimes
\mathrm{gr}_{dpth}^2(\mathfrak{grt}_1)\ar^{\ \ \ \ \ \ \ \ \ \langle,\rangle}[r]& \mathrm{gr}_{dpth}^3(\mathfrak{grt}_1)
}
$$
which implies the identity 
$$
(\Lambda^3(\Sigma)\stackrel{\mathrm{id}\otimes\{,\}}{\longrightarrow}\Sigma\otimes\mathbb{M}_0^{min}(\Sigma)
\to\mathrm{gr}_{dpth}^3(\mathfrak{grt}_1))=0. 
$$
The $\Sigma$-structure property of $\mathbb{M}^{min}(\Sigma)$ implies that the sequence of maps 
\begin{equation}\label{compl:M}
\Lambda^3(\Sigma)\stackrel{\mathrm{id}\otimes\{,\}}{\to}\Sigma\otimes\mathbb{M}_0^{min}(\Sigma)\to\mathbb{M}_1^{min}(\Sigma)
\end{equation}
is a complex, therefore that the image of $\Lambda^3(\Sigma)\stackrel{\mathrm{id}\otimes\{,\}}{\to}
\Sigma\otimes\mathbb{M}_0^{min}(\Sigma)$ is contained in $\mathrm{ker}(\Sigma\otimes\mathbb{M}_0^{min}(\Sigma)
\to\mathbb{M}_1^{min}(\Sigma))$. It follows that one can augment (\ref{first:diag:concl}) as follows 
$$
\xymatrix{
\Lambda^3(\Sigma)\ar^0[dr]\ar[d] & \\
\mathrm{ker}(\Sigma\otimes\mathbb{M}_0^{min}(\Sigma)\to\mathbb{M}_1^{min}(\Sigma))\ar[r]\ar@{^{(}->}[d]& 
\mathrm{gr}_{dpth}^3(\mathfrak{grt}_1)\cap\mathbb{L}_3(\underline{V}_+)\ar@{^{(}->}[d]\\
\Sigma\otimes\mathbb{M}_0^{min}(\Sigma)\ar[r]& \mathrm{gr}_{dpth}^3(\mathfrak{grt}_1)}
$$
Denoting the middle homology of the complex (\ref{compl:M}) by $H(\mathbb{M}^{min}(\Sigma))$, one then obtains a map  
\begin{equation}\label{map:coh}
H(\mathbb{M}^{min}(\Sigma))\to\mathrm{gr}_{dpth}^3(\mathfrak{grt}_1)\cap\mathbb{L}_3(\underline{V}_+). 
\end{equation}

\subsubsection{} \label{subsect:1012:surj} The following diagram commutes
$$
\xymatrix{
\Sigma^{\otimes3}\ar^{\!\!\!\!\!\!\!\!\!\!\!\!\!\!\!\!\!\!\!\!\sim}[r]
\ar_{\mathrm{id}\otimes\{,\}}[d]&\mathrm{gr}_{dpth}^1(\mathfrak{grt}_1)^{\otimes3}
\ar@{->>}^{\langle,\rangle\circ(\mathrm{id}\otimes \langle,\rangle)}[d]\\
\Sigma\otimes\mathbb{M}_0^{min}(\Sigma)\ar[r]&\mathrm{gr}_{dpth}^3(\mathfrak{grt}_1)}
$$
In this diagram, the surjectivity of the left vertical map follows from the construction of $\mathbb{M}^{min}(\Sigma)$, and the surjectivity 
of the right vertical map follows from Subsection \ref{subsect:material}. It follows that the map $\Sigma\otimes\mathbb{M}_0^{min}(\Sigma)
\to\mathrm{gr}_{dpth}^3(\mathfrak{grt}_1)$ is surjective, as indicated in the following commutative diagram 
$$
\xymatrix{
\mathrm{ker}(\Sigma\otimes\mathbb{M}_0^{min}(\Sigma)\to\mathbb{M}_1^{min}(\Sigma))\ar[r]\ar@{^{(}->}[d]& \mathrm{gr}_{dpth}^3(\mathfrak{grt}_1)\cap\mathbb{L}_3(\underline{V}_+)\ar@{^{(}->}[d]\\
\Sigma\otimes\mathbb{M}_0^{min}(\Sigma)\ar@{->>}[r]\ar@{->>}[d]& \mathrm{gr}_{dpth}^3(\mathfrak{grt}_1)\ar[d]\\
\mathbb{M}_1^{min}(\Sigma)\ar@{^{(}->}[r]&[\underline{V}_+,[\xi[0],\underline{V}_+]]}
$$
with exact columns. Inspection of this diagram then implies that {\it the top horizontal map of this diagram is surjective,} therefore 
that {\it the map (\ref{map:coh}) is surjective.}

\subsubsection{} There is an exact sequence 
$$
0\to\mathrm{gr}_{dpth}^3(\mathfrak{grt}_1)\cap\mathbb{L}_3(\underline{V}_+)\to\mathrm{gr}_{dpth}^3(\mathfrak{grt}_1)\to
\mathrm{im}(\mathrm{gr}_{dpth}^3(\mathfrak{grt}_1)\to[\underline{V}_+,[\xi[0],\underline{V}_+]])\to0.
$$
The last space is isomorphic to the image of the injection (\ref{test:map}), therefore its Hilbert series is 
${t^{11}(1+t^2)\over(1-t^2)(1-t^6)}$; according to Subsection \ref{subsect:material}, the Hilbert series of $\mathrm{gr}_{dpth}^3(\mathfrak{grt}_1)$ is ${t^{11}(1+t^2-t^4)\over(1-t^2)(1-t^4)(1-t^6)}$. It follows that the Hilbert series 
of $\mathrm{gr}_{dpth}^3(\mathfrak{grt}_1)\cap\mathbb{L}_3(\underline{V}_+)$ is ${t^{17}\over(1-t^2)(1-t^4)(1-t^6)}$. 

On the other hand, the maps of the complex (\ref{compl:M}) have the following properties: the last map 
$\Sigma\otimes\mathbb{M}_0^{min}(\Sigma)\to\mathbb{M}_1^{min}(\Sigma)$ is surjective, and it follows from \cite{G}, Theorem 1.5 
that the first map $\Lambda^3(\Sigma)\stackrel{\mathrm{id}\otimes \langle,\rangle}{\longrightarrow}\Sigma\otimes\mathbb{M}_0^{min}(\Sigma)$
is injective. Then the Hilbert series of $H(\mathbb{M}^{min}(\Sigma))$ is computed as follows 
\begin{align*}
& P_{H(\mathbb{M}^{min}(\Sigma))}(t)=P_{\Sigma\otimes\mathbb{M}_0^{min}(\Sigma)}(t)-P_{\Lambda^3(\Sigma)}(t)
-P_{\mathbb{M}_1^{min}(\Sigma)}(t)
\\ & ={t^3\over 1-t^2}\cdot{t^8\over(1-t^2)(1-t^6)}-{t^{15}\over(1-t^2)(1-t^4)(1-t^6)}-(t^3+t^5)\cdot{t^8\over(1-t^2)(1-t^6)}
\\ & ={t^{17}\over(1-t^2)(1-t^4)(1-t^6)}, 
\end{align*}
so $P_{H(\mathbb{M}^{min}(\Sigma))}(t)=P_{\mathrm{gr}_{dpth}^3(\mathfrak{grt}_1)\cap\mathbb{L}_3(\underline{V}_+)}(t)$. 
Combined with the fact that the map (\ref{map:coh}) is a graded epimorphism (Subsubsection \ref{subsect:1012:surj}), this equality implies 
that {\it the map (\ref{test:map}) is a graded isomorphism.} 

\subsubsection{} The results of Subsection \ref{subsect:depth3} can be summarized as follows. 

\begin{theorem}
The space $\mathbb{L}_3(\underline{V})$ admits the decomposition 
$$
\mathbb{L}_3(\underline{V})=[\xi[0],[\xi[0],\underline{V}_+]]\oplus[\underline{V}_+,[\xi[0],\underline{V}_+]]\oplus
\mathbb{L}_3(\underline{V}_+). 
$$ 
Its subspace $\mathrm{gr}_{dpth}^3(\mathfrak{grt}_1)$ satisfies
$$
\mathrm{gr}_{dpth}^3(\mathfrak{grt}_1)\subset[\underline{V}_+,[\xi[0],\underline{V}_+]]\oplus\mathbb{L}_3(\underline{V}_+).
$$
The exact sequence $0\to\mathbb{L}_3(\underline{V}_+)\to[\underline{V}_+,[\xi[0],
\underline{V}_+]]\oplus\mathbb{L}_3(\underline{V}_+)\to[\underline{V}_+,[\xi[0],\underline{V}_+]]\to0$ gives rise to an exact sequence
$$
0\to\mathrm{gr}_{dpth}^3(\mathfrak{grt}_1)\cap\mathbb{L}_3(\underline{V}_+)\to\mathrm{gr}_{dpth}^3(\mathfrak{grt}_1)\to
\mathrm{im}(\mathrm{gr}_{dpth}^3(\mathfrak{grt}_1)\to[\underline{V}_+,[\xi[0],\underline{V}_+]])\to0. 
$$
The isomorphism $[\underline{V}_+,[\xi[0],\underline{V}_+]]\simeq ABA'(B')^2\C[A,A']$ given by (\ref{iso:gr:L13}) gives rise to an isomorphism 
$$
\mathrm{im}(\mathrm{gr}_{dpth}^3(\mathfrak{grt}_1)\to[\underline{V}_+,[\xi[0],\underline{V}_+]])\simeq\mathrm{im}(\text{map\ }
(\ref{test:map})).
$$ The map (\ref{map:coh}) is an isomorphism 
$$
H(\mathbb{M}^{min}
(\Sigma))\stackrel{\sim}{\to}\mathrm{gr}_{dpth}^3(\mathfrak{grt}_1)\cap\mathbb{L}_3(\underline{V}_+).
$$ 
The Hilbert series of these spaces are given by 
$$P_{\mathrm{gr}_{dpth}^3(\mathfrak{grt}_1)\cap\mathbb{L}_3(\underline{V}_+)}(t)
={t^{17}\over(1-t^2)(1-t^4)(1-t^6)}, \  
P_{\mathrm{im}(\mathrm{gr}_{dpth}^3(\mathfrak{grt}_1)\to[\underline{V}_+,[\xi[0],\underline{V}_+]])}(t)={t^8(t^3+t^5)\over(1-t^2)(1-t^6)},
$$
$P_{\mathrm{gr}_{dpth}^3(\mathfrak{grt}_1)}={t^{11}(1+t^2-t^4)\over(1-t^2)(1-t^4)(1-t^6)}$.
\end{theorem}
\begin{remark} Explicit computation of $H(\mathbb{M}^{min}(\Sigma))$ and of the map (\ref{map:coh}) would therefore yield an 
explicit description of $\mathrm{gr}_{dpth}^3(\mathfrak{grt}_1)\cap\mathbb{L}_3(\underline{V}_+)$.  
\end{remark}

\subsection*{Acknowledgements} We thank L. Schneps for discussions related to this work, in particular on the subject of relating the lower
bounds found there with the depth filtration computations for \cite{G}.


\begin{thebibliography}{Mgm}

\bibitem[A]{A}
Y. Andr\'{e}, {\em Introduction aux motifs (motifs purs, motifs mixtes, p\'{e}riodes)}, Panoramas et Synth\`{e}ses, no. 17 (2004), Soc. Math. France.

\bibitem[BK]{BK}
 D.J. Broadhurst and D. Kreimer, {\em Association of multiple zeta values with positive knots via Feynman diagrams up to 9 loops}, Phys. Lett. B 393 (1997), 403-412.

\bibitem[Br1]{Br1}
F.C.S. Brown, {\em Mixed Tate motives over $\mathbb{Z}$}, Ann. of Math. (2) 175 (2012), no. 2, 949–976. 

\bibitem[Br2]{Br2}
F.C.S. Brown, {\em Depth-graded motivic multiple zeta values,} preprint arXiv:1301.3053. 

\bibitem[Dr]{Dr}
 V. G. Drinfeld, {\em On quasitriangular quasi-Hopf algebras and on a group that is closely connected with $\mathrm{Gal}(\overline{\Q}/\Q)$}, Algebra i Analiz 2 (1990), 149-181 (in Russian). (English translation in Leningrad Math. J. 2(4) (1991), 829-860).
 
\bibitem[Ec]{Ec}
J. Ecalle, {\em The flexion structure and dimorphy: flexion units, singulators, generators, and the enumeration of multizeta irreducibles}. Costion, O. (ed.) et al., Asymptotics in dynamics, geometry and PDEs. Generalized Borel summation. Vol. II. Proceedings of the conference, CRM, Pisa, Italy, October 12-16, 2009. Piza: Edizioni della Normale. Centro di Ricerca Matematica Ennio De Giorgi (CRM) Series (Nuova Serie) 12, 2, 27-211 (2011).

\bibitem[En1]{En1}
B. Enriquez, {\em On the Drinfeld generators of $\mathfrak{grt}_1(k)$ and $\Gamma$-functions for associators}, Math. Res. Lett. 13 (2006), 
no. 2-3, 231-243.

\bibitem[En2]{En2}
B. Enriquez, {\em Analogues elliptiques des nombres multiz\'{e}tas}, preprint arXiv:1301.3042. 

\bibitem[G]{G} A. Goncharov, {\em Multiple polylogarithms, cyclotomy and modular complexes,} Math. Res. Lett. 5 (1998), 497-516.

\bibitem[Ih1]{Ih1}
 Y. Ihara, {\em On the stable derivation algebra associated with some braid groups}, Israel J. Math. 80 (1992), no. 1-2, 135-153.

\bibitem[Ih2]{Ih2}
Y. Ihara, {\em Some arithmetic aspects of Galois actions on the pro-$p$ fundamental group of $\P^1\ \{0,1\infty\}$}, in Arithmetic Fundamental Groups and Noncommutative Algebra, Proc. Symos. Pure Math. 70, Berkeley, CA, 1999, 247-273.


\bibitem[Mpl]{Mpl} M.~Monagan, K.~Geddes, K. Heal, G. Labahn, S.~Vorkoetter, J. McCarron, P. DeMarco, 
{\em Maple~10 Programming Guide}, 2005, Maplesoft, Waterloo ON, Canada. 

\bibitem[Mgm]{Mgm} W. Bosma, J. Cannon, C. Playoust, {\em The Magma algebra system. I. The user language,} 
J. Symbolic Comput., 24 (1997), 235–265. 

\bibitem[M]{M}
T. Molien, {\em Ueber die Invarianten der linearen Substitutionsgruppen}, Sitzungsber. Koenig. Preuss. Akad. Wiss. (1897), no. 52, 1156-1156.

\bibitem[Sc]{Sc}
L. Schneps, {\em On the Poisson bracket on the free Lie algebra in two generators}, J. Lie Theory 16 (2006), no. 1, 19-37.

\bibitem[St]{St}
P. Stanley, {\em Invariants of finite groups and their applications to combinatorics}, Bull. of the AMS (New Series), vol. 1, no.3 (1979), 475-511.

\bibitem[Za]{Za}
D. Zagier, {\em Periods of modular forms and Jacobi theta functions}, Invent. Math. 104 (1991), 449-465.

\end{thebibliography}
\end{document}